\documentclass[11pt,reqno]{amsart}

\usepackage[T1]{fontenc}
\usepackage[utf8]{inputenc}
\usepackage{lmodern}
\usepackage{microtype}
\usepackage{amsmath,amssymb,amsthm,amscd,mathtools,mathrsfs}
\usepackage{enumitem}
\usepackage{xcolor}
\usepackage{hyperref}
\usepackage{aliascnt}
\usepackage[margin=1.15in]{geometry}

\hypersetup{
  colorlinks=true,
  linkcolor=blue!55!black,
  citecolor=green!45!black,
  urlcolor=blue!65!black,
  pdftitle={p-adic Rigidity and Supercongruences at Rank Two Attractors},
  pdfauthor={Yu Fu}
}

\newtheorem{theorem}{Theorem}[section]
\newaliascnt{proposition}{theorem}
\newtheorem{proposition}[proposition]{Proposition}
\aliascntresetthe{proposition}
\newaliascnt{lemma}{theorem}
\newtheorem{lemma}[lemma]{Lemma}
\aliascntresetthe{lemma}
\newaliascnt{corollary}{theorem}
\newtheorem{corollary}[corollary]{Corollary}
\aliascntresetthe{corollary}
\newaliascnt{conjecture}{theorem}
\newtheorem{conjecture}[conjecture]{Conjecture}
\aliascntresetthe{conjecture}
\theoremstyle{definition}
\newaliascnt{definition}{theorem}

\aliascntresetthe{definition}
\newaliascnt{assumption}{theorem}

\aliascntresetthe{assumption}
\theoremstyle{remark}
\newaliascnt{remark}{theorem}
\newtheorem{remark}[remark]{Remark}
\aliascntresetthe{remark}
\newaliascnt{question}{theorem}

\aliascntresetthe{question}

\usepackage[nameinlink,capitalise,noabbrev]{cleveref}
\crefname{theorem}{Theorem}{Theorems}
\crefname{proposition}{Proposition}{Propositions}
\crefname{lemma}{Lemma}{Lemmas}
\crefname{corollary}{Corollary}{Corollaries}
\crefname{conjecture}{Conjecture}{Conjectures}
\crefname{definition}{Definition}{Definitions}
\crefname{assumption}{Assumption}{Assumptions}
\crefname{remark}{Remark}{Remarks}
\crefname{question}{Question}{Questions}
\Crefname{theorem}{Theorem}{Theorems}
\Crefname{proposition}{Proposition}{Propositions}
\Crefname{lemma}{Lemma}{Lemmas}
\Crefname{corollary}{Corollary}{Corollaries}
\Crefname{conjecture}{Conjecture}{Conjectures}
\Crefname{definition}{Definition}{Definitions}
\Crefname{assumption}{Assumption}{Assumptions}
\Crefname{remark}{Remark}{Remarks}
\Crefname{question}{Question}{Questions}

\newcommand{\Q}{\mathbb Q}
\newcommand{\Z}{\mathbb Z}
\newcommand{\F}{\mathbb F}
\newcommand{\C}{\mathbb C}
\newcommand{\Gm}{\mathbf G_m}

\newcommand{\Fp}{\mathbb F_p}
\newcommand{\PP}{\mathbb P}
\newcommand{\et}{\mathrm{\acute et}}
\newcommand{\cris}{\mathrm{cris}}
\newcommand{\dR}{\mathrm{dR}}
\newcommand{\rig}{\mathrm{rig}}
\newcommand{\Fil}{\operatorname{Fil}}
\newcommand{\Sl}{\operatorname{Slp}}
\newcommand{\Res}{\operatorname{Res}}
\newcommand{\ord}{\operatorname{ord}}
\newcommand{\CT}{\operatorname{CT}}
\newcommand{\Frob}{\operatorname{Frob}}
\newcommand{\cC}{\mathscr C}
\newcommand{\cF}{\mathscr F}

\title{\(p\)-adic Rigidity and Supercongruences at Rank Two Attractors}
\author{Yu Fu}

\begin{document}

\begin{abstract}
We study a \(p\)-adic rigidity phenomenon suggested by rank two attractors. Near an ordinary special fiber of a Calabi--Yau family, the excellent Frobenius of Beukers and Vlasenko acts on the parameter disk. Achinger--Zdanowicz construct a canonical $W_2$-lift of the Frobenius twist, while Brantner--Taelman construct a canonical formal lift. We conjecture that the Dwork and Achinger--Zdanowicz first-order obstruction maps agree and that the fixed points of the excellent Frobenius are precisely the parameters of the Brantner--Taelman canonical formal lifts.

Our main theorem gives a cohomological criterion in arbitrary dimension. Under its hypotheses, a Frobenius-stable rank two factor of the $p$-adic cohomology containing the holomorphic class forces that class to span a Cartier-stable line and gives supercongruences modulo $p^{2s}$ for every $s\geq1$. Moreover, if the hypotheses of the Beukers--Vlasenko excellent lift theorem hold, then the excellent Frobenius fixes the corresponding parameter.

We apply the criterion to the Hulek--Verrill $A_4$ family at $t_*=-1/7$, a parameter predicted to be a rank two attractor. Assuming Dummigan's Conjecture~1.1, we prove that the excellent Frobenius fixes $t_*$, the supercongruences hold, and the fiber with parameter $t_*\bmod p^2$ is the Achinger--Zdanowicz canonical lift at every prime where the required Dwork ordinarity conditions are satisfied. A CM fiber of the mirror quartic K3 pencil gives a second application. We also prove that the $q$-power map on a split formal torus has the unit section as its unique periodic point in the identity residue disk over every finite extension of $W(\F_q)[1/p]$, where $q$ is the cardinality of the residue field.
\end{abstract}

\maketitle
\setcounter{tocdepth}{1}
\tableofcontents

\section{Introduction}

Throughout the paper, $p$ denotes an odd prime.

\subsection{Rank two attractors and an arithmetic question}

Let $X$ be a Calabi--Yau threefold over $\C$. Its holomorphic line $H^{3,0}(X)$ is generated by a holomorphic volume form. A point in a family of Calabi--Yau threefolds is a rank two attractor if its fiber $X$ admits a rank two rational Hodge substructure $E\subset H^3(X,\Q)$ such that $E_{\C}=H^{3,0}(X)\oplus H^{0,3}(X)$ \cite[\S3.6.3, equations~(3.73)--(3.74)]{Moore2007}. Polarization gives a rational Hodge complement to $E$.

The terminology comes from supersymmetric black hole attractors in type IIB string theory. A charge is given by an integral cohomology class $\gamma\in H^3(X,\Z)$, and the attractor equations require its image in $H^3(X,\C)$ to lie in $H^{3,0}(X)\oplus H^{0,3}(X)$ \cite{FerraraKallosh1996}. At a rank two attractor, this subspace contains two linearly independent integral classes \cite[\S1.2]{CandelasDeLaOssaElmiVanStraten2020}.

We ask how this rank two splitting is reflected in $p$-adic cohomology. Let $F$ be a number field, let $\mathscr X\to S$ be a family of Calabi--Yau threefolds over $F$, and fix $t\in S(F)$. At a finite place of good reduction. Let $k$ be the residue field. Let $\mathscr X\to S$ be an integral model and let $\bar t\in S(k)$ be the reduction of $t$. Assume that $S$ is smooth of relative dimension one at $\bar t$ and that $\mathscr X\to S$ is smooth along the fiber over \(\bar t\). Let $X_{\bar t}=\mathscr X\times_{S,\bar t}\operatorname{Spec}k$. Let $\widehat S_{\bar t}$ be the formal completion at $\bar t$ that parametrizes the lifts of $\bar t$ in $S$. For any lift $t$ of $\bar t$, the fiber $\mathscr X_t$ is a lift of $X_{\bar t}$ within the given family.

A \emph{coefficient Frobenius} is an endomorphism of the completed coordinate ring of the base that lifts the absolute Frobenius in characteristic $p$. For a Laurent model of the family, Beukers and Vlasenko construct an excellent coefficient Frobenius under their ordinarity hypotheses \cite[Theorem~7.3]{BeukersVlasenkoIII}. It induces a map $\sigma_{\mathrm{exc}}:\widehat S_{\bar t}\to \widehat S_{\bar t^p}$. 

On the geometric side, Achinger and Zdanowicz construct a canonical $W_2(k)$-lift of the Frobenius twist of $X_{\bar t}$ \cite[Corollary~4.1.3]{AchingerZdanowicz2021}. For a Bloch--Kato $2$-ordinary fiber satisfying a torsion free condition, Brantner and Taelman construct a canonical formal lift of $X_{\bar t}$ over $W(k)$ \cite[Theorem~B and Definition~8.31]{BrantnerTaelman2024}. For the comparison with the Achinger--Zdanowicz lift, we assume that $k=\F_p$. Then the Frobenius twist is canonically identified with $X_{\bar t}$.

The excellent Frobenius acts on the residue disk in the base, whereas the constructions of Achinger--Zdanowicz and Brantner--Taelman give deformations of $X_{\bar t}$. Suppose that the classifying morphism from $\widehat S_{\bar t}$ to the relevant deformation space identifies the residue disk with its image. If the Achinger--Zdanowicz lift lies in this image, then we write $t_{\mathrm{can}}^{(2)}\in\widehat S_{\bar t}(W_2(k))$ as its unique parameter. If the Brantner--Taelman canonical formal lift lies in this image, write $t_{\mathrm{can}}^{(\infty)}\in\widehat S_{\bar t}(W(k))$ as its unique parameter. We ask whether the excellent Frobenius fixes $t_{\mathrm{can}}^{(2)}$ modulo $p^2$ and fixes $t_{\mathrm{can}}^{(\infty)}$ at every Witt order. We call these fixed point properties \emph{$p$-adic rigidity}.

To study this comparison, we relate Dwork theory to crystalline cohomology. Under geometric hypotheses, we construct a comparison map that relates the Cartier operator to crystalline Frobenius; see \cref{thm:general-bounded-residue}. We use this map to connect rank two subisocrystals with fixed points of the excellent Frobenius and supercongruences.

The motivating examples are rank two attractor points in families of Calabi--Yau threefolds, while our comparison conjectures and \cref{thm:rank two-criterion} apply to Calabi--Yau $d$-folds.

\subsection{Main results}

We use the notation of \cref{sec:vertex-determinants}. Let $f_t=1-tg$ be a one parameter Laurent family, let $t_0\in\Z_p^\times$, and let $M_{t_0}$ be the specialization at $t_0$ of the Calabi--Yau crystal associated with $g$. The auxiliary coefficient Frobenius $\sigma_{t_0}(t)=t_0^{\,1-p}t^p$ on the completed residue disk fixes $t_0$. Let $C$ be the resulting Cartier operator on $M_{t_0}$.

If a coefficient Frobenius sends $t$ to $u$, then under the first two Hasse--Witt conditions (also introduced in \cref{sec:vertex-determinants}), the holomorphic class $1/f_u$ and its logarithmic derivative $\theta(1/f_u)$ form a basis of the level two quotient, where $\theta=t\,d/dt$ \cite[Corollary~5.9 and the discussion following equation~(20)] {BeukersVlasenkoIII}. Let $\lambda_1(t,u)$ denote the coefficient of $\theta(1/f_u)$ in the Cartier image of $1/f_t$, as in \eqref{eq:lambda-row}. Let $d\geq2$, and let $H$ be a weakly admissible filtered isocrystal over $\Q_p$ with Frobenius $\varphi$.

Our general fixed point criterion is as follows.

\begin{theorem}
\label{thm:rank two-criterion}
Assume the first two Hasse--Witt conditions and suppose that the following conditions hold.
\begin{enumerate}[label=\textup{(\roman*)}]
\item There is a bounded surjection $\rho:M_{t_0}[1/p]\twoheadrightarrow H$ satisfying $\rho C=p^d\varphi^{-1}\rho$.
\item The high slope quotient satisfies $\dim H/H_{\varphi,\leq d-2}=2$.
\item There is a rank two subisocrystal $E\subset H$ whose Newton slopes are $0$ and $d$, and $\rho(1/f_{t_0})\in E$.
\end{enumerate}
Then $\rho$ induces an isomorphism $(M_{t_0}/\cF_2M_{t_0})[1/p]\simeq H/H_{\varphi,\leq d-2}$, and the line spanned by the class of $1/f_{t_0}$ is Cartier-stable. In particular, $\lambda_1(t_0,t_0)=0$.

For every vertex $\mathbf b$ with unit coefficient, every $\mathbf m,\mathbf n\in C_{\mathbf b}\cap\Z^n$, and every $s\geq1$, one has
\[D_s^{\mathbf m,\mathbf n}(t_0,t_0)\equiv0\pmod{p^{2s}}.\]
If the excellent Frobenius theorem of Beukers and Vlasenko applies on the residue disk, then $\sigma_{\mathrm{exc}}(t_0)=t_0$.
\end{theorem}

The quantity $D_s^{\mathbf m,\mathbf n}(t,u)$ is the $2$-by-$2$ determinant of two vertex-coefficient vectors at consecutive Frobenius levels. Its vanishing modulo $p^{2s}$ says that the exterior product of these vectors vanishes modulo $p^{2s}$. Cartier stability implies these congruences among the vertex coefficients. 

\medskip We apply \cref{thm:rank two-criterion} to two families.  Our first application is the Hulek--Verrill $A_4$ family with Laurent model
\begin{equation}\label{eq:A4-g}
 g(\mathbf x)=
 (1+x_1+x_2+x_3+x_4)
 \left(1+x_1^{-1}+x_2^{-1}+x_3^{-1}+x_4^{-1}\right),
 \quad f_t=1-tg.
\end{equation}
Throughout this application, we assume Dummigan's Conjecture~1.1; see \cref{sec:family}. 

Let $X_t$ denote the Calabi--Yau threefold obtained from this Laurent model by the Hulek--Verrill quotient construction \cite[\S2]{CandelasDeLaOssaElmiVanStraten2020}. We consider the parameter $t_*=-1/7$. Numerical period relations found by Candelas, de la Ossa, Elmi, and van Straten suggest that $t_*$ is a rank two attractor \cite[\S4.2]{CandelasDeLaOssaElmiVanStraten2020}. Yang studies the Hodge-theoretic and Deligne-period consequences of this predicted rational splitting \cite[\S\S4--5]{Yang2021}, but the rational Betti splitting at $t_*$ is still conjectural.

The arithmetic input for our application is Dummigan's conjectural modular semisimplification of the associated $\ell$-adic representation \cite[Conjecture~1.1]{Dummigan2026}. We note that his exact local Frobenius factorizations and residual calculations provide strong evidence for this conjecture \cite[Proposition~4.6 and Propositions~5.2--5.3]{Dummigan2026}. In \cref{prop:actual-split}, polarization upgrades the conjectural semisimplification to an actual direct sum. Crystalline comparison then gives the filtered $p$-adic factor required by \cref{thm:rank two-criterion} \cite{Tsuji1999}. 

For $p>5$, $p\neq7$, we say that $p$ is \emph{two-Dwork-ordinary at $t_*$} if the first Hasse--Witt scalar and the normalized second Hasse--Witt determinant are units at $t_*$. The determinant $D_s(t,u)$ for this family is defined in \eqref{eq:Ds-intro}. Let $H_p=D_{\cris}(H^3_{\et}(X_{t_*,\overline\Q_p},\Q_p))$, and let $(H_p)_{\leq1}$ be its subisocrystal with slopes at most $1$.

Applying the general criterion to this splitting gives the following theorem.

\begin{theorem}
\label{thm:main}
Let $p>5$, $p\neq7$, and assume that $p$ is two-Dwork-ordinary at $t_*$. Assume Dummigan's Conjecture~1.1. Then
\[\left(CY(g)_{t_*}/\cF_2CY(g)_{t_*}\right)[1/p] \simeq H_p/(H_p)_{\leq1}.\]
The crystalline slopes of $H_p$ are $\{0,1,2,3\}$. The factor attached to the weight four modular form has slopes $\{0,3\}$ and contains the top Hodge line, while the factor attached to the Tate twist of the weight two modular form has slopes $\{1,2\}$. Moreover, $\lambda_1(t_*,t_*)=0$ and $\sigma_{\mathrm{exc},p}(t_*)=t_*$ in $\Z_p$. For every $s\geq1$,
\begin{equation}\label{eq:main-supercong}
 D_s(t_*,t_*)\equiv0\pmod{p^{2s}}.
\end{equation}
\end{theorem}

Under Dummigan's Conjecture~1.1, the first Hasse--Witt condition alone is sufficient for the comparison modulo $p^2$. Let $k=\F_p$ and $A=W_2(k)$. Let $X_0$ be the reduction of the $A_4$ fiber at $t_*$, and let $X_*^A$ be the lift of $X_0$ in the family at the image of $t_*$ in $A$.

\begin{theorem}
\label{thm:A4-W2-canonical}
Let $p>5$, $p\neq7$, and assume Dummigan's Conjecture~1.1 and the first Dwork Hasse--Witt condition at $t_*$. Then $X_*^A$ is the Achinger--Zdanowicz canonical $W_2(k)$-lift of $X_0$.
\end{theorem}

At two-Dwork-ordinary primes, we compare the obstruction maps on the entire $W_2$ residue disk. Let $a=\overline{t_*}$, let $h_p=\CT(f_a^{p-1})\in k$, where $\CT$ denotes the constant term, and let $t_\delta=t_*+p\delta\in A$ for $\delta\in k$. The general obstruction maps are defined in \cref{sec:canonical-lifts}; their integral identification and normalization for this family are defined in \cref{sec:A4-W2}. We have the following theorem.

\begin{theorem}
\label{thm:A4-W2-comparison}
Let $p>5$, $p\neq7$, and assume that $p$ is two-Dwork-ordinary at $t_*$. Assume Dummigan's Conjecture~1.1. For every $\delta\in k$, one has $\kappa_a(\operatorname{ob}_D(t_\delta))= \operatorname{ob}_C(t_\delta)$, and in the normalization of \cref{sec:A4-W2} the scalar of either side is
\[\frac{\lambda_1(t_\delta,t_\delta)}p =-\frac{h_p}{a}\,\delta \quad\text{in }k.\]
Their common zero is $\delta=0$.
\end{theorem}

The second application is the mirror quartic K3 pencil $f_t=1-t\bigl(x+y+z+(xyz)^{-1}\bigr)$; see \cref{sec:K3}. Let $a_Q^{K3}(t)$ be the vertex coefficients defined in \eqref{eq:K3-vertex}. We prove the following theorem.

\begin{theorem}
Let $\tau\in\Z_{37}$ be the root of $\tau^4=-1/12288$ with $\tau\equiv15\pmod{37}$. Then $\sigma_{\mathrm{exc},37}(\tau)=\tau$. For every $s\geq1$, one has
\[a_{37^s}^{K3}(\tau)a_{2\cdot37^{s-1}}^{K3}(\tau) -a_{2\cdot37^s}^{K3}(\tau)a_{37^{s-1}}^{K3}(\tau) \equiv0\pmod{37^{2s}}.\]
\end{theorem}

\subsection{Two comparison conjectures} We give two conjectures comparing excellent Frobenius with canonical lifting.
Let $k=\F_p$. Assume that the Laurent parameter is nonzero at $\bar t$, that $X_{\bar t}$ is $1$-ordinary, and that the level two Dwork quotient and the excellent Frobenius are defined on the residue disk. Let $\iota:\widehat S_{\bar t}\to\operatorname{Def}(X_{\bar t})$ be the classifying morphism of the given family. Assume that $\iota$ identifies $\widehat S_{\bar t}$ with a one-dimensional polarized or equivariant formal subscheme \(\mathcal D_{\bar t}\subset\operatorname{Def}(X_{\bar t}).\) Assume also that the point defined by the Achinger--Zdanowicz canonical lift belongs to $\mathcal D_{\bar t}(W_2(k))$. The two obstruction maps and the integral identification $\kappa_{\bar t}$ are defined in \cref{sec:canonical-lifts}.

For $t\in\widehat S_{\bar t}(W_2(k))$, the Dwork obstruction $\operatorname{ob}_D(t)$ measures whether $t$ is fixed by the excellent Frobenius, while the crystalline obstruction $\operatorname{ob}_C(t)$ measures whether the corresponding lift is the Achinger--Zdanowicz canonical lift. The two obstruction maps take values in different spaces. We use the residue, Kodaira--Spencer, and polarization maps to define an isomorphism $\kappa_{\bar t}$ between these target spaces.

\begin{conjecture}\label{conj:W2-comparison}
For every $t\in\widehat S_{\bar t}(W_2(k))$, one has $\kappa_{\bar t}(\operatorname{ob}_D(t)) =\operatorname{ob}_C(t)$.
\end{conjecture}

The conjecture identifies the two normalized obstruction maps and their zero sets. The uniqueness assertions in \cref{lem:pointwise,lem:canonical-obstruction} then identify their common zero with the parameter of that canonical lift.

Suppose in addition that the deformation group of Brantner and Taelman exists and that its canonical formal lift belongs to $\mathcal D_{\bar t}$. Under their equivalence, the unit section $e$ corresponds to the canonical formal lift \cite[Theorem~B and Definition~8.31]{BrantnerTaelman2024}. We make the following conjecture.

\begin{conjecture}
\label{conj:full-canonical-comparison}
For every $t\in\widehat S_{\bar t}(W(k))$, one has $\sigma_{\mathrm{exc}}(t)=t$ if and only if $\iota(t)=e$.
\end{conjecture}

We note that both conjectures would follow from an integral comparison at all Witt orders between the construction of Beukers and Vlasenko and the geometric construction of the canonical lift.

\subsection{Motivation for the comparison conjectures}

There are several reasons to compare the excellent and geometric lifts. For a versal family of ordinary curves, Dwork and Ogus construct an obstruction to the Serre--Tate canonical lift of the Jacobian remaining in the Torelli locus modulo $p^2$. Under their Torelli hypotheses, the vanishing of this obstruction is equivalent to the existence of an excellent lifting of Frobenius \cite[\S2, especially Remark~(2.9)]{DworkOgus1986}. Their argument uses $H^1$ and the Torelli map, whereas the present problem concerns middle cohomology and does not use a Jacobian description.

The two obstruction maps arise from parallel cohomological conditions. By the pointwise uniqueness of the excellent Frobenius, a parameter $t$ is fixed if and only if $\lambda_1(t,t)=0$; see \cref{lem:pointwise} and \cite[Lemma~7.13 and Definition~7.14]{BeukersVlasenkoIII}. Under the hypotheses of Achinger and Zdanowicz, their canonical $W_2$-lift is characterized by the preservation of $\Fil^1$ under crystalline Frobenius \cite[Theorem~5.7.1]{AchingerZdanowicz2021}. Thus each construction selects a lift by a Frobenius preservation condition. Grabowski proves that the Achinger--Zdanowicz lift of an ordinary Dwork hypersurface remains in the Dwork family and computes its parameter in terms of Hasse--Dwork polynomials \cite[Theorem~3.17]{Grabowski2025}. This places the geometric lift in the same coordinates used in Dwork theory.

The rank two condition gives further evidence for the comparison. \Cref{thm:rank two-criterion} gives a cohomological condition for excellent fixedness. For the $A_4$ family, the same splitting also gives geometric canonical lifting modulo $p^2$ under Dummigan's Conjecture~1.1. \Cref{thm:main,thm:A4-W2-canonical,thm:A4-W2-comparison} verify \cref{conj:W2-comparison} on the entire $W_2$ residue disk and give  evidence for \cref{conj:full-canonical-comparison}.

\subsection*{Acknowledgments}
The author thanks  Frits Beukers, Neil Dummigan and Qingtao Chen for reading an earlier draft of this paper and for their helpful comments and suggestions.

She thanks Prof. Beukers for explaining the modular interpretation of excellent Frobenius discussed in \cref{sec:K3}, for sharing a note on excellent Frobenius and canonical lifting, and for raising the algebraicity question for attractor points discussed in \cref{sec:future}. She thanks Prof. Dummigan for suggesting a generalization of \cref{prop:birational-H3} to the covering varieties and for many writing suggestions. She is also grateful to him for sharing a revised version of his manuscript and for promptly communicating the correction to its modularity argument. She also thanks Prof. Qingtao Chen for suggestions that helped make the introduction clearer.

The author thanks OpenAI for providing access to GPT 5.6-sol, which assisted with the language and structural refinement of the paper. Furthermore, the model served as an initial mathematical referee for the first draft of the paper. In particular, it helps the author to weaken the original splitting hypothesis in \cref{thm:rank two-criterion}\textup{(iii)}.

\section{Preliminaries}

This section introduces the Dwork and crystalline objects used in the comparison and proves the preliminary results needed for the general criterion and its applications.

\subsection{Filtered isocrystals}

A filtered isocrystal over $\Q_p$ is a finite-dimensional $\Q_p$-vector space $H$ with an invertible $\Q_p$-linear map $\varphi$ and an exhaustive, separated, decreasing filtration $\Fil^iH$.  Let $\operatorname{gr}^iH=\Fil^iH/\Fil^{i+1}H$. The Hodge degrees are the integers $i$, with multiplicity $\dim\operatorname{gr}^iH$.  Let $v_p$ be the valuation on $\overline{\Q}_p$ normalized by $v_p(p)=1$. After extending scalars to $\overline{\Q}_p$, we group the generalized eigenspaces of $\varphi$ by the $v_p$-valuations of their eigenvalues.  The resulting slope subspaces descend to $\Q_p$ \cite{Katz1979}.  We write $H_{\varphi,\leq r}$ and $H_{\varphi,>r}$ for the sums of the slope subspaces with slopes at most $r$ and greater than $r$, respectively.

The Hodge and Newton numbers are defined by $t_H(H)=\sum_i i\dim\operatorname{gr}^iH$ and $t_N(H)=v_p(\det\varphi)$. The filtered isocrystal $H$ is \emph{weakly admissible} if $t_H(H)=t_N(H)$ and $t_H(D)\leq t_N(D)$ for every $\varphi$-stable subspace $D\subset H$, where $D$ has the induced filtration $\Fil^iD=D\cap\Fil^iH$; see \cite[\S3]{ColmezFontaine2000}. A decomposition $H=E\oplus H'$ as filtered isocrystals requires both summands to be $\varphi$-stable and $\Fil^iH=\Fil^iE\oplus\Fil^iH'$ for every $i$.

Our convention is that $\Q_p(-j)$ has Hodge degree $j$ and Frobenius $p^j$.  A nondegenerate Frobenius-compatible pairing $H\times H\to\Q_p(-d)$ therefore makes the Newton slopes symmetric under $\lambda\mapsto d-\lambda$.

For a $\Z_p$-module $M$, a $\Q_p$-linear map $\rho:M[1/p]\to H$ is \emph{bounded} if there exist a $\Z_p$-lattice $\Lambda\subset H$ and an integer $c\geq0$ such that $\rho(M)\subset p^{-c}\Lambda$.

\subsection{Dwork crystals and Cartier}
\label{sec:vertex-determinants}

We now define the Dwork crystal whose level two quotient will be compared with the filtered isocrystals above.

Let $p>2$, let $T$ be an indeterminate, and let $f_T=1-Tg(\mathbf x)$, where $g$ is an integral Laurent polynomial with reflexive Newton polytope $\Delta$. Let $\Gamma\subset\Z^n$ be the lattice generated by $\operatorname{Supp}(g)$. For the excellent Frobenius theorem \cite[Theorem~7.3]{BeukersVlasenkoIII}, assume $p\nmid[\Z^n:\Gamma]$. A \emph{monomial substitution} is an automorphism of the Laurent polynomial ring $R[x_1^{\pm1},\ldots,x_n^{\pm1}]$ of the form $x_i\mapsto\prod_jx_j^{a_{ij}}$, where $(a_{ij})\in\operatorname{GL}_n(\Z)$. Let $\mathcal G$ be a finite group of monomial substitutions preserving $g$, with $p\nmid|\mathcal G|$. Following \cite[\S7]{BeukersVlasenkoIII}, we say that $g$ is completely symmetric with respect to $\mathcal G$ if the nonzero lattice points of $\Delta$ are its vertices and $\mathcal G$ acts transitively on them. We assume from now on that $g$ is completely symmetric with respect to $\mathcal G$. Since $\mathcal G$ acts transitively on the vertices and preserves $g$, all vertex monomials have the same coefficient. We assume that this coefficient is a $p$-adic unit.

Let $R$ be the $p$-adic completion of $\Z_p[T,1/P(T)]$, where $P\in\Z_p[T]$ and $P(0)\in\Z_p^\times$. Expansion as formal power series at $T=0$ embeds $R$ in $\Z_p[[T]]$. A \emph{coefficient Frobenius} is a continuous endomorphism $\sigma:R\to R$ that restricts to the identity on $\Z_p$ and reduces modulo $p$ to the $p$-power Frobenius. We assume that $\sigma(T)=T^pv_\sigma(T)$ for some $v_\sigma(T)\in1+pR$. For a specialization $T\mapsto t_0$, let $u_0=\left.\sigma(T)\right|_{T=t_0}$. Thus $u_0\equiv t_0^p\pmod p$, and the specialized Cartier map has source $t_0$ and target $u_0$. For a lattice point $\mathbf m$, let $\deg_\Delta(\mathbf m)=\min\{r\geq0:\mathbf m\in r\Delta\}$. A Laurent polynomial $A=\sum_{\mathbf m\in\Gamma} a_{\mathbf m}(T)\mathbf x^{\mathbf m}$ is admissible if $\operatorname{ord}_T a_{\mathbf m}\geq\deg_\Delta(\mathbf m)$ for every $\mathbf m$. Let $L_{\mathrm{num}}$ be the module of $\mathcal G$-invariant admissible numerators.  These are the support lattice and admissibility conditions of \cite[Lemma~2.4, Definition~2.5, and Definition~7.1]{BeukersVlasenkoIII}. Write $f=f_T$. Let $\Omega_{L_{\mathrm{num}},f}(\Delta^\circ)$ be the $R$-module generated by $(m-1)!A(\mathbf x)/f(\mathbf x)^m$, where $m\geq1$, $A\in L_{\mathrm{num}}$, and $\operatorname{Supp}(A)\subset m\Delta^\circ$. Let $M=CY(g)=\widehat\Omega_{L_{\mathrm{num}},f}(\Delta^\circ)$ be its $p$-adic completion. This is the Calabi--Yau crystal of Beukers and Vlasenko \cite[\S7 and Definition~7.1]{BeukersVlasenkoIII}. All specializations below lie in the locus where $P(T)$ is a unit. For such a point $t_0$, we write $M_{t_0}=M\widehat\otimes_{R,t_0}\Z_p$ and use the same convention for its submodules and quotients. For $r\geq1$, the level submodule $M(r)$ consists of the elements with denominator $f_T^r$, as in \cite[Definitions~2.8--2.9]{BeukersVlasenkoIII}.

The Cartier operator is defined on formal Laurent expansions by
\[\cC_p\left(\sum_{\mathbf u}c_{\mathbf u}\mathbf x^{\mathbf u}\right) =\sum_{\mathbf u}c_{p\mathbf u}\mathbf x^{\mathbf u}.\]
It maps $M$ to the coefficient-twisted module $M^\sigma$ obtained by applying $\sigma$ to the coefficients of $f_T$ and of the numerators. Let $\Omega_f$ be the module generated by the elements $(m-1)!A/f_T^m$ with $\operatorname{Supp}(A)\subset m\Delta$, and let $\widehat\Omega_f$ be its $p$-adic completion. The inclusion $\Omega_{L_{\mathrm{num}},f}(\Delta^\circ)\hookrightarrow\Omega_f$ extends to the continuous Cartier-compatible map \cite[Proposition~2.7 and Definition~2.8]{BeukersVlasenkoIII}
\begin{equation}\label{eq:CY-full-inclusion}
 CY(g)\longrightarrow\widehat\Omega_f.
\end{equation}

For $r\geq1$, the module of $r$th formal derivatives is
\begin{equation*}
 \cF_rM=
 \left\{\eta\in M:\cC_p^j(\eta)\in p^{rj}M^{\sigma^j}
 \text{ for every }j\geq1\right\}.
\end{equation*}
This is Definition~3.1 of \cite{BeukersVlasenkoIII}. The first two Hasse--Witt conditions require the first Hasse--Witt scalar and the normalized second determinant $p^{-1}\det HW^{(2)}$ to be units; the matrices are defined in \cite[Definition~5.1]{BeukersVlasenkoIII}. Under these conditions, \cite[Theorem~4.2 and Corollary~5.9]{BeukersVlasenkoIII} give
\begin{equation}\label{eq:BV-split}
 M=M(2)\oplus\cF_2M.
\end{equation}
Let $Q^{(2)}=M/\cF_2M$. It has the following basis \cite[discussion following equation~(20)]{BeukersVlasenkoIII}:
\begin{equation*}
 \omega_0=\frac1{f_T},
 \quad
 \omega_1=\theta\left(\frac1{f_T}\right)=\frac{Tg}{f_T^2},
 \quad \theta=T\frac d{dT}.
\end{equation*}
We use the same symbols for these rational functions and their classes. Let $\omega_0^\sigma$ and $\omega_1^\sigma$ denote the coefficientwise $\sigma$-twists of $\omega_0$ and $\omega_1$; their classes form the corresponding basis of $M^\sigma/\cF_2M^\sigma$. Write the image of the first basis vector as
\begin{equation}\label{eq:lambda-row}
 \cC_p(\omega_0)
 \equiv\lambda_0\omega_0^\sigma+\lambda_1\omega_1^\sigma
 \pmod{p^2\cF_2(M^\sigma)}.
\end{equation}
For a specialization $T\mapsto t_0$ with target $u_0=\left.\sigma(T)\right|_{T=t_0}$, we write the induced Cartier map and its coefficients as $\cC_{t_0,u_0}:M_{t_0}\to M_{u_0}$ and $\lambda_i(t_0,u_0)$. Proposition~7.7 of \cite{BeukersVlasenkoIII} says that the universal coefficient $\lambda_1$ belongs to $pR$; in particular, $\lambda_1(t_0,u_0)\in p\Z_p$ after specialization. An \emph{excellent Frobenius} is a coefficient Frobenius for which the universal coefficient $\lambda_1$ vanishes identically; see \cite[Lemma~7.13, Definition~7.14, and Theorem~7.3]{BeukersVlasenkoIII}.

For $t_0\in\Z_p^\times$, the formula $\sigma_{t_0}(T)=t_0^{\,1-p}T^p$ defines a coefficient Frobenius on the completed residue disk of $t_0$ and fixes the point $t_0$. It also defines the endomorphism $C=\cC_{t_0,t_0}$ on $M_{t_0}[1/p]$. For a bounded map $\rho:M_{t_0}[1/p]\to H$ and a fixed integer $d$, Frobenius compatibility means $\rho C=p^d\varphi^{-1}\rho$. 

The Cartier slopes of a finite-dimensional $\Q_p$-space with an invertible Cartier operator are the $p$-adic valuations of its eigenvalues, counted with multiplicity after extending scalars to $\overline\Q_p$. We first determine the Cartier slopes of the level two quotient. This calculation will be used in \cref{lem:low-weight-ordinary,thm:Q2-comparison}.

\begin{lemma}
\label{lem:cartier-slopes}
Let $M_{t_0}$ be a specialization of the Calabi--Yau crystal defined above at $t_0\in\Z_p^\times$. Assume the first two Hasse--Witt conditions and use the auxiliary coefficient Frobenius defined above. Let $C=\cC_{t_0,t_0}$. Then the Cartier slopes on $(M_{t_0}/\cF_2M_{t_0})[1/p]$ are $0$ and $1$.
\end{lemma}

\begin{proof}
The first two Hasse--Witt conditions give a Cartier-stable exact sequence
\[0\longrightarrow\cF_1M_{t_0}/\cF_2M_{t_0} \longrightarrow M_{t_0}/\cF_2M_{t_0} \longrightarrow M_{t_0}/\cF_1M_{t_0}\longrightarrow0\]
whose left and right terms are free of rank one \cite[Corollary~5.9 and proof of Theorem~4.2]{BeukersVlasenkoIII}. The first condition makes Cartier invertible modulo $p$ on the right term. The proof of \cite[Theorem~4.2]{BeukersVlasenkoIII} at level two shows that $p^{-1}C$ is invertible modulo $p$ on the left term. The two slopes are therefore $0$ and $1$.
\end{proof}

\subsection{Vertex coefficients}

We next extract coefficient sequences from the holomorphic Dwork class. The Cartier relation for these sequences gives the supercongruences used later.

Let $\mathbf b$ be a vertex of $\Delta$, and let $c_{\mathbf b}(T)$ be the coefficient of $\mathbf x^{\mathbf b}$ in $f_T$. Assume that $c_{\mathbf b}(T)$ is a unit on the chosen residue disk. The expansion of $1/f_T$ at $\mathbf b$ is
\[\left(\frac1{f_T}\right)_{\mathbf b} = \frac{\mathbf x^{-\mathbf b}}{c_{\mathbf b}(T)} \sum_{j\geq0} \left(1-\frac{f_T}{c_{\mathbf b}(T)\mathbf x^{\mathbf b}}\right)^j.\]
Let $C_{\mathbf b}\subset\mathbb R^n$ be the cone generated by $\Delta-\mathbf b$; the support of this expansion lies in $C_{\mathbf b}$. For $\mathbf m\in C_{\mathbf b}\cap\Z^n$, let $a_{\mathbf m}(T)= \left[\mathbf x^{\mathbf m}\right]_{\mathbf b}(1/f_T)$, where the subscript $\mathbf b$ means coefficient extraction from the $\mathbf b$-expansion above. For parameter values $t,u$ in the chosen residue disk, $\mathbf m,\mathbf n\in C_{\mathbf b}\cap\Z^n$, and $s\geq1$, let
\[D_s^{\mathbf m,\mathbf n}(t,u)= a_{p^s\mathbf m}(t)a_{p^{s-1}\mathbf n}(u) -a_{p^s\mathbf n}(t)a_{p^{s-1}\mathbf m}(u).\]
The coefficient of $\mathbf x^{\mathbf m}$ in $\omega_1$ is $(\theta a_{\mathbf m})(T)$. If $\mathbf v=-\mathbf b$, let
\begin{equation}\label{eq:aQ-def-intro}
 a_Q(T)=a_{Q\mathbf v}(T)
 =\left[\mathbf x^{Q\mathbf v}\right]_{\mathbf b}\frac1{f_T}
 \quad(Q\geq1)
\end{equation}
and
\begin{equation}\label{eq:Ds-intro}
 D_s(t,u)=D_s^{\mathbf v,2\mathbf v}(t,u)
 =a_{p^s}(t)a_{2p^{s-1}}(u)
 -a_{2p^s}(t)a_{p^{s-1}}(u).
\end{equation}

We next show that the vanishing of the off-diagonal Cartier coefficient implies the vertex congruences. 

\begin{proposition}
\label{prop:excellent-congruence}
Assume $p>2$ and that the first two Hasse--Witt conditions hold on the residue disk of a point of $\F_p^\times$. Let $\sigma$ be a coefficient Frobenius, let $t$ be a point of this disk, and let $u=\left.\sigma(T)\right|_{T=t}$. Assume that $\lambda_1(t,u)=0$. Fix a vertex $\mathbf b$ for which the coefficient of $\mathbf x^{\mathbf b}$ in $f_T$ is a unit. For every $\mathbf m,\mathbf n\in C_{\mathbf b}\cap\Z^n$ and every $s\geq1$,
\begin{equation}\label{eq:excellent-congruence}
 D_s^{\mathbf m,\mathbf n}(t,u)\equiv0\pmod{p^{2s}}.
\end{equation}
\end{proposition}

\begin{proof}
For any $\mathbf r\in C_{\mathbf b}\cap\Z^n$, \cite[Proposition~5.11]{BeukersVlasenkoIII}, at level two, gives
\[a_{p^s\mathbf r}(t)\equiv \lambda_0(t,u)a_{p^{s-1}\mathbf r}(u) +\lambda_1(t,u)(\theta a_{p^{s-1}\mathbf r})(u) \pmod{p^{2s}}.\]
The coefficients of $g$ lie in $\Z$ and $\sigma$ restricts to the identity on $\Z_p$, so specialization at $t$ gives $\left.\sigma(a_{p^{s-1}\mathbf r}(T))\right|_{T=t} =a_{p^{s-1}\mathbf r}(u)$ and the same identity for $\theta a_{p^{s-1}\mathbf r}$. The second term vanishes because $\lambda_1(t,u)=0$. Apply the congruence with $\mathbf r=\mathbf m$ and $\mathbf r=\mathbf n$, cross multiply, and subtract.  This eliminates $\lambda_0(t,u)$ and proves \eqref{eq:excellent-congruence}.
\end{proof}
The following proposition extends the vertex congruences from residue degree one to a closed point of arbitrary residue degree.

\begin{proposition}
\label{prop:finite-frobenius-orbit}
Let $q=p^f$, let $k=\F_q$, and let $t_0\in W(k)^\times$. Assume that the Beukers--Vlasenko excellent Frobenius is defined after base change to $W(k)$, with Witt Frobenius on the constants.  Let $t_i=\sigma_{\mathrm{exc}}^i(t_0)$ for $0\leq i\leq f$.  Assume that the first two Hasse--Witt conditions hold along this orbit and that $t_f=t_0$. Fix a vertex $\mathbf b$ whose coefficient is a unit.  For $\mathbf m\in C_{\mathbf b}\cap\Z^n$, let $a_{\mathbf m}(t_i)$ be the coefficient of $\mathbf x^{\mathbf m}$ in the expansion of $1/f_{t_i}$ at $\mathbf b$. For every $\mathbf m,\mathbf n\in C_{\mathbf b}\cap\Z^n$ and every $s\geq1$,
\begin{align*}
 a_{q^s\mathbf m}(t_0)a_{q^{s-1}\mathbf n}(t_0)
 -a_{q^s\mathbf n}(t_0)a_{q^{s-1}\mathbf m}(t_0) \equiv0
 \pmod{p^{2(f(s-1)+1)}}.
\end{align*}
Equivalently, the modulus is $p^2q^{2(s-1)}$.
\end{proposition}

\begin{proof}
Write $c_i=\lambda_0(t_i,t_{i+1})\in W(k)$. Since $t_{i+1}=\sigma_{\mathrm{exc}}(t_i)$, one has $\lambda_1(t_i,t_{i+1})=0$. For every $S\geq1$ and every $\mathbf r\in C_{\mathbf b}\cap\Z^n$, \cite[Proposition~5.11]{BeukersVlasenkoIII}, applied at level two, gives
\[a_{p^S\mathbf r}(t_i)\equiv c_i a_{p^{S-1}\mathbf r}(t_{i+1}) \pmod{p^{2S}}.\]
Fix $N\geq0$ and apply these congruences successively with $S=N+f-i$ for $i=0,\ldots,f-1$.  The respective moduli are $p^{2(N+f)},p^{2(N+f-1)},\ldots,p^{2(N+1)}$.  Successive substitution therefore gives
\[a_{p^{N+f}\mathbf r}(t_0)\equiv \left(\prod_{i=0}^{f-1}c_i\right)a_{p^N\mathbf r}(t_f) \pmod{p^{2(N+1)}}.\]
Since $t_f=t_0$, apply this formula with $\mathbf r=\mathbf m$ and $\mathbf r=\mathbf n$, cross multiply, and subtract.  We obtain
\[a_{p^{N+f}\mathbf m}(t_0)a_{p^N\mathbf n}(t_0) -a_{p^{N+f}\mathbf n}(t_0)a_{p^N\mathbf m}(t_0) \equiv0\pmod{p^{2(N+1)}}.\]
Taking $N=f(s-1)$ proves the proposition.
\end{proof}

For $f>1$, the iteration gives only $p^{2(f(s-1)+1)}$. If one expect the stronger modulus $q^{2s}$ then it would follow from an $f$-fold Cartier relation whose error lies in $q^2\cF_2$.
\subsection{Canonical lifts and comparison maps}
\label{sec:canonical-lifts}
We first recall the geometric lifting construction used in the $W_2$ comparison. Recall that a smooth proper $d$-fold $X/k$ is $1$-ordinary if absolute Frobenius acts bijectively on $H^d(X,\mathcal O_X)$ \cite[Definition~1.2.1(a)]{AchingerZdanowicz2021}. Achinger and Zdanowicz attach to a $1$-ordinary smooth proper variety with trivial canonical bundle a canonical $W_2$-lift of its Frobenius twist \cite[Corollary~4.1.3]{AchingerZdanowicz2021}. Take $k=\F_p$ so that the Frobenius twist is canonically identified with $X$. Let $F_k$ denote absolute Frobenius and let $\sigma_W$ denote Witt-vector Frobenius. For a lift $Y/W_2(k)$ of $X/k$, crystalline--de Rham comparison identifies $H^d_{\cris}(X/W_2(k))$ with $H^d_{\dR}(Y/W_2(k))$. For each lift $Y$, let $\Fil_Y^\bullet$ denote the filtration on $H^d_{\cris}(X/W_2(k))$ corresponding to the Hodge filtration on $H^d_{\dR}(Y/W_2(k))$ under this identification. We use the linearized crystalline Frobenius
\[\Phi:\sigma_W^*H^d_{\cris}(X/W_2(k)) \longrightarrow H^d_{\cris}(X/W_2(k)).\]
Under the hypotheses stated in \cite[\S1.2]{Grabowski2025}, crystalline Frobenius satisfies $\Phi(\sigma_W^*\Fil_Y^1)\subset pH^d_{\cris}(X/W_2(k))$. The divided Frobenius therefore defines an obstruction $\gamma_Y:F_k^*\operatorname{gr}^1\to\operatorname{gr}^0$. By the definition of $\gamma_Y$, its vanishing is equivalent to $\Phi(\sigma_W^*\Fil_Y^1)\subset\Fil_Y^1$. Achinger and Zdanowicz prove that crystalline Frobenius preserves $\Fil^1$ for the canonical lift \cite[Theorem~5.0.1]{AchingerZdanowicz2021}. Hence the canonical lift has zero obstruction.

We now define the two obstruction maps and the comparison isomorphism between their target spaces.

Let $S$ be smooth of relative dimension one over $W(k)$, and let $\mathscr X\to S$ be a smooth family of Calabi--Yau $d$-folds with a local coordinate $T$ and a Laurent model $f_T=1-Tg$ that is completely symmetric with respect to $\mathcal G$. Fix $\bar t\in S(k)$ such that $T(\bar t)\in k^\times$ and the first two Hasse--Witt conditions hold. Let $X_{\bar t}=\mathscr X\times_{S,\bar t}\operatorname{Spec}k$.  Assume that $X_{\bar t}$ is $1$-ordinary and that the level two Dwork quotient and the excellent Frobenius are defined on the residue disk. A marked lift of $X_{\bar t}$ is a smooth proper lift $Y/W_2(k)$ equipped with an identification $Y\otimes_{W_2(k)}k\simeq X_{\bar t}$. Fix either the polarized deformation problem or the $\mathcal G$-equivariant deformation problem, and let $\operatorname{Def}(X_{\bar t})$ denote its formal deformation space. Let $\widehat S_{\bar t}$ be the formal completion of the base $S$ at $\bar t$, and let $\iota:\widehat S_{\bar t}\to\operatorname{Def}(X_{\bar t})$ be the classifying morphism. Assume that $\iota$ identifies $\widehat S_{\bar t}$ with a one-dimensional formal subscheme $\mathcal D_{\bar t}\subset\operatorname{Def}(X_{\bar t})$. Let $Y_{\mathrm{can}}$ denote the Achinger--Zdanowicz canonical lift, regarded as a marked lift of $X_{\bar t}$ through the canonical identification $X_{\bar t}^{(p)}\simeq X_{\bar t}$. Assume that the $W_2(k)$-point defined by $Y_{\mathrm{can}}$ belongs to $\mathcal D_{\bar t}(W_2(k))$. Assume that the hypotheses of \cite[\S1.2]{Grabowski2025} hold for every marked lift parametrized by $\mathcal D_{\bar t}$.

For $t\in\widehat S_{\bar t}(W_2(k))$, let $a=T(t)\in W_2(k)^\times$ and let $\mathscr X_t=\mathscr X\times_{S,t}\operatorname{Spec}W_2(k)$ with its natural marking. Let $Q_t=Q^{(2)}\otimes_{R,T\mapsto a}W_2(k)$, and let $L_t\subset Q_t$ be the rank one direct summand generated by $1/f_a$ in the chosen Dwork basis.  Since the reductions depend only on $\bar t$, write $\overline Q_{\bar t}=Q_t/pQ_t$ and $\overline L_{\bar t}=L_t/pL_t$. Because $k=\F_p$, the Witt-vector Frobenius on $W_2(k)$ is the identity. For the source value $a$, the auxiliary coefficient Frobenius $\sigma_a(T)=a^{1-p}T^p$ has target $a$. It induces a $W_2(k)$-linear Cartier map $\mathcal C_t:Q_t\to Q_t$ from $\cC_{a,a}$.

Let $N_t=Q_t/L_t$ and let $\pi_t:Q_t\to N_t$ be the quotient map. The divisibility $\lambda_1\in pR$ from \cite[Proposition~7.7]{BeukersVlasenkoIII} gives $\pi_t\mathcal C_t(L_t)\subset pN_t$. Since $\mathcal C_t$ is $W_2(k)$-linear and $p^2=0$, one has $\pi_t\mathcal C_t(pL_t)=0$. Thus $\pi_t\mathcal C_t|_{L_t}$ descends to a map $\overline L_{\bar t}\to pN_t$. The module $N_t$ is free over $W_2(k)$, so multiplication by $p$ induces an isomorphism $N_t/pN_t\xrightarrow{\sim}pN_t$. Composing with its inverse defines
\[\operatorname{ob}_D(t)\in \operatorname{Hom}_k\bigl(\overline L_{\bar t}, \overline Q_{\bar t}/\overline L_{\bar t}\bigr).\]
In the Dwork basis $(\omega_0,\omega_1)$, the scalar of this map is the class $\overline{\lambda_1(a,a)/p}\in k$, defined as the unique class whose image under multiplication by $p$ is $\lambda_1(a,a)$.

Let $S_k=S\otimes_{W(k)}k$, and let $I_{\bar t}$ be the image of the Kodaira--Spencer map
\[T_{\bar t}S_k\otimes\operatorname{gr}^dH^d_{\dR}(X_{\bar t}/k) \longrightarrow\operatorname{gr}^{d-1}H^d_{\dR}(X_{\bar t}/k), \quad v\otimes\omega\longmapsto \operatorname{KS}_{\bar t}(v)\mathbin{\lrcorner}\omega.\]
Here $\lrcorner$ denotes the contraction induced by $T_{X_{\bar t}}\otimes\Omega^d_{X_{\bar t}} \to\Omega^{d-1}_{X_{\bar t}}$. Assume that $I_{\bar t}$ is one-dimensional. Choose a line $K_{\bar t}\subset\operatorname{gr}^1H^d_{\dR}(X_{\bar t}/k)$ that pairs perfectly with $I_{\bar t}$ under polarization. The Achinger--Zdanowicz obstruction restricted to $\mathcal D_{\bar t}$ is
\[\operatorname{ob}_C(t) :=\gamma_{\mathscr X_t}|_{F_k^*K_{\bar t}}\in \operatorname{Hom}_k\bigl(F_k^*K_{\bar t}, \operatorname{gr}^0H^d_{\dR}(X_{\bar t}/k)\bigr).\]
As part of the comparison data, assume that the integral residue comparison identifies $\overline L_{\bar t}$ with $\operatorname{gr}^dH^d_{\dR}(X_{\bar t}/k)$ and $\overline Q_{\bar t}/\overline L_{\bar t}$ with $I_{\bar t}$. Together with polarization and the choice of $K_{\bar t}$, these identifications define an isomorphism
\[\kappa_{\bar t}:\operatorname{Hom}_k\bigl(\overline L_{\bar t}, \overline Q_{\bar t}/\overline L_{\bar t}\bigr) \xrightarrow{\sim} \operatorname{Hom}_k\bigl(F_k^*K_{\bar t}, \operatorname{gr}^0H^d_{\dR}(X_{\bar t}/k)\bigr).\]
Because $k=\F_p$ we have canonical identifications $F_k^*V=V$ for $k$-vector spaces $V$.

Assume that $\omega_{X_{\bar t}}\simeq\mathcal O_{X_{\bar t}}$, and that all the modules $H^j(Y,\Omega^i_{Y/W_2(k)})$ are free over $W_2(k)$ and the Hodge--de Rham spectral sequence for $Y/W_2(k)$ degenerates at $E_1$ for every marked lift $Y/W_2(k)$ of $X_{\bar t}$. The next lemma identifies the zero locus of the crystalline obstruction with the canonical lift. This is the geometric vanishing criterion used in the comparison conjectures.

\begin{lemma}
\label{lem:canonical-obstruction}
For every $t\in\widehat S_{\bar t}(W_2(k))$, one has $\operatorname{ob}_C(t)=0$ if and only if $\mathscr X_t$ and $Y_{\mathrm{can}}$ are isomorphic as marked lifts.
\end{lemma}

\begin{proof}
Fix $t$ and let $Y=\mathscr X_t$. The kernel of $W_2(k)\to k$ is square-zero. Hence the isomorphism classes of marked lifts of $X_{\bar t}$ form a torsor under $H^1(X_{\bar t},T_{X_{\bar t}})$. Using $pW_2(k)\simeq k$, let $v(Y,Y_{\mathrm{can}})\in H^1(X_{\bar t},T_{X_{\bar t}})$ denote the unique class whose torsor action on $Y_{\mathrm{can}}$ gives $Y$. The differential of $\iota$ identifies the tangent line of $\mathcal D_{\bar t}$ with a one-dimensional subspace of $H^1(X_{\bar t},T_{X_{\bar t}})$. Since both $Y$ and $Y_{\mathrm{can}}$ are parametrized by $\mathcal D_{\bar t}$, their difference class belongs to this subspace.

\cite[Theorem~5.0.1]{AchingerZdanowicz2021} gives $\gamma_{Y_{\mathrm{can}}}=0$. Let $\eta(v)$ denote the infinitesimal period map $\operatorname{gr}^1\to\operatorname{gr}^0$ determined by $v=v(Y,Y_{\mathrm{can}})$, and let $HW(0):F_k^*\operatorname{gr}^0\to\operatorname{gr}^0$ be the Hasse--Witt map. The calculation in the proof of \cite[Theorem~5.7.1]{AchingerZdanowicz2021} gives
\[\gamma_Y-\gamma_{Y_{\mathrm{can}}} =HW(0)\circ F_k^*\eta(v).\]
The map $\eta(v)$ is cup product by the difference class \cite[Corollary~2.12]{Ogus1978}. Its compatibility with polarization gives
\[\langle\eta(v)(\alpha),\omega\rangle =\pm\langle\alpha,v\mathbin{\lrcorner}\omega\rangle\]
for $\alpha\in K_{\bar t}$ and $\omega\in\operatorname{gr}^dH^d_{\dR}(X_{\bar t}/k)$. For $v$ on the tangent line of $\mathcal D_{\bar t}$, the class $v\mathbin{\lrcorner}\omega$ belongs to $I_{\bar t}$ and is nonzero when $v\neq0$. Since $K_{\bar t}$ pairs perfectly with $I_{\bar t}$, this identity shows that $\eta(v)|_{K_{\bar t}}=0$ if and only if $v=0$. Since $X_{\bar t}$ is $1$-ordinary, $HW(0)$ is an isomorphism. Therefore $\operatorname{ob}_C(t)=0$ if and only if $v(\mathscr X_t,Y_{\mathrm{can}})=0$, which is equivalent to an isomorphism of the two marked lifts.
\end{proof}

\section{A general criterion}

This section proves \cref{thm:rank two-criterion}. We first identify the vanishing of the off-diagonal Cartier coefficient with the excellent Frobenius target. We then compare the Dwork filtration with the Newton filtration.

\subsection{Pointwise uniqueness for the excellent Frobenius}

The following pointwise uniqueness statement on a unit residue disk follows from Taylor transport for the Gauss--Manin connection.

\begin{lemma}
\label{lem:pointwise}
Let $k=\F_p$.  Fix a source $t\in\Z_p^\times$ in a unit residue disk on which the first two Hasse--Witt conditions hold.  Suppose the completely symmetric Calabi--Yau setting above and that the hypotheses of \cite[Theorem~7.3]{BeukersVlasenkoIII} hold.  Let $v=\sigma_{\mathrm{exc}}(t)$.  If $u\in\Z_p^\times$ satisfies $u\equiv t^p\pmod p$, then $\lambda_1(t,u)=0$ if and only if $u=v$.
\end{lemma}

\begin{proof}
Over the residue disk, let $e_0=[1/f]$ and $e_1=[\theta(1/f)]$ be sections of the level two quotient, and write $e_i(w)$ for their values at $w$.  The first two Hasse--Witt conditions make $e_0(w),e_1(w)$ a basis of $Q_w^{(2)}$ by \cite[discussion following equation~(20)]{BeukersVlasenkoIII}.  By \cite[Proposition~7.6]{BeukersVlasenkoIII}, there are integral analytic functions $A$ and $B$ on the residue disk such that $\theta e_0=e_1$ and $\theta e_1=Ae_0+Be_1$.  For $r\geq0$, write $\theta^re_0=a_re_0+b_re_1$, where $b_0=0$ and $b_1=1$.  The connection equations show by induction that $a_r$ and $b_r$ are integral analytic functions on the same disk.

Theorems~7.3 and~7.15 of \cite{BeukersVlasenkoIII} show that the excellent target $v$ is defined on the residue disk.  Both $u$ and $v$ are congruent to $t^p$ modulo $p$, so $x=\log(v/u)$ belongs to $p\Z_p$. Let $T_{v,u}:M_v\to M_u$ be the map obtained by expanding the denominators $f_v^{-m}$ in powers of $(v-u)g/f_u$.  In $M_u$ one has
\[T_{v,u}(1/f_v) =\sum_{r\geq0}\frac{x^r}{r!} \left.\theta^r(1/f_t)\right|_{t=u}.\]
The identity $f_v=f_u-(v-u)g$ gives these binomial expansions.  For an interior generator with numerator supported in $m\Delta^\circ$, the term of order $j$ has numerator supported in $m\Delta^\circ+j\Delta\subset(m+j)\Delta^\circ$.  Thus the expansion preserves the Calabi--Yau module.  It converges because $v-u\in p\Z_p$, and interchanging $u$ and $v$ gives the inverse.  By \cite[Corollary~3.4]{BeukersVlasenkoIII}, $\cF_2M=M\cap d^2\Omega_{\mathrm{formal}}$, where $d^2\Omega_{\mathrm{formal}}$ is the submodule of second formal derivatives defined there. Under the embeddings into the common formal Laurent-series module, $T_{v,u}$ is induced by the identity and commutes with logarithmic differentiation in the $\mathbf x$-variables.  Hence $T_{v,u}(\cF_2M_v)=\cF_2M_u$. It therefore induces the Taylor isomorphism $Q_v^{(2)}\xrightarrow{\sim}Q_u^{(2)}$.  The Gauss--Manin Taylor identity \cite[the calculation immediately preceding Corollary~6.4] {BeukersVlasenkoIII} gives
\[T_{v,u}(e_0(v)) =\sum_{r\geq0}\frac{x^r}{r!}(\theta^re_0)(u) =\alpha(x)e_0(u)+\beta(x)e_1(u) \quad\text{in }Q_u^{(2)},\]
where $\beta(x)=x+\sum_{r\geq2}b_r(u)x^r/r!$. For $w\equiv t^p\pmod p$, coefficient extraction defines the specialized Cartier map $\cC_{t,w}:M_t\to M_w$; this map depends only on its source and target.  Since Taylor transport is induced by the identity on the common formal Laurent-series module, one has $T_{v,u}\circ\cC_{t,v}=\cC_{t,u}$ and $T_{v,u}(p^2\cF_2M_v)=p^2\cF_2M_u$.  We use the same symbols for the induced maps on the level two quotients.

For $r\geq1$, Legendre's estimate gives $v_p(r!)\leq(r-1)/(p-1)$. Since $b_r(u)\in\Z_p$, $x\in p\Z_p$, and $p$ is odd, the series defining $\beta(x)$ converges, and $U(x):=1+\sum_{r\geq2}b_r(u)x^{r-1}/r!$ belongs to $1+p\Z_p$. Indeed, each summand belongs to $p\Z_p$, and its $p$-adic valuation tends to infinity.  Thus $\beta(x)=xU(x)$.

For $v=\sigma_{\mathrm{exc}}(t)$, \cite[Theorem~7.3]{BeukersVlasenkoIII} gives $\cC_{t,v}(e_0(t))=c\,e_0(v)$ in $Q_v^{(2)}$, where $c$ is a unit by the first Hasse--Witt condition. Applying $T_{v,u}$ and using $T_{v,u}\circ\cC_{t,v}=\cC_{t,u}$ gives
\[\cC_{t,u}(e_0(t)) =c\{\alpha(x)e_0(u)+\beta(x)e_1(u)\} \quad\text{in }Q_u^{(2)}.\]
Comparison in the basis $(e_0(u),e_1(u))$ yields $\lambda_1(t,u)=c\beta(x)=cxU(x)$. Since $c$ and $U(x)$ are units, $\lambda_1(t,u)=0$ if and only if $x=0$. The $p$-adic logarithm is injective on $1+p\Z_p$ for odd $p$, so $x=0$ if and only if $u=v$.
\end{proof}

\subsection{Slopes and the rank two subisocrystal}

The following lemma relates the Dwork filtration to the Newton filtration.

\begin{lemma}
\label{lem:slope}
Fix an integer $d$, and let $K=\Q_p$.  Let $M$ be a $p$-torsion free, $p$-adically separated and complete $\Z_p$-module with a $K$-linear Cartier operator $C$ on $M[1/p]$, and let $(H,\varphi)$ be a finite-dimensional $K$-vector space with a $K$-linear invertible Frobenius.  Suppose that a surjection $\rho:M[1/p]\twoheadrightarrow H$ satisfies $\rho C=p^d\varphi^{-1}\rho$ and $\rho(M)\subset p^{-c}\Lambda$ for a lattice $\Lambda\subset H$ and a fixed integer $c\geq0$. For $r\geq1$, let $\cF_rM=\{x\in M:C^sx\in p^{rs}M\text{ for every }s\geq1\}$.  Then $\rho(\cF_rM)\subset H_{\varphi,\leq d-r}$, and $\rho$ induces a surjection $(M/\cF_rM)[1/p]\twoheadrightarrow H/H_{\varphi,\leq d-r}$. If the two sides have equal dimension, this map is an isomorphism.
\end{lemma}

\begin{proof}
Let $x\in\cF_rM$.  For every $s\geq1$ one has $p^{ds}\varphi^{-s}\rho(x)=\rho(C^sx)\in p^{rs-c}\Lambda$. After a finite scalar extension, decompose $H$ into its slope subspaces and let $\operatorname{pr}_\lambda$ be the projector onto the slope-$\lambda$ subspace.  This projector is bounded for any fixed lattice.  For a vector in a slope subspace, $v_p$ denotes the minimum valuation of its coordinates in a fixed basis. If $\operatorname{pr}_\lambda\rho(x)\neq0$, the growth characterization of Newton slopes \cite{Katz1979} gives
\[\lim_{s\to\infty}\frac1s v_p\!\left(p^{ds}\varphi^{-s} \operatorname{pr}_\lambda\rho(x)\right)=d-\lambda.\]
The lattice inclusion above and the boundedness of $\operatorname{pr}_\lambda$ imply that the same limit is at least $r$.  Thus $\lambda\leq d-r$.  This proves the containment, which descends to $K$. The induced map on quotients is surjective because $\rho$ is surjective. If the source and target have the same finite dimension, the induced surjection is an isomorphism.
\end{proof}

The slope lemma now gives the comparison in the general criterion.

\begin{proof}[Proof of \cref{thm:rank two-criterion}]
Let $M=M_{t_0}$, let $Q^{(2)}=M/\cF_2M$, and let $H_{\mathrm{hi}}=H/H_{\varphi,\leq d-2}$.  Write $\bar\omega_i$ for the class of $\omega_i=\theta^i(1/f_{t_0})$ in $Q^{(2)}$.  Applying \cref{lem:slope} with $r=2$ gives a surjection $\bar\rho:Q^{(2)}[1/p]\twoheadrightarrow H_{\mathrm{hi}}$.  The first two Hasse--Witt conditions make $(\bar\omega_0,\bar\omega_1)$ a basis of $Q^{(2)}[1/p]$, while assumption~\textup{(ii)} gives $\dim H_{\mathrm{hi}}=2$.  Hence $\bar\rho$ is an isomorphism.  It intertwines $C$ with $p^d\varphi^{-1}$.  Since the Cartier slopes on the source are $0$ and $1$, the Frobenius slopes on $H_{\mathrm{hi}}$ are $d$ and $d-1$.

Let $h_0=\rho(\omega_0)=\rho(1/f_{t_0})$, and let $\bar h_0$ be its image in $H_{\mathrm{hi}}$.  Then $\bar h_0=\bar\rho(\bar\omega_0)\neq0$, because $\bar\omega_0$ is a basis vector and $\bar\rho$ is an isomorphism.  By functoriality of the slope decomposition for $\varphi$-stable subspaces, and because $d\geq2$ and the slopes of $E$ are $0$ and $d$, the intersection $E_0=E\cap H_{\varphi,\leq d-2}$ is the one-dimensional slope zero subisocrystal of $E$.  Consequently the image
\[\bar E=(E+H_{\varphi,\leq d-2})/H_{\varphi,\leq d-2} \simeq E/E_0\]
is a one-dimensional $\varphi$-stable subspace of $H_{\mathrm{hi}}$, of slope $d$. The assumption $h_0\in E$ and the nonvanishing of $\bar h_0$ show that $\bar E=\Q_p\bar h_0$.  Thus $\Q_p\bar h_0$ is stable under $p^d\varphi^{-1}$.  Transporting this line through $\bar\rho$ gives $\bar\rho^{-1}(\Q_p\bar h_0)=\Q_p\bar\omega_0$, so $\Q_p\bar\omega_0$ is Cartier-stable.  In the basis $(\bar\omega_0,\bar\omega_1)$ this is exactly $\lambda_1(t_0,t_0)=0$.

The determinant congruence follows from \cref{prop:excellent-congruence}, applied to the auxiliary coefficient Frobenius $\sigma_{t_0}$ and the equality $\lambda_1(t_0,t_0)=0$.

When the Beukers--Vlasenko excellent Frobenius theorem applies, $\lambda_1(t_0,t_0)=0$ and \cref{lem:pointwise} give $\sigma_{\mathrm{exc}}(t_0)=t_0$.
\end{proof}

We also record the resulting Newton polygons in weights two and three. The weight two case will be used for the mirror quartic K3 pencil example.

\begin{lemma}
\label{lem:low-weight-ordinary}
Let $M,H,\rho,C,\varphi$ satisfy the hypotheses of \cref{lem:slope}. Assume that $M$ is the specialization of a completely symmetric Calabi--Yau crystal at a point of $\Z_p^\times$ fixed by the Frobenius lift on the coefficient ring, that the first two Hasse--Witt conditions hold, and that $H$ has a weight-$d$ polarization.  If $(d,\dim H)=(2,3)$ or $(3,4)$, then the map induced by $\rho$ is an isomorphism $(M/\cF_2M)[1/p]\simeq H/H_{\varphi,\leq d-2}$ that intertwines $C$ with $p^d\varphi^{-1}$.  The slopes of $H$ are respectively $\{0,1,2\}$ or $\{0,1,2,3\}$.
\end{lemma}

\begin{proof}
Polarization makes the Newton multiset symmetric under $\lambda\leftrightarrow d-\lambda$.  In weight two and rank three, the central slope is $1$ and every remaining symmetric pair contributes a slope strictly greater than $0$.  Thus $\dim H_{\varphi,>0}\geq2$.  In weight three and rank four, each symmetric pair contributes at least one slope strictly greater than $1$, so $\dim H_{\varphi,>1}\geq2$.

In either case \cref{lem:slope} gives a surjection from the two-dimensional space $(M/\cF_2M)[1/p]$ onto the high slope quotient. The lower bound shows that this quotient also has dimension at least two, so the surjection is an isomorphism.  By \cref{lem:cartier-slopes}, the Cartier slopes $0,1$ become the Frobenius slopes $d,d-1$, and polarization gives the complementary slopes displayed above.
\end{proof}

\section{A bounded Dwork--crystalline comparison}

The following theorem constructs the comparison map in \cref{thm:rank two-criterion} from a restriction isomorphism and a residue spanning condition for a compactification. We will use it in \cref{sec:bounded-residue,prop:K3-BR}. Write $\mathcal G$ for the monomial symmetry group defining $CY(g)$ and $G$ for the group acting on the compactification.

Assume that $p>2$.  Let $k=\F_p$, $W=W(k)=\Z_p$, and $K=W[1/p]$. Let $d\geq1$ and $n=d+1$.  Let $f_t=1-tg$ and $CY(g)$ be as in \cref{sec:vertex-determinants}. Choose $t_0\in W^\times$ and a coefficient Frobenius that fixes $t_0$, and let $f=f_{t_0}$ and $M=CY(g)_{t_0}$.  Let $\Omega_T=d\log x_1\wedge\cdots\wedge d\log x_n$. Let $\Delta$ be the Newton polytope of $f$. Let $M_{\mathrm{alg}}\subset M$ be the image under specialization at $t_0$ of the module $\Omega_{L_{\mathrm{num}},f_t}(\Delta^\circ)$ defined in \cref{sec:vertex-determinants}. It consists of finite sums of specialized admissible, $\mathcal G$-invariant interior generators and is dense in $M$. Let $C=\cC_{t_0,t_0}:M\to M$ be the Cartier operator associated with this Frobenius lift.  Let
\[\mathscr Y=V(f)\subset(\Gm^n)_W,\quad Y=\mathscr Y_k,\quad U=(\Gm^n)_k\setminus Y.\]
\begin{theorem}
\label{thm:general-bounded-residue}
Let $G$ be a finite abstract group acting on $(\Gm^n)_W$ by monomial substitutions defined over $W$. Assume that the following conditions hold.
\begin{enumerate}[label=\textup{(\roman*)}]
\item The hypersurface $\mathscr Y$ has a $G$-equivariant smooth projective compactification $\widetilde{\mathscr Y}/W$ whose boundary is a relative simple normal crossings divisor.
\item The action of $G$ preserves $f$ and the logarithmic volume form, and $|G|$ is prime to $p$.  Let
\[e_G=|G|^{-1}\sum_{g\in G}g, \quad H^\circ=e_GH^d_{\rig}(Y/K).\]
\item Let $V=H^d_{\cris}(\widetilde{\mathscr Y}_k/W)[1/p]$, identified with proper rigid cohomology and equipped with the Hodge filtration coming from $\widetilde{\mathscr Y}/W$.  Let $W_\bullet$ denote the weight filtration on $H^d_{\rig}(Y/K)$.  There is a direct summand $H\subset e_GV$ in the category of weakly admissible filtered isocrystals such that restriction induces a Frobenius-equivariant isomorphism
\begin{equation*}
 j_H:H\xrightarrow{\ \sim\ }W_dH^\circ.
\end{equation*}
\item Let $h=\dim_KH$.  There are integers $r_1,\ldots,r_h\geq0$ and a basis $h_1,\ldots,h_h$ of $H$ such that
\begin{equation}\label{eq:general-proper-residues}
 e_G\Res_{\rig}\!\left(
 \left[\left.\theta^{r_i}(1/f_t)\right|_{t=t_0}\Omega_T\right]\right)
 =j_H(h_i)
 \quad(1\leq i\leq h).
\end{equation}
Here we identify the underlying $K$-space of the Poincar\'e residue target $H^d_{\rig}(Y/K)(-1)$ with $H^d_{\rig}(Y/K)$, and $\theta=t\,d/dt$ is applied before specialization at $t_0$.
\end{enumerate}
Then there is a continuous surjection \(\rho:M[1/p]\twoheadrightarrow H.\) For every $W$-lattice $\Lambda\subset H$, there is an integer $c\geq0$ such that $\rho(M)\subset p^{-c}\Lambda$.  If $\varphi$ is crystalline pullback Frobenius on $H$, then
\begin{equation}\label{eq:general-Cartier-Frobenius}
 \rho C=p^d\varphi^{-1}\rho.
\end{equation}
\end{theorem}

\begin{proof}
Let $c_\Delta=\max_{v\in\Delta}\sum_{i=1}^n|v_i|$.  The exponential presentation of \cite[Proposition~A.3 and the paragraph following Corollary~A.4]{BeukersVlasenkoI} identifies $\widehat\Omega_f$ with $\widehat{W[\Delta]}^+/ \mathscr D_{0,f}(\widehat{W[\Delta]}^+)$. Here $\widehat{W[\Delta]}^+$ consists of series $\xi=\sum_{m,u}a_{m,u}X_0^mX^u$ with $m\geq1$, $u\in m\Delta\cap\Z^n$, $a_{m,u}\in W$, and $\ord_p(a_{m,u})\to\infty$ as $m\to\infty$.  Moreover, $\mathscr D_{0,f}=X_0\partial_{X_0}+X_0f$, where $X_0$ is an auxiliary variable and $X^u=X_1^{u_1}\cdots X_n^{u_n}$. We equip $\widehat{W[\Delta]}^+$ with the norm $\|\xi\|_p=\sup_{m,u}|a_{m,u}|_p$.  On finite sums define
\begin{equation*}
 \widetilde{\mathcal L}_f(\xi)=
 \sum_{m,u}(-1)^ma_{m,u}(m-1)!\frac{x^u}{f^m}\Omega_T.
\end{equation*}
Choose $r\in p^\Q$ such that $1<r<p^{1/((p-1)(c_\Delta+1))}$, and consider the strict neighbourhood $U_r=\{r^{-1}\leq|x_i|\leq r,\ |f|\geq r^{-1}\}$. We write $\|\cdot\|_r$ for the supremum norm on $U_r$. For $u\in m\Delta$, convexity gives $\lVert u\rVert_1\leq mc_\Delta$, and hence
\[\left\|(m-1)!\frac{x^u}{f^m}\right\|_r \leq |(m-1)!|_p r^{(c_\Delta+1)m}.\]
Legendre's formula gives $\lim_{m\to\infty}|(m-1)!|_p^{1/m}=p^{-1/(p-1)}$, so the right hand side tends to zero exponentially.  In particular,
\[\|\widetilde{\mathcal L}_f(\xi)\|_r\leq \left(\sup_{m\geq1}|(m-1)!|_pr^{(c_\Delta+1)m}\right)\|\xi\|_p.\]
Since each degree contains only finitely many lattice points, the series converges uniformly on $U_r$ and defines a continuous map to the dagger top forms.

On a monomial $X_0^mX^u$, the two terms of $\mathscr D_{0,f}(X_0^mX^u)$ map under $\widetilde{\mathcal L}_f$ to $(-1)^m m!x^u\Omega_T/f^m$ and $(-1)^{m+1}m!x^uf\Omega_T/f^{m+1}$, respectively. So their sum is zero. For $1\leq i\leq n$, let $\mathscr D_{i,f}=X_i\partial_{X_i}+X_0X_i\partial f/\partial X_i$. On finite sums, $\widetilde{\mathcal L}_f$ intertwines $\mathscr D_{i,f}$ with logarithmic differentiation $x_i\partial/\partial x_i$ of the coefficient of $\Omega_T$ \cite[Proposition~A.3]{BeukersVlasenkoI}. For every dagger function $h$ on $U$, the form $(x_i\partial h/\partial x_i)\Omega_T$ is exact.  Exact top forms are closed in the dagger de Rham complex \cite[Theorem~3.1 and Corollary~3.2]{GrosseKlonne2004}.  Thus $\widetilde{\mathcal L}_f$ descends through the $\mathscr D_{0,f}$ quotient, and the other twisted derivatives map to exact forms.  Passing to cohomology and applying Monsky--Washnitzer--rigid comparison \cite[Proposition~3.6]{GrosseKlonne2004} give a continuous map
\[\mathcal L_f:\widehat\Omega_f[1/p]\longrightarrow H^n_{\rig}(U/K).\]
We use the same notation for its composite with the map $M[1/p]\to\widehat\Omega_f[1/p]$ induced by \eqref{eq:CY-full-inclusion}. For a normalized generator $\eta=(m-1)!x^u/f^m$, its exponential representative is $(-1)^mX_0^mX^u$, and hence $\mathcal L_f(\eta)=[\eta\Omega_T]$.

We next construct the projector onto $W_dH^\circ$. The cohomology $H^d_{\rig}(Y/K)$ has weights in $[d,2d]$ \cite[Theorem~2.2]{Chiarellotto1998Weights}. Let $\varphi^\circ$ denote Frobenius on $H^d_{\rig}(Y/K)$ and also its restriction to $H^\circ$, and let
\[P_d(X)=\det(X-\varphi^\circ\mid W_dH^\circ),\quad P_{>d}(X)=\det(X-\varphi^\circ\mid H^\circ/W_dH^\circ).\]
The roots of these polynomials have different Weil weights, so the polynomials are coprime.  Choose $U_d,V_d\in K[X]$ such that $U_dP_d+V_dP_{>d}=1$ and let
\begin{equation*}
 \Pi_d=V_d(\varphi^\circ)P_{>d}(\varphi^\circ).
\end{equation*}
Cayley--Hamilton and coprimality show that $\Pi_d$ is the projector of $H^\circ$ onto $W_dH^\circ$.  It commutes with Frobenius.

Rigid Poincar\'e residue takes values in $H^d_{\rig}(Y/K)(-1)$.  To define $\rho$, we identify the underlying $K$-space of this Tate twist with that of $H^d_{\rig}(Y/K)$. Under this identification, Frobenius on $H^d_{\rig}(Y/K)(-1)$ corresponds to $p\varphi^\circ$ on $H^d_{\rig}(Y/K)$. Let
\begin{equation}\label{eq:general-rho}
 \rho=j_H^{-1}\circ\Pi_d\circ e_G\circ
 \Res_{\rig}\circ\mathcal L_f.
\end{equation}
All maps in this composite are continuous.  By induction, $\theta^r(1/f_t)$ is a finite sum of terms $c_{r,j}j!t^jg^j/f_t^{j+1}$ with $c_{r,j}\in\Z$; the numerator of the $j$th term is supported in $j\Delta$.  Since $0\in\Delta^\circ$, one has $j\Delta\subset(j+1)\Delta^\circ$; hence $\left.\theta^{r_i}(1/f_t)\right|_{t=t_0}\in M_{\mathrm{alg}}$.  Its image under $\mathcal L_f$ is the class of $\left.\theta^{r_i}(1/f_t)\right|_{t=t_0}\Omega_T$.  Therefore \eqref{eq:general-proper-residues} gives $\rho\!\left(\left.\theta^{r_i}(1/f_t)\right|_{t=t_0}\right)=h_i$, so $\rho$ is surjective.  If $\Lambda\subset H$ is a $W$-lattice, then $\Lambda$ is an open neighbourhood of zero.  Continuity and the $p$-adic topology on $M$ give an integer $c\geq0$ such that $\rho(p^cM)\subset\Lambda$, or equivalently $\rho(M)\subset p^{-c}\Lambda$.

It remains to prove \eqref{eq:general-Cartier-Frobenius}.  Let
\[A_0=W[x_1,y_1,\ldots,x_n,y_n,z]/ (x_iy_i-1\ (1\leq i\leq n),\ zf-1), \quad A^\dagger=A_0^\dagger\otimes_WK.\]
Let $F$ act trivially on $K$ and set $F(x_i)=x_i^p$ and $F(y_i)=y_i^p$.  Then $F(f)\equiv f^p\pmod p$, and setting $F(z)=F(f)^{-1}$ defines an endomorphism of $A^\dagger$.  Indeed, if $b=(F(f)-f^p)/p$, the expansions
\[F(f)^{-1}=f^{-p}\sum_{j\geq0}(-pb/f^p)^j, \quad f^{-p}=F(f)^{-1}\sum_{j\geq0}(pb/F(f))^j\]
converge on a sufficiently small strict neighbourhood.  They show that the dagger localizations at $f$ and at $F(f)$ coincide.  The standard decomposition of the torus algebra over its $p$-power subalgebra therefore remains valid after dagger localization.  If $B=F(A^\dagger)$, it gives
\[A^\dagger= \bigoplus_{\mathbf a\in\{0,\ldots,p-1\}^n} B\mathbf x^{\mathbf a}.\]
Thus $B\subset A^\dagger$ is finite free of degree $p^n$.  Since $p$ and the $x_i$ are units, the relative differentials vanish, so the extension is finite \'{e}tale.  Let $\operatorname{pr}_0:A^\dagger\to B$ be projection to the summand indexed by $\mathbf a=\mathbf0$, and let $P_0=F^{-1}\circ\operatorname{pr}_0$, where $F^{-1}:B\to A^\dagger$ is the inverse of $F:A^\dagger\xrightarrow{\sim}B$.  Thus $P_0(F(a)c)=aP_0(c)$. For $\mathbf a\neq\mathbf0$, multiplication by $\mathbf x^{\mathbf a}$ has zero diagonal in this $B$-basis.  Hence
\[\operatorname{Tr}_{A^\dagger/B}(c) =p^nF(P_0(c)).\]
Let $\operatorname{Tr}_F:\Omega^\bullet(A^\dagger/K) \to\Omega^\bullet(B/K)$ denote the trace on differential forms and let $\Psi=F^{-1}\circ\operatorname{Tr}_F$.  Let $F_U^*$ denote the induced pullback on rigid cohomology under the Monsky--Washnitzer comparison.  Since $F^*\Omega_T=p^n\Omega_T$, the definition of the trace on top forms gives $\operatorname{Tr}_F(c\Omega_T)=F(P_0(c))F^*\Omega_T$. Hence $\Psi(c\Omega_T)=P_0(c)\Omega_T$. We use the same notation for the induced map on rigid cohomology. The trace identity for a finite \'etale map of degree $p^n$ gives
\begin{equation}\label{eq:general-trace}
 \Psi\circ F_U^*=p^n\operatorname{id}.
\end{equation}
On Laurent polynomials, $P_0$ agrees with the coefficient-extraction operator $\cC_p$.

We first compare this coefficient extraction with the Cartier operator on the ambient algebraic module $\Omega_f[1/p]$.  Take a normalized generator $\eta=(m-1)!x^u/f^m$, let $\ell=\lceil m/p\rceil$, $\nu=p\ell-m$, and use $b=(F(f)-f^p)/p$ as above.  The expansion
\[\frac1{f^m}=\frac{f^\nu}{F(f)^\ell} \sum_{r\geq0}p^r\binom{\ell+r-1}{r}\frac{b^r}{F(f)^r}\]
converges in $A^\dagger$.  Since $P_0(F(a)c)=aP_0(c)$, it gives
\[P_0(\eta)= \sum_{r\geq0}\frac{p^r}{r!}\frac{(m-1)!}{(\ell-1)!} (\ell+r-1)!\frac{\cC_p(x^uf^\nu b^r)}{f^{\ell+r}}.\]
Here $\cC_p$ in the numerator means coefficient extraction on a Laurent polynomial. This is the Cartier formula in the proof of \cite[Proposition~3.3]{BeukersVlasenkoI}. The convergence estimate above permits termwise application of $\widetilde{\mathcal L}_f$ to the Cartier series.  Together with \eqref{eq:general-trace}, this gives
\begin{equation}\label{eq:general-Laplace-Cartier}
 \mathcal L_f(\cC_p\eta)=\Psi\mathcal L_f(\eta)
 \quad\text{in }H^n_{\rig}(U/K).
\end{equation}
By linearity the identity holds on $\Omega_f[1/p]$.  Proposition~2.7 of \cite{BeukersVlasenkoIII} identifies $C$ with the restriction of $\cC_p$ to $M$.  Hence \eqref{eq:general-Laplace-Cartier} gives $\mathcal L_f(C\eta)=\Psi\mathcal L_f(\eta)$ for every $\eta\in M_{\mathrm{alg}}[1/p]$.

Equation \eqref{eq:general-trace} gives $\Psi=p^n(F_U^*)^{-1}$ on rigid cohomology.  Poincar\'e residue is Frobenius equivariant with target $H^d_{\rig}(Y/K)(-1)$.  Since our convention makes Frobenius on $K(-1)$ equal to $p$, this says
\begin{equation*}
 \Res_{\rig}\circ F_U^*
 =p\varphi^\circ\circ\Res_{\rig}.
\end{equation*}
It follows that
\[\Res_{\rig}\circ\Psi =p^n\Res_{\rig}\circ(F_U^*)^{-1} =p^{n-1}(\varphi^\circ)^{-1}\circ\Res_{\rig}.\]
The reduction of the $W$-defined $G$-action commutes with absolute Frobenius, so $e_G$ commutes with $\varphi^\circ$.  The operator $\Pi_d$ commutes with Frobenius by construction, and $j_H$ is Frobenius equivariant.  Therefore, for $\eta\in M_{\mathrm{alg}}[1/p]$,
\begin{align*}
 \rho(C\eta)
 &=j_H^{-1}\Pi_de_G\Res_{\rig}\Psi\mathcal L_f(\eta)\\
 &=p^{n-1}j_H^{-1}\Pi_de_G(\varphi^\circ)^{-1}
       \Res_{\rig}\mathcal L_f(\eta)\\
 &=p^{n-1}\varphi^{-1}\rho(\eta)
  =p^d\varphi^{-1}\rho(\eta).
\end{align*}
Since $M_{\mathrm{alg}}$ is dense in $M$ and both $C$ and $\rho$ are continuous, the identity extends to $M[1/p]$.
\end{proof}

\section{Geometric canonical lifting and local dynamics}

Over a finite field $k_0=\F_q$, Brantner and Taelman study the full deformation functor of a smooth proper geometrically irreducible Calabi--Yau $d$-fold $X/k_0$ with trivial canonical bundle. They assume Bloch--Kato $2$-ordinarity and torsion-freeness of $H^d_{\et}(X_{\overline{k_0}},\Z_p)$. Under these assumptions, the deformation functor is canonically equivalent to a formal group of multiplicative type. Under this equivalence, the deformation associated with the unit section is the canonical lift \cite[Theorem~B (Theorem~8.28) and Definition~8.31]{BrantnerTaelman2024}.

We first isolate the local dynamical statement for a split formal torus. Let $q=p^f$, let $r\geq1$, let $K_0=W(\F_q)[1/p]$, and let $K/K_0$ be a finite extension with ring of integers $\mathcal O_K$ and maximal ideal $\mathfrak m_K$. Let $[q]$ be the multiplication-by-$q$ endomorphism of $\widehat{\Gm}^{\,r}$. In multiplicative coordinates it is given by $[q](Q_1,\ldots,Q_r)=(Q_1^q,\ldots,Q_r^q)$. We prove the following proposition.

\begin{proposition}
\label{prop:ST-contraction}
The map $[q]$ is a strict contraction on $\widehat{\Gm}^{\,r}(\mathcal O_K)=(1+\mathfrak m_K)^r$.  Every orbit converges to the unit point, which is the unique periodic point.
\end{proposition}

\begin{proof}
For $Q,R\in1+\mathfrak m_K$, one has $Q^q-R^q=(Q-R)\sum_{j=0}^{q-1}Q^{q-1-j}R^j$. The sum belongs to $\mathfrak m_K$, because it is congruent to $q$ modulo $\mathfrak m_K$.  If $\pi_K$ is a uniformizer, then $|Q^q-R^q|\leq|\pi_K|\,|Q-R|$.  Applying this estimate coordinatewise shows that $[q]$ is a strict contraction for the sup metric.  The space $(1+\mathfrak m_K)^r$ is complete, so every orbit converges to the unit point, which is fixed by $[q]$.  A periodic orbit that converges to the unit point is constant.
\end{proof}

This applies directly to the deformation group of Brantner and Taelman when that group is split.

\begin{corollary}
\label{cor:BT-contraction}
Let $X_0/\F_q$ satisfy the hypotheses of Brantner and Taelman stated above. Suppose that its deformation formal group is split, so that $\operatorname{Def}(X_0)\simeq\widehat{\Gm}^{\,r}$ over $W(\F_q)$.  Let $K_0=W(\F_q)[1/p]$.  For every finite extension $K/K_0$, with ring of integers $\mathcal O_K$, multiplication by $q$ is a strict contraction on $\operatorname{Def}(X_0)(\mathcal O_K)$.  Its unique periodic point is the unit section, which corresponds to the Brantner--Taelman canonical formal lift of $X_0$.
\end{corollary}

\begin{proof}
Under the chosen isomorphism, multiplication by $q$ is the map in \cref{prop:ST-contraction}.  The description of the unit section follows from \cite[Theorem~B and Definition~8.31]{BrantnerTaelman2024}.
\end{proof}

Under the hypotheses of \cref{cor:BT-contraction}, let $\mathscr X\to S$ be a family over $W(\F_q)$ whose special fiber at $\bar t\in S(\F_q)$ is $X_0$, and let $\widehat S_{\bar t}$ be the formal completion of $S$ at $\bar t$. Assume that the excellent Frobenius is defined on this residue disk. Write $q=p^f$ and $\sigma_{\mathrm{exc}}^f$ for its $f$-fold iterate. Let $\iota:\widehat S_{\bar t}\to\operatorname{Def}(X_0)$ be the classifying morphism of this family. Suppose that $\iota$ identifies the parameter disk with a formal subspace and that the canonical formal lift belongs to the family.  Consider the equivariance relation
\begin{equation}\label{eq:excellent-ST-equivariance}
 \iota\circ\sigma_{\mathrm{exc}}^f=[q]\circ\iota.
\end{equation}
If \eqref{eq:excellent-ST-equivariance} holds, then \cref{cor:BT-contraction} shows that the parameter of the canonical formal lift is the unique periodic point of $\sigma_{\mathrm{exc}}^f$ on the residue disk. When $f=1$, this equivariance implies the fixed point comparison in \cref{conj:full-canonical-comparison}. For general $f$, it gives the corresponding statement for the return map $\sigma_{\mathrm{exc}}^f$.

\section{The \texorpdfstring{$A_4$}{A4} family at
\texorpdfstring{$t_*=-1/7$}{t*=-1/7}}
\label{sec:family}
From this section onward, we assume Dummigan's Conjecture~1.1 whenever we use the modular decomposition of the $A_4$ fiber. 

The goal of this section is to apply the general criterion to the $A_4$ family. We first prove the geometric and Galois theoretic result and then construct the bounded Dwork--crystalline comparison.

\subsection{Geometry and the rank four variation}
Let $y_0,\dots,y_4$ be homogeneous torus coordinates with $x_i=y_i/y_0$ for $1\leq i\leq4$. Then \eqref{eq:A4-g} becomes $g=(y_0+\cdots+y_4)(y_0^{-1}+\cdots+y_4^{-1})$. Let $\Sigma=\{0,1,1/9,1/25,\infty\}$. For $t\notin\Sigma$, the affine hypersurface $Y_t=\{f_t=0\}\subset\Gm^4$ belongs to the Hulek--Verrill family associated with the root lattice $A_4$.  The toric closure has thirty persistent boundary nodes.  We use the smooth toric compactification compatible with the symmetry from \cite[\S2.4]{CandelasDeLaOssaKuuselaMcGovern2023}.

The ambient toric variety is the five-dimensional permutohedral variety $V$ of \cref{app:good-pair}.  The line bundles $\mathcal L_1,\mathcal L_2$ correspond to opposite simplices and form the Batyrev--Borisov nef partition used in \cite[\S2.4]{CandelasDeLaOssaKuuselaMcGovern2023}. The compactified cover $\widetilde Y_t$ is their smooth complete intersection. On the dense torus, eliminating the auxiliary coordinate in this complete intersection gives the Laurent hypersurface $Y_t$; the equations are written explicitly in the proof of \cref{prop:birational-H3}. Let $G\simeq C_{10}$ be the group generated by the five-cycle and simultaneous inversion.  For $t\notin\Sigma$, this action is free \cite[\S2]{CandelasDeLaOssaElmiVanStraten2020}, and we write $X_t=\widetilde Y_t/G$ for the smooth quotient.

We first compute the deformation spaces of the cover and its quotient. The same calculation gives the Hodge numbers and Kodaira--Spencer statement used in \cref{prop:A4-W2-cohomology}.

\begin{proposition}
\label{prop:permutohedral-deformations}
Let $k$ be a field of characteristic prime to $10$, and let $\widetilde Y\subset V_k$ be a smooth complete intersection cut out by two sections of $\mathcal L_1$ and $\mathcal L_2$ in which all six simplex monomials have nonzero coefficients.  Then $H^0(\widetilde Y,T_{\widetilde Y})=0$ and $h^1(\widetilde Y,T_{\widetilde Y})=5$. After fixing the first section and normalizing one coefficient of the second, variation of its five remaining coefficients induces an isomorphism onto $H^1(\widetilde Y,T_{\widetilde Y})$.  If moreover $\widetilde Y$ lies in the one-parameter $G$-equivariant subfamily obtained by setting the five free coefficients equal, and the order-ten action is free, then $\dim H^1(\widetilde Y,T_{\widetilde Y})^G=1$, and the Kodaira--Spencer class of the diagonal parameter generates this space.  Consequently the smooth quotient $X=\widetilde Y/G$ has $h^{3,0}=h^{2,1}=h^{1,2}=h^{0,3}=1$.
\end{proposition}

\begin{proof}
The two-section Koszul resolution and \cref{prop:permutohedral-vanishing} give $H^1(\widetilde Y,\mathcal O_{\widetilde Y}) =H^2(\widetilde Y,\mathcal O_{\widetilde Y})=0$. For every toric boundary divisor $E_S\subset V_k$, the same arguments give
\begin{equation}\label{eq:perm-Y-boundary-cohomology}
 H^0(\widetilde Y,\mathcal O_{\widetilde Y}(E_S))=k,
 \quad
 H^1(\widetilde Y,\mathcal O_{\widetilde Y}(E_S))=0.
\end{equation}
The same resolution, twisted by $\mathcal L_i$, shows that the restriction map $H^0(V,\mathcal L_i)\to H^0(\widetilde Y,\mathcal L_i|_{\widetilde Y})$ is surjective.  Its kernel is spanned by the defining section.  Since $H^0(V,\mathcal L_i)$ has a basis of six monomials, \(h^0(\widetilde Y,\mathcal L_i|_{\widetilde Y})=5.\)

The toric Euler sequence restricts to
\[0\longrightarrow\mathcal O_{\widetilde Y}^{57} \longrightarrow\bigoplus_S\mathcal O_{\widetilde Y}(E_S) \longrightarrow T_V|_{\widetilde Y}\longrightarrow0.\]
Equation \eqref{eq:perm-Y-boundary-cohomology} gives
\begin{equation*}
 h^0(T_V|_{\widetilde Y})=5,
 \quad H^1(T_V|_{\widetilde Y})=0.
\end{equation*}
The 5-dimensional torus Lie algebra injects into this space: a nonzero translation-invariant vector field does not vanish on the dense torus. Since both spaces have dimension five, this map is an isomorphism. The logarithmic derivatives of the two simplex equations identify its image with the anti-diagonal subspace of $H^0(\mathcal L_1|_{\widetilde Y})\oplus H^0(\mathcal L_2|_{\widetilde Y})$. The normal sequence now gives
\begin{equation}\label{eq:perm-Y-normal-quotient}
 H^0(T_{\widetilde Y})=0,
 \quad
 H^1(T_{\widetilde Y})\simeq
 \frac{k^5\oplus k^5}{\{(z,-z):z\in k^5\}}.
\end{equation}
Here the two bases have been rescaled by units.

Let $C_{\mathrm{cyc}}$ be the five-cycle on $k^5$. In the bases of \eqref{eq:perm-Y-normal-quotient}, the order-ten generator acts by $(x,y)\mapsto(C_{\mathrm{cyc}}y,C_{\mathrm{cyc}}x)$ and $z\mapsto-C_{\mathrm{cyc}}z$ on the numerator and denominator, respectively.  The fixed space in the numerator has dimension one.  The denominator has no fixed vectors.  Since $10$ is invertible in $k$, invariants are exact, and the invariant quotient is one-dimensional.  The diagonal coefficient class is invariant modulo $\{(z,-z)\}$ and is not zero there.

For a free action, finite \'etale descent identifies this invariant space with $H^1(X,T_X)$.  The holomorphic volume form is invariant; contraction with it identifies $T_X$ with $\Omega_X^2$.  Adjunction and Serre duality then give the four displayed Hodge numbers.
\end{proof}

For $t\notin\Sigma$, the hypotheses of the proposition hold. The $G$-invariant part of the middle de Rham variation of $\widetilde Y_t$ has rank four. It is identified with the middle de Rham variation of $X_t$. The restriction of its holomorphic form to the torus is
\begin{equation}\label{eq:omega-hol}
 \omega_t=\Res_{Y_t}\left(\frac{\Omega_T}{f_t}\right),
 \quad
 \Omega_T=d\log x_1\wedge\cdots\wedge d\log x_4.
\end{equation}
The four classes
\[\omega_t,\quad \nabla_\theta\omega_t,\quad \nabla_\theta^2\omega_t,\quad\nabla_\theta^3\omega_t, \quad \theta=t\frac d{dt},\]
are governed by the minimal fourth order AESZ34 Picard--Fuchs operator \cite[\S\S3.1--3.2]{CandelasDeLaOssaElmiVanStraten2020}. Its leading coefficient is proportional to $(t-1)(9t-1)(25t-1)$, so $t_*=-1/7$ is a regular point. In \cref{prop:interior-residue}, we show that these classes have $G$-invariant lifts to proper de Rham cohomology and that the four lifts are linearly independent at $t_*$.

\subsection{Dummigan's conjecture and polarized splitting}

Dummigan works with the quotient of the Hulek--Verrill small-resolution model.  He also discusses the toric compactification used here and notes that the two covering models have the same point counts at primes of good reduction \cite[Remark~4.10]{Dummigan2026}.  We compare the two models in order to apply his conjectural modular decomposition to the toric model.

For $t\in\Q\setminus\Sigma$, let $\widetilde Y_t^{\mathrm{sm}}$ denote the smooth projective Hulek--Verrill compactification of the covering hypersurface obtained by small resolution, and let \(X_t^{\mathrm{sm}}=\widetilde Y_t^{\mathrm{sm}}/G.\) Thus $\widetilde Y_t$ and $\widetilde Y_t^{\mathrm{sm}}$ are the two compactifications of the cover, while $X_t$ and $X_t^{\mathrm{sm}}$ are their quotients.

The following proposition compares the semisimplified $\ell$-adic cohomology of both the covers and their quotients.

\begin{proposition}
\label{prop:birational-H3}
Let $t\in\Q\setminus\Sigma$.  For every prime $\ell$, there are $G_\Q$-equivariant isomorphisms
\[H^3_{\et}(\widetilde Y_{t,\overline\Q},\Q_\ell)^{\mathrm{ss}} \simeq H^3_{\et}(\widetilde Y^{\mathrm{sm}}_{t,\overline\Q}, \Q_\ell)^{\mathrm{ss}}\]
and
\[H^3_{\et}(X_{t,\overline\Q},\Q_\ell)^{\mathrm{ss}} \simeq H^3_{\et}(X^{\mathrm{sm}}_{t,\overline\Q}, \Q_\ell)^{\mathrm{ss}}.\]
\end{proposition}

\begin{proof}
On $T^5=(\Gm)^6/\Gm$, the open complete intersection is given by
\begin{equation}\label{eq:BB-open-model}
 P_1=\sum_{\mu=0}^5X_\mu=0,\quad
 P_{2,t}=X_0^{-1}+t\sum_{i=1}^5X_i^{-1}=0.
\end{equation}
Eliminating $X_0=-\sum_{i=1}^5X_i$ gives
\[t\left(\sum_{i=1}^5X_i\right) \left(\sum_{i=1}^5X_i^{-1}\right)=1,\]
which defines the Hulek--Verrill hypersurface in the torus.

Let $c$ fix $X_0$ and cyclically permute $X_1,\ldots,X_5$.  Since $t\neq0$, the map
\[\iota_t(X_0,X_1,\ldots,X_5) = \left(X_0^{-1},tX_1^{-1},\ldots,tX_5^{-1}\right)\]
is a well-defined involution of $T^5$. It satisfies $P_1\circ\iota_t=P_{2,t}$ and $P_{2,t}\circ\iota_t=P_1$, and it commutes with $c$.  After eliminating $X_0$, the maps $c$ and $\iota_t$ become the five-cycle and simultaneous inversion on $(X_1:\cdots:X_5)$.  Hence the birational map between $\widetilde Y_t$ and $\widetilde Y_t^{\mathrm{sm}}$ is $G$-equivariant. Passing to invariant function fields gives a birational map between $X_t$ and $X_t^{\mathrm{sm}}$.

For $t\notin\Sigma$, the four varieties are smooth and projective with trivial canonical bundle.  For the quotients, this also follows from the freeness of the $G$-action and the invariance of the holomorphic volume form.  Apply the $p$-adic change-of-variables argument in \cite[proof of Theorem~1.1 and Remark~4.5]{Batyrev1999} separately to the two birational pairs.  After spreading out over a localization of $\Z$, the two members of each pair have the same Weil zeta function at every prime outside a finite set.  Purity separates the factors by cohomological degree, so their characteristic polynomials of Frobenius on $H^3$ agree at all such primes.  Chebotarev density and Brauer--Nesbitt give the two isomorphisms after semisimplification.
\end{proof}

Let $V_\ell=H^3_{\et}(X_{t_*,\overline\Q},\Q_\ell)$. Let $f_{14,4}$ and $f_{14,2}$ denote the rational newforms labelled $14.4.a.a$ and $14.2.a.a$, and let $A_\ell=\rho_{f_{14,4},\ell}$ and $B_\ell=\rho_{f_{14,2},\ell}(-1)$. By \cref{prop:birational-H3}, Dummigan's Conjecture~1.1 predicts that, for every prime $\ell$,
\begin{equation}\label{eq:dummigan-ss}
 V_\ell^{\mathrm{ss}}\simeq A_\ell\oplus B_\ell.
\end{equation}
He proves the predicted local Frobenius factorizations at $p=3,11,13,17,19,29,31$, and $113$, and he determines the predicted residual Jordan--H\"older factors at $\ell=5$ \cite[Proposition~4.6 and Propositions~5.2--5.3]{Dummigan2026}.

We now use the polarization to lift \eqref{eq:dummigan-ss} from a semisimplification to an actual direct sum.

\begin{proposition}
\label{prop:actual-split}
For every prime $\ell$ one has a $G_\Q$-equivariant decomposition $V_\ell\simeq A_\ell\oplus B_\ell$.
\end{proposition}

\begin{proof}
Poincar\'e duality gives a perfect alternating pairing $\langle\ ,\ \rangle:V_\ell\times V_\ell\to\Q_\ell(-3)$. The two modular representations are absolutely irreducible and nonisomorphic; their Hodge--Tate weights are respectively $\{0,3\}$ and $\{1,2\}$.  Both have determinant $\Q_\ell(-3)$, and hence $A_\ell^\vee(-3)\simeq A_\ell$ and $B_\ell^\vee(-3)\simeq B_\ell$.

Choose a simple subrepresentation $S\subset V_\ell$.  By \eqref{eq:dummigan-ss}, $S$ is isomorphic to $A_\ell$ or $B_\ell$.  If the pairing restricted to $S$ is nonzero, the induced map $S\to S^\vee(-3)$ is an isomorphism by Schur's lemma. Thus $S$ is nondegenerate and $V_\ell=S\oplus S^\perp$.

If instead the restriction is zero, then the two-dimensional $S$ is Lagrangian, so $S=S^\perp$.  The perfect pairing then induces $V_\ell/S\simeq S^\vee(-3)\simeq S$. Both Jordan--H\"older factors of $V_\ell$ would be isomorphic to $S$, contrary to \eqref{eq:dummigan-ss}.  The isotropic case is impossible, proving the claim.
\end{proof}

\subsection{The Calabi--Yau crystal and its vertex coefficients}
\label{sec:A4-Dwork}

We now describe the Dwork crystal and the vertex determinants for the same family.

For this polynomial, the module $CY(g)$ is defined using the $\mathcal G$-invariant admissible numerators of \cref{sec:vertex-determinants}, where the full monomial symmetry group is $\mathcal G\simeq S_5\times C_2$. The geometric quotient uses $G\simeq C_{10}$. The common vertex coefficient is $1$, the support generates $\Z^4$, and $|\mathcal G|=240$.  Thus the group order and the index of the support lattice required by \cite[Theorem~7.3]{BeukersVlasenkoIII} are units for $p>5$.

Fix a vertex $\mathbf b$ of the Newton polytope of $g$ and let $\mathbf v=-\mathbf b$.  The coefficient of $\mathbf x^{Q\mathbf v}$ in $\omega_0$ is $a_Q(t)$, while the corresponding coefficient in $\omega_1$ is $\theta a_Q(t)$.  For $r=1,2$ and $j\geq 0$ set $A_j^{(r)}=a_{rp^j}(t_*)$ and $B_j^{(r)}=(\theta a_{rp^j})(t_*)$, and
\begin{equation*}
 M_j=
 \begin{pmatrix}
 A_j^{(1)}&A_j^{(2)}\\
 B_j^{(1)}&B_j^{(2)}
 \end{pmatrix}.
\end{equation*}

For $p>5$, $p\neq7$, define an auxiliary coefficient Frobenius on the coefficient ring of the residue disk of $t_*$ by $\sigma_{t_*}(t)=t_*^{1-p}t^p$. Fermat's congruence shows that this is a Frobenius lift, and it fixes $t_*$ exactly. Apply \cite[Proposition~5.11]{BeukersVlasenkoIII} at level two with $\mathbf m=\mathbf v,2\mathbf v$. Since $\sigma_{t_*}(t_*)=t_*$ and Frobenius acts trivially on $W(\F_p)=\Z_p$, the resulting congruence is
\begin{equation}\label{eq:matrix-recurrence}
 M_j\equiv\Lambda M_{j-1}\pmod{p^{2j}}.
\end{equation}
Here $\Lambda=\Lambda(t_*,t_*)$ is the level two Cartier matrix, and we abbreviate $\lambda_i(t_*,t_*)$ to $\lambda_i(t_*)$. The first row of $\Lambda$ is $(\lambda_0,\lambda_1)$. The finite vertex formula in \cref{app:vertex} gives $a_1(t_*)=7$, $a_2(t_*)=-42$, $\theta a_1(t_*)=-7$, and $\theta a_2(t_*)=91$. Consequently
\begin{equation}\label{eq:M0}
 M_0=
 \begin{pmatrix}7&-42\\-7&91\end{pmatrix},
 \quad \det M_0=343=7^3.
\end{equation}

The determinant recurrence gives the following valuation.

\begin{lemma}\label{lem:det-valuation}
Assume $p>5$, $p\neq7$, and that the first two Hasse--Witt conditions hold. Then $v_p(\det M_j)=j$ for $j\geq0$.
\end{lemma}

\begin{proof}
The normalized second Hasse--Witt determinant is a unit, so $v_p(\det\Lambda)=1$.  Taking determinants in \eqref{eq:matrix-recurrence} gives $\det M_j\equiv(\det\Lambda)(\det M_{j-1})\pmod{p^{2j}}$. The case $j=0$ follows from \eqref{eq:M0}.  If $v_p(\det M_{j-1})=j-1$, the product $(\det\Lambda)(\det M_{j-1})$ has valuation $j$, while the remaining terms have valuation at least $2j$. Induction completes the proof.
\end{proof}

We now relate the two-ray determinants to the off-diagonal Cartier coefficient. Together with the preceding lemma, this gives a converse to the supercongruence conclusion of \cref{thm:main}.

\begin{proposition}
\label{prop:off-diagonal-congruence}
Assume $p>5$, $p\neq7$, and the first two Hasse--Witt conditions at $t_*$.  Then
\begin{equation}\label{eq:Ds-lambda}
 D_s(t_*,t_*)\equiv-\lambda_1(t_*)\det M_{s-1}\pmod{p^{2s}}.
\end{equation}
In particular,
\[D_s(t_*,t_*)\equiv0\pmod{p^{2s}}\quad\text{for all }s\geq1 \quad\Longleftrightarrow\quad \lambda_1(t_*)=0.\]
\end{proposition}

\begin{proof}
The first row of \eqref{eq:matrix-recurrence} says $A_s^{(r)}\equiv\lambda_0 A_{s-1}^{(r)} +\lambda_1B_{s-1}^{(r)}\pmod{p^{2s}}$ for $r=1,2$. Substitution into \eqref{eq:Ds-intro} cancels the $\lambda_0$ terms and gives \eqref{eq:Ds-lambda}.  If all determinant congruences hold, then \cref{lem:det-valuation} implies $v_p(\lambda_1)+(s-1)\geq2s$ for every $s$, whence $\lambda_1=0$.  The converse follows immediately from \eqref{eq:Ds-lambda}.
\end{proof}

\subsection{The bounded residue realization}
\label{sec:bounded-residue}

We next realize the Dwork crystal in proper crystalline cohomology and construct the comparison required by the general criterion.

Fix $p>5$, $p\neq7$, and let $W=W(\Fp)=\Z_p$ and $K=W[1/p]$. By \cref{lem:good-pair} below, the symmetric compactification extends to a smooth proper $W$-model. Let $\widetilde Y_{t_*}$ denote the smooth symmetric toric compactification described in \cref{sec:family}, and let $e_G=|G|^{-1}\sum_{\gamma\in G}\gamma$.  Let $H=e_GH^3_{\cris}(\widetilde Y_{t_*,\Fp}/W)[1/p]$. Because $p\nmid10$, this is a direct summand.  The $G$-action on the characteristic zero fiber at $t_*$ is free.  Finite \'etale descent on the generic fiber gives
\[H^3_{\et}(X_{t_*,\overline\Q_p},\Q_p) \simeq e_GH^3_{\et}(\widetilde Y_{t_*,\overline\Q_p},\Q_p).\]
The representation on the right is a direct summand of the crystalline representation furnished by the smooth proper toric model.  Applying $D_{\cris}$ and smooth proper comparison \cite{Tsuji1999} therefore identifies $H$ with $H_p$. The target $H$ of the residue comparison is this rank four invariant part.

To verify the two geometric hypotheses of \cref{thm:general-bounded-residue}, we use the symmetry-equivariant Batyrev--Borisov model of \cite[\S2.4]{CandelasDeLaOssaKuuselaMcGovern2023}.  We regard the open complete intersection \eqref{eq:BB-open-model} over $W$. The invariant regular fan of \cite[\S2.4]{CandelasDeLaOssaKuuselaMcGovern2023} gives a smooth toric ambient scheme.  Let $\overline{\mathscr T}$ be the closure of $P_1=0$, and denote the closure of $P_1=P_{2,t}=0$ by $\widetilde{\mathscr Y}_t$.  We write $\widetilde Y_t$ for its generic fiber, in accordance with the notation above. The map $(X_1:\cdots:X_5)\mapsto (-X_1-\cdots-X_5:X_1:\cdots:X_5)$ embeds $\mathscr T=(\Gm)^5/\Gm\simeq\Gm^4$ as an open subscheme of $\overline{\mathscr T}$; let $B=\overline{\mathscr T}\setminus\mathscr T$. The group $G$ acts on this model, and its generic-fiber quotient is $X_t$.

We first verify the simple normal crossings hypotheses required by \cref{thm:general-bounded-residue}.

\begin{lemma}
\label{lem:good-pair}
At $t_*=-1/7$, the preceding model defines over $R_0=\Z[1/14]$ a smooth proper scheme $\overline{\mathscr T}$ with relative simple normal crossings boundary $B=\overline{\mathscr T}\setminus\mathscr T$.  The scheme $\widetilde{\mathscr Y}=\widetilde{\mathscr Y}_{t_*}$ is smooth and proper, and it meets $B$ transversely.  Thus both $(\widetilde{\mathscr Y},D)$, where $D=\widetilde{\mathscr Y}\cap B$, and $(\overline{\mathscr T},B+\widetilde{\mathscr Y})$ are relative simple normal crossings pairs.  The first pair is $G$-equivariant.  In particular, they give the required pairs over $W$ for every $p>5$, $p\neq7$.
\end{lemma}

\begin{proof}
After multiplying $P_{2,t_*}$ by $-7$, the two sections have supports $A_1=\{0,e_1,\ldots,e_5\}$ and $A_2=-A_1$ and coefficient vectors $(1,\ldots,1)$ and $\boldsymbol a=(-7,1,1,1,1,1)$. Let $w=(w_0,\ldots,w_5)$ be a nonconstant toric weight. We consider such weights modulo addition of a constant. Let $I=\operatorname*{argmin}_\mu w_\mu$ and $J=\operatorname*{argmax}_\mu w_\mu$. Then $I\cap J=\varnothing$; the initial form of $P_1$ is supported on $I$ and that of $P_{2,t_*}$ on $J$.  If both initial forms vanish, choose $i\in I$ and $j\in J$.  In logarithmic coordinates the corresponding $2\times2$ Jacobian minor is $\det\bigl(\begin{smallmatrix}X_i&0\\0&-a_jX_j^{-1}\end{smallmatrix}\bigr) =-a_jX_iX_j^{-1}$, a unit times a torus monomial over $R_0$.  Thus every proper initial system is smooth of codimension two.  Thus the two sections meet every boundary stratum transversely.

It remains to treat the dense torus.  Its two logarithmic Jacobian rows are $(X_\mu)_\mu$ and $(-a_\mu X_\mu^{-1})_\mu$.  Rank loss therefore gives $X_\mu^2=c a_\mu$ for a common $c\neq0$. After adjoining square roots, this would imply $\sqrt{-7}+\epsilon_1+\cdots+\epsilon_5=0$ for some $\epsilon_i\in\{\pm1\}$.  This is impossible over $R_0$: the odd integer $m=\epsilon_1+\cdots+\epsilon_5$ satisfies $|m|\leq5$, whereas $m^2+7\in\{8,16,32\}\subset R_0^\times$. Finally, the ambient fan is the $A_5$ braid fan.  By \cref{app:good-pair}, it is unimodular, complete, and projective over $\Z$. Hence the ambient toric scheme is smooth and proper over $R_0$.  The initial form of $P_1$ on every orbit is a sum of nonzero torus coordinates, and one of its logarithmic partial derivatives is a unit there.  Hence $P_1=0$ is smooth and transverse to every orbit.  The $2\times2$ minors above show that the restriction of $P_{2,t_*}$ to $P_1=0$ is also smooth and transverse to every boundary stratum.  Thus $\overline{\mathscr T}$ and $\widetilde{\mathscr Y}$ are smooth, $B$ is relative simple normal crossings, and $B+\widetilde{\mathscr Y}$ is relative simple normal crossings on $\overline{\mathscr T}$. The fan is invariant under permutations and inversion.  The cyclic permutation of $X_1,\ldots,X_5$ and the involution $X_\mu\mapsto a_\mu/X_\mu$ preserve the fan.  The involution swaps the two equations.  Hence these maps extend the action of $G$.
\end{proof}

We next compare the proper and open middle cohomology.

\begin{lemma}
\label{lem:proper-open-injective}
Let $D=\widetilde Y_{t_*}\setminus Y_{t_*}$ be the simple normal crossings boundary.  Then restriction is injective:
\[H^3(\widetilde Y_{t_*},\Q)\longrightarrow H^3(Y_{t_*},\Q).\]
For every prime $p>5$, $p\neq7$, the analogous rigid restriction map
\[H^3_{\rig}(\widetilde Y_{t_*,\F_p}/K) \longrightarrow H^3_{\rig}(Y_{t_*,\F_p}/K)\]
is also injective.
\end{lemma}

\begin{proof}
Write $D_S=\widetilde Y_{t_*}\cap E_S$ for an irreducible boundary component.  The divisor sequence
\[0\longrightarrow\mathcal O_{\widetilde Y_{t_*}}(-D_S) \longrightarrow\mathcal O_{\widetilde Y_{t_*}} \longrightarrow\mathcal O_{D_S}\longrightarrow0\]
and \eqref{eq:perm-Y-boundary-cohomology} give $H^1(D_S,\mathcal O_{D_S})=0$.  Indeed, $H^1(\widetilde Y_{t_*},\mathcal O)=0$, while Serre duality and $\omega_{\widetilde Y_{t_*}}\simeq\mathcal O$ give
\[H^2\bigl(\widetilde Y_{t_*},\mathcal O(-D_S)\bigr)^\vee \simeq H^1\bigl(\widetilde Y_{t_*},\mathcal O(D_S)\bigr)=0.\]
Since every $D_S$ is smooth and projective, Hodge symmetry therefore gives $H^1(D_S,\Q)=0$.

The edge sequence of Deligne's weight spectral sequence for the smooth pair $(\widetilde Y_{t_*},D)$ contains
\[\bigoplus_S H^1(D_S,\Q)(-1) \longrightarrow H^3(\widetilde Y_{t_*},\Q) \longrightarrow \operatorname{Gr}_3^W H^3(Y_{t_*},\Q) \longrightarrow0;\]
see \cite[\S3.2]{Deligne1971}.  Its left hand term is zero, so restriction is injective. The groups $H^0(D_S\cap D_T,\Q)(-2)$ in cohomology with support have weight four, so their maps to the pure weight-three group $H^3(\widetilde Y_{t_*},\Q)$ vanish.

For such a prime, \cite[Theorem~2.3]{GrosseKlonne2004} compares algebraic de Rham cohomology of the generic fiber with dagger cohomology. Proposition~3.6 and Corollary~3.8(c) there identify dagger cohomology with rigid cohomology of the special fiber and preserve restriction.  Thus one obtains comparison isomorphisms compatible with restriction:
\[\begin{CD}
 H^3_{\dR}(\widetilde Y_{t_*,K}/K)
 @>{j^*}>> H^3_{\dR}(Y_{t_*,K}/K)\\
 @V{\sim}VV @VV{\sim}V\\
 H^3_{\rig}(\widetilde Y_{t_*,\F_p}/K)
 @>{j^*}>> H^3_{\rig}(Y_{t_*,\F_p}/K).
\end{CD}\]
These comparison maps are induced by the relative de Rham complexes. They preserve the differential form representatives and their residues. By algebraic de Rham comparison \cite{Grothendieck1966}, the top map is obtained from the Betti restriction map above by extension of scalars and is injective.  Hence the bottom map is injective as well.
\end{proof}

We can now verify the cyclic spanning hypothesis in \cref{thm:general-bounded-residue}.

\begin{proposition}
\label{prop:interior-residue}
For $t\notin\Sigma$ and $0\leq j\leq3$, the cyclic residue $\Res_{Y_t}\theta^j(\Omega_T/f_t)$ has a $G$-invariant lift to $H^3_{\dR}(\widetilde Y_t)$. At $t=t_*=-1/7$, the four lifted classes form a basis of $H^3_{\dR}(\widetilde Y_{t_*}/\Q)^G$. For every prime $p>5$, $p\neq7$, the rigid specializations of these proper classes form a basis of $H$.
\end{proposition}

\begin{proof}
On the dense torus of the permutohedral fivefold, we normalize $X_5=1$ and write $x_0=X_0$, $s=1+x_1+\cdots+x_4$, and $r=1+x_1^{-1}+\cdots+x_4^{-1}$. Then $P_1=x_0+s$, $P_2=x_0^{-1}+tr$, and $f_t=1-tsr$. Let $\Omega_5=d\log x_0\wedge\Omega_T$. The one-variable residue at $x_0=-s$ gives $\Res_{P_1=0}(\Omega_5/P_1)=-\Omega_T/s$. Since $P_2=-f_t/s$ on $P_1=0$, for every Laurent polynomial $A$ in $x_1^{\pm1},\ldots,x_4^{\pm1}$ and every $m\geq1$ one has
\begin{equation*}
 \Res_{P_1=0}\left(
   \frac{x_0^{1-m}A\Omega_5}{P_1P_2^m}
 \right)=\frac{A\Omega_T}{f_t^m}.
\end{equation*}
Indeed, the scalar obtained after substitution is $-s^{-1}(-s)^{1-m}(-f_t/s)^{-m}=f_t^{-m}$.

Let $D_i$ be the zero divisor of the defining section $P_i$, so that $\mathcal L_i=\mathcal O(D_i)$ on the toric fivefold $V$.  The nef partition satisfies $K_V+D_1+D_2=0$.  Thus the double residue $\widetilde\omega_t=\Res_{P_1=P_2=0}(\Omega_5/(P_1P_2))$ is a regular three form on the smooth proper complete intersection.  Let $a=tr$, so that $\theta P_2=a$.  For $j\geq1$,
\[\theta^j(P_2^{-1})= \sum_{k=1}^j(-1)^k k!S(j,k)\frac{a^k}{P_2^{k+1}},\]
where $S(j,k)$ is a Stirling number of the second kind.  The numerator $a^k$ is a section of $\mathcal L_2^k$, and $K_V+D_1+(k+1)D_2=kD_2$. Consequently each summand is a global meromorphic top form with poles only along $P_1$ and $P_2$; it has no poles along the toric boundary. Iterated residue and Gauss--Manin differentiation therefore give $\Res_{Y_t}\theta^j(\Omega_T/f_t)= \left.\nabla_\theta^j\widetilde\omega_t\right|_{Y_t}$. On the original four-dimensional torus, $f_t$ and $\Omega_T$ are $G$-invariant, and the action commutes with the Gauss--Manin connection. Thus the four open classes are $G$-invariant. Averaging $\widetilde\omega_t$ over $G$ gives a $G$-invariant proper class $\widetilde\omega_t^G$ whose restriction is $\omega_t$. Its first three Gauss--Manin derivatives are also $G$-invariant classes in $H^3_{\dR}(\widetilde Y_t)$.

The AESZ34 operator is the minimal Picard--Fuchs operator and has order four \cite[\S\S3.1--3.2]{CandelasDeLaOssaElmiVanStraten2020}. After multiplying it by a suitable nonzero constant, we write it as $a_4(t)\theta^4+a_3(t)\theta^3+\cdots+a_0(t)$, where $a_4(t)=(t-1)(9t-1)(25t-1)$. On the locus where $ta_4(t)\neq0$, let
\[\mathcal W=\widetilde\omega_t^G\wedge \nabla_\theta\widetilde\omega_t^G\wedge \nabla_\theta^2\widetilde\omega_t^G\wedge \nabla_\theta^3\widetilde\omega_t^G\]
be the resulting section of the determinant line of the $G$-invariant middle de Rham variation of the proper compactification. Minimality of the Picard--Fuchs operator makes $\mathcal W$ nonzero at the generic point, and the Picard--Fuchs relation gives $\nabla_\theta\mathcal W=-a_3(t)\mathcal W/a_4(t)$. A nonzero solution cannot vanish at a regular point, so $\mathcal W(t_*)\neq0$. The invariant proper cohomology has dimension four by \cref{prop:permutohedral-deformations}. Thus these four classes form a basis of $H^3_{\dR}(\widetilde Y_{t_*}/\Q)^G$. By \cref{lem:proper-open-injective} and algebraic de Rham comparison \cite{Grothendieck1966}, their restrictions to $Y_{t_*}$ are also independent.

For every prime $p>5$, $p\neq7$, we extend scalars from $\Q$ to $K$. The comparison square in the proof of \cref{lem:proper-open-injective} carries this basis in the de Rham cohomology of the proper compactification to a basis of $e_GH^3_{\rig}(\widetilde Y_{t_*,\F_p}/K)$, which is $H$. The injectivity of the lower horizontal map shows that the four restricted rigid classes are independent in $e_GH^3_{\rig}(Y_{t_*,\Fp}/K)$.
\end{proof}

Let $Y_0=Y_{t_*,\Fp}$ and $H^\circ=e_GH^3_{\rig}(Y_0/K)$.  In this section and in the K3 application, $W_\bullet$ denotes the weight filtration on rigid cohomology. Smooth proper comparison and restriction define a Frobenius-equivariant map
\[j_p:H\xrightarrow{\sim} e_GH^3_{\rig}(\widetilde Y_{t_*,\Fp}/K) \longrightarrow H^\circ.\]

The following lemma identifies the proper summand inside the open rigid cohomology.

\begin{lemma}
\label{lem:weight-three-restriction}
The map $j_p$ identifies $H$ with the weight-three part $W_3H^\circ$.
\end{lemma}

\begin{proof}
The rigid localization sequence contains
\[H^3_{\rig}(\widetilde Y_{t_*,\Fp}/K) \xrightarrow{j^*}H^3_{\rig}(Y_0/K) \longrightarrow H^4_{D,\rig}(\widetilde Y_{t_*,\Fp}/K).\]
The first group is pure of weight $3$ by the smooth proper theorem of \cite{KatzMessing1974}.  The middle group has weights in $[3,6]$ by \cite[Theorem~2.2]{Chiarellotto1998Weights}, while the group with support has weights in $[4,6]$ by \cite[Theorem~2.3]{Chiarellotto1998Weights}, because $D$ has codimension one.  Exactness and the canonical decomposition by integral weights described in \cite[\S2.0]{Chiarellotto1998Weights} therefore give $W_3H^3_{\rig}(Y_0/K)=\operatorname{im}(j^*)$. Taking $G$-invariants is exact because $p\nmid|G|$, and the weight decomposition commutes with $e_G$.  Hence $W_3H^\circ=e_GW_3H^3_{\rig}(Y_0/K)=\operatorname{im}(e_Gj^*)$. The rigid specializations in \cref{prop:interior-residue} are independent after restriction.  Thus $j_p$ is injective.  Its image is $W_3H^\circ$, so it is an isomorphism onto that space.
\end{proof}

\begin{theorem}
\label{thm:BR}
For every prime $p>5$, $p\neq7$, there is a continuous surjection
\begin{equation}\label{eq:rho}
 \rho_K:CY(g)_{t_*}[1/p]\longrightarrow H.
\end{equation}
For every $W$-lattice $\Lambda\subset H$, there is a constant $c$, independent of Cartier iteration, such that
\begin{equation}\label{eq:rho-bounded}
 \rho_K(CY(g)_{t_*})\subset p^{-c}\Lambda.
\end{equation}
\end{theorem}

\begin{proof}
Apply \cref{thm:general-bounded-residue} with $n=4$, $d=3$, and the geometric group $G\simeq C_{10}$.  The smooth projective equivariant compactification and its relative simple normal crossings boundary are given by \cref{lem:good-pair}.  The target is the weakly admissible filtered isocrystal $H=e_GH^3_{\cris}(\widetilde Y_{t_*,\F_p}/W)[1/p]$. By \cref{lem:weight-three-restriction}, restriction identifies $H$ with $W_3H^\circ$.  By \cref{prop:interior-residue}, the cyclic residues with exponents $0,1,2,3$ lift to a basis of $H$.  Thus all the hypotheses of \cref{thm:general-bounded-residue} hold.  Its map \eqref{eq:general-rho} is the map $\rho_K$ in \eqref{eq:rho}, and the theorem gives both surjectivity and \eqref{eq:rho-bounded}.
\end{proof}

\subsection{The Frobenius normalization}

Let $\varphi$ be the usual crystalline pullback Frobenius on $H$.  We use the convention that $K(-1)$ has Frobenius $p$.

We next determine the Frobenius normalization of the comparison map.

\begin{proposition}
\label{prop:normalization}
On the completed module $CY(g)_{t_*}[1/p]$, one has
\begin{equation*}
 \rho_K\circ\cC_p=p^3\varphi^{-1}\circ\rho_K.
\end{equation*}
Consequently, a $\varphi$-slope $\lambda$ corresponds to a slope $3-\lambda$ for $\cC_p$ under $\rho_K$.
\end{proposition}

\begin{proof}
This is \eqref{eq:general-Cartier-Frobenius} in \cref{thm:general-bounded-residue} with $n=4$ and $d=3$.  The final assertion follows because multiplication by $p^3$ followed by $\varphi^{-1}$ changes a Frobenius slope $\lambda$ into the Cartier slope $3-\lambda$.
\end{proof}

\begin{remark}
If $A$ is a matrix for $\varphi$ and $P$ is the polarization matrix, then $A^{\mathsf t}PA=p^3P$, so $p^3A^{-1}=P^{-1}A^{\mathsf t}P$. The operator $p^3\varphi^{-1}$ induced by Cartier on $H$ has the same characteristic polynomial as $\varphi$.
\end{remark}

\subsection{The high slope quotient}

We now identify the level two Dwork quotient with the high slope quotient of $H$.

\begin{theorem}
\label{thm:Q2-comparison}
At a two-Dwork-ordinary prime $p>5$, $p\neq7$, one has
\begin{equation}\label{eq:Q2-high}
 Q^{(2)}_{t_*}[1/p]
 \simeq H/H_{\varphi,\leq1}=H_{\varphi,>1}.
\end{equation}
Moreover the Cartier slopes on $Q^{(2)}_{t_*}[1/p]$ are $\{0,1\}$ and the crystalline slopes on $H$ are $\{0,1,2,3\}$.
\end{theorem}

\begin{proof}
Apply \cref{lem:slope} to the map in \cref{thm:BR}, with $d=3$ and $r=2$. This gives a surjection $Q^{(2)}_{t_*}[1/p]\twoheadrightarrow H_{\varphi,>1}$. The source has dimension $2$ by \eqref{eq:BV-split}.  The polarization makes the four Newton slopes symmetric under $\lambda\leftrightarrow3-\lambda$. Each symmetric pair contributes at least one slope $>1$, so $\dim H_{\varphi,>1}\geq2$.  Surjectivity gives the reverse inequality.  Both sides have dimension $2$, proving \eqref{eq:Q2-high}.

By \cref{lem:cartier-slopes}, the Cartier slopes of $Q^{(2)}$ are $0$ and $1$.  By \cref{prop:normalization}, the corresponding $\varphi$-slopes on $H_{\varphi,>1}$ are $3$ and $2$.  Polarization gives the complementary slopes $0$ and $1$.
\end{proof}

\subsection{The filtered modular splitting}
To pass from the cover to the quotient over $\Z_p$, we prove the following specialization lemma.

\begin{lemma}
\label{lem:tame-free-specialization}
Let $R$ be a discrete valuation ring.  Let a finite group $G$ act on a smooth projective $R$-scheme $Y$.  Assume that $|G|$ is a unit in $R$.  If $G$ acts geometrically freely on the generic fiber, then it acts freely on $Y$.  The quotient $Y/G$ is smooth over $R$, and $Y\to Y/G$ is finite \'etale.
\end{lemma}

\begin{proof}
Fix $g\neq1$, let $n$ be its order, and let $Y^g\subset Y$ be its fixed-point subscheme. We apply the infinitesimal lifting criterion to $Y^g$. For a square-zero extension, the lifts of a given point to the smooth scheme $Y$ form a torsor under a module on which the cyclic group $\langle g\rangle$ acts.  The obstruction to choosing a fixed lift is a class in the first cohomology of this cyclic group. This first cohomology group vanishes because the averaging operator $n^{-1}\sum_{i=0}^{n-1}g^i$ is defined over $R$. Thus $Y^g$ is formally smooth over $R$. Since it is of finite presentation, it is smooth over $R$. Its generic fiber is empty. A nonempty smooth $R$-scheme has open image in $\operatorname{Spec}R$.  Hence $Y^g$ is empty. This holds for every $g\neq1$.  The quotient map is finite \'etale because the action is free, and smoothness descends along this map.
\end{proof}
Take $\ell=p$ in \cref{prop:actual-split}.  By \cref{lem:good-pair}, the toric cover is smooth and proper over $\mathbb Z_p$.  Its generic order-ten action is geometrically free, and $p\nmid10$; the fixed-locus argument of \cref{lem:tame-free-specialization} therefore shows that the quotient is also smooth and proper over $\mathbb Z_p$.  Hence $V_p$ is crystalline by smooth proper comparison \cite{Tsuji1999}.  A direct summand of a crystalline representation is crystalline, and $D_{\cris}$ preserves direct sums.  Let $D_f=D_{\cris}(A_p)$ and $D_g=D_{\cris}(B_p)$. These filtered $\varphi$-modules are weakly admissible \cite{ColmezFontaine2000}, and one has a direct sum
\begin{equation}\label{eq:Dcris-split}
 H=D_f\oplus D_g.
\end{equation}
The Hodge--Tate weights of $A_p$ and $B_p$ are
\begin{equation}\label{eq:HT}
 \operatorname{HT}(A_p)=\{0,3\},
 \quad
 \operatorname{HT}(B_p)=\{1,2\}.
\end{equation}

To apply the rank two criterion, we first determine the Newton slopes of the two modular factors.

\begin{lemma}
\label{lem:slope-allocation}
Under the hypotheses of \cref{thm:Q2-comparison}, $\Sl(D_f)=\{0,3\}$ and $\Sl(D_g)=\{1,2\}$.
\end{lemma}

\begin{proof}
Weak admissibility gives $t_N(D_f)=t_H(D_f)=3$ and similarly for $D_g$.  By \cref{thm:Q2-comparison}, the total slopes are $\{0,1,2,3\}$.  The only two partitions into pairs of sum $3$ are $\{0,3\}\sqcup\{1,2\}$, in the two possible orders.

Let $L_0\subset H$ be the slope zero line. The factor $D_g$ satisfies $\Fil^1D_g=D_g$. If it contained $L_0$, then the induced filtration would have $t_H(L_0)\geq1$ while $t_N(L_0)=0$, contradicting weak admissibility. Hence the slope-zero line lies in $D_f$, so $\Sl(D_f)=\{0,3\}$ and $\Sl(D_g)=\{1,2\}$.
\end{proof}

By the construction of $\rho_K$ and \cref{prop:interior-residue}, the class $\rho_K(1/f)$ is the proper holomorphic residue class and belongs to $\Fil^3H$. From \eqref{eq:HT}, $\Fil^3D_g=0$, so
\begin{equation}\label{eq:omega-in-f}
 \rho_K(1/f)\in D_f.
\end{equation}
Its image in $H/H_{\varphi,\leq1}$ is nonzero.  Indeed, if the Hodge line lay in the slope zero line of $D_f$, that line would have $t_H\geq3$ and $t_N=0$, again violating weak admissibility.  Since the image of $D_f$ in the high slope quotient is one-dimensional, the image of $\rho_K(1/f)$ spans it.

\subsection{Proof of the \texorpdfstring{$A_4$}{A4} theorem}

\begin{proof}[Proof of \cref{thm:main}]
The support of $g$ generates $\Z^4$, every vertex coefficient is $1$, and the full monomial symmetry group is $\mathcal G\simeq S_5\times C_2$. Thus the relevant lattice index is $1$ and $p\nmid240$, so the additional hypotheses of \cite[Theorem~7.3]{BeukersVlasenkoIII} hold.  The first two Hasse--Witt conditions hold by two-Dwork-ordinarity.

\Cref{thm:BR,prop:normalization} give the bounded comparison and its Frobenius compatibility.  \Cref{thm:Q2-comparison} gives the high slope quotient and the crystalline slopes $\{0,1,2,3\}$.  The direct sum \eqref{eq:Dcris-split} makes $D_f$ a rank two subisocrystal, \cref{lem:slope-allocation} gives its Newton slopes $0$ and $3$, and \eqref{eq:omega-in-f} gives $\rho_K(1/f)\in D_f$.  These verify the hypotheses of \cref{thm:rank two-criterion} with $d=3$.  That theorem gives the Cartier-stable line, $\lambda_1(t_*,t_*)=0$, the equality $\sigma_{\mathrm{exc},p}(t_*)=t_*$, and \eqref{eq:main-supercong}. The complementary factor $D_g$ has slopes $1$ and $2$.
\end{proof}

\section{The canonical lift modulo \texorpdfstring{$p^2$}{p squared}}
\label{sec:A4-W2}

We first compare the lift of the $A_4$ fiber at $t_*$ with the Achinger--Zdanowicz canonical lift modulo $p^2$. Let $p>5$, $p\neq7$, and let $k=\F_p$, $W=\Z_p$, $A=W_2(k)$, $a=\bar t_*$, and $h_p=\CT(f_a^{p-1})\in k$. By \cref{lem:good-pair}, the fiber over $a$ of the variable parameter compactified toric model is smooth.  Since smoothness is open, this model gives a smooth proper family $\mathscr Y_{S_a}\to S_a$ over an open neighborhood $S_a$ of $a$ in $\mathbb A^1_W$. The action of $C_{10}$ on the fiber over $a$ is free by \cref{lem:tame-free-specialization}. For every nontrivial $g\in C_{10}$, the fixed locus $\mathscr Y_{S_a}^g\to S_a$ is proper, so its image is closed and does not contain $a$.  After shrinking $S_a$, all these fixed loci are empty.  Thus $\mathscr X_{S_a}=\mathscr Y_{S_a}/C_{10}$ is a smooth proper family over $S_a$.  Let $\widehat S_a$ be the formal completion of $S_a$ at $a$ and let
\[\widehat{\mathscr X}=\mathscr X_{S_a}\times_{S_a}\widehat S_a \longrightarrow\widehat S_a\]
be its formal completion. Every lift of $a$ to $A$ has the form $t_\delta=t_*+p\delta$, where $\delta\in k$; the product $p\delta\in A$ is independent of the lift of $\delta$ to $A$.  Write $X_0/k$ its fiber over $a$ and $X_\delta^A/A$ for its fiber at $t_\delta$. The lift at $\delta=0$ is denoted by $X_*^A$.

\subsection{Hodge and crystalline cohomology}

We begin by establishing the integral Hodge and crystalline properties of the quotient family.

Let $\mathscr X/W$ be the fiber of $\widehat{\mathscr X}$ at $t_*$.  It is the quotient by $C_{10}$ of the smooth proper toric model at $t_*$ from \cite[\S2.4]{CandelasDeLaOssaKuuselaMcGovern2023}. The residue volume form $\omega_{t_*}$ in \eqref{eq:omega-hol} is invariant and descends to $\mathscr X$. In particular, $X_0$ is a Calabi--Yau threefold.

We use marked lifts in the sense of \cref{sec:canonical-lifts}.

The following length estimate gives freeness and base change for the Hodge groups of every marked lift.

\begin{lemma}
\label{lem:Wn-length}
Let $A_n=W(k)/p^n$.  Let $P$ be a perfect complex over $A_n$.  Then
\[\operatorname{length}_{A_n}H^j(P) \leq n\dim_kH^j(P\otimes_{A_n}^{\mathbf L}k).\]
If equality holds, then $H^j(P)$ is free, and its natural base change map to $H^j(P\otimes_{A_n}^{\mathbf L}k)$ is an isomorphism after reduction modulo $p$.
\end{lemma}

\begin{proof}
By elementary changes of bases, split off all contractible summands arising from unit entries in a finite free representative of $P$. The remaining complex is minimal, so every differential is divisible by $p$. Its reduction modulo $p$ has zero differential.  The rank of its term in degree $j$ is therefore $\dim_kH^j(P\otimes_{A_n}^{\mathbf L}k)$.  The group $H^j(P)$ is a subquotient of this term, so its length has the stated bound.  If equality holds, the kernel of the outgoing differential has the full length of this term, and the image of the incoming differential has length zero. Both differentials therefore vanish.  The cohomology in degree $j$ is the free term itself, which also gives the asserted base change map.
\end{proof}

We apply the length estimate to obtain the Hodge and crystalline properties needed for the canonical lift comparison.

\begin{proposition}
\label{prop:A4-W2-cohomology}
The special fiber satisfies $\dim_kH^{3-i}(X_0,\Omega^i_{X_0/k})=1$ for $0\leq i\leq3$. Its diagonal Kodaira--Spencer class generates $H^1(X_0,T_{X_0})$. Moreover, $H^3_{\cris}(X_0/W)$ and $H^4_{\cris}(X_0/W)$ are torsion free. For every marked smooth proper lift $Y/A$ with special fiber $X_0$ and every $0\leq i\leq3$, the group $H^{3-i}(Y,\Omega^i_{Y/A})$ is free of rank one, and the natural map
\[H^{3-i}(Y,\Omega^i_{Y/A})\otimes_Ak \longrightarrow H^{3-i}(X_0,\Omega^i_{X_0/k})\]
is an isomorphism. The Hodge--de Rham spectral sequence degenerates in total degree three.  For the model $\mathscr X/W$, the degree three Hodge filtration consists of saturated direct summands.  In particular, the integral filtration equals the intersection of the de Rham lattice with its rational Hodge filtration.
\end{proposition}

\begin{proof}
\Cref{prop:permutohedral-deformations,prop:permutohedral-vanishing} give the four Hodge dimensions and the Kodaira--Spencer statement.

The variety $X_0$ lifts to $W_2(k)$ and has dimension $3<p$. Deligne--Illusie degeneration therefore gives $\dim_kH^3_{\dR}(X_0/k)=4$ \cite{DeligneIllusie1987}. The generic fiber has third Betti number four.  Thus $\operatorname{rank}_W H^3_{\cris}(X_0/W)=4$. Crystalline derived base change gives an exact sequence
\[0\longrightarrow H^3_{\cris}(X_0/W)\otimes_Wk \longrightarrow H^3_{\dR}(X_0/k) \longrightarrow H^4_{\cris}(X_0/W)[p] \longrightarrow0.\]
See \cite{BerthelotOgus1978}. Let $\tau_i=\dim_kH^i_{\cris}(X_0/W)[p]$.  Taking dimensions gives $4=4+\tau_3+\tau_4$. Both terms are nonnegative, so $\tau_3=\tau_4=0$.

Crystalline base change now gives $H^3_{\dR}(Y/A)\simeq A^4$. Its $A$-length is eight.  Apply \cref{lem:Wn-length} to the four coherent complexes $R\Gamma(Y,\Omega^i_{Y/A})$.  Each Hodge group in total degree three has length at most two.  The Hodge spectral sequence gives
\[8 =\sum_{i=0}^3\operatorname{length}_AE_\infty^{i,3-i} \leq\sum_{i=0}^3\operatorname{length}_AE_1^{i,3-i} \leq8.\]
Every inequality is an equality.  The equality case of \cref{lem:Wn-length} makes each Hodge group in total degree three free of rank one and identifies its reduction with the corresponding Hodge group of $X_0$. The equality of the $E_1$ and $E_\infty$ lengths also shows that no Hodge differential enters or leaves total degree three. The same length argument applies to $\mathscr X\otimes_WW_n(k)$ for every $n$.  After passage to the inverse limit, its degree three Hodge filtration consists of saturated direct summands.  In particular, the integral filtration equals the intersection of the de Rham lattice with its rational Hodge filtration.
\end{proof}

We next prove a degree three version of the Achinger--Zdanowicz uniqueness argument. For the following lemma, let $k=\F_p$, let $A=W_2(k)$, and let $Z/k$ be a $1$-ordinary smooth projective Calabi--Yau threefold. Assume that $H^3_{\cris}(Z/W(k))$ and $H^4_{\cris}(Z/W(k))$ are torsion free. For every marked smooth proper lift $Y/A$ of $Z$, assume that the Hodge--de Rham spectral sequence degenerates in total degree three and that the groups $H^{3-i}(Y,\Omega^i_{Y/A})$ are free and satisfy base change:
\[H^{3-i}(Y,\Omega^i_{Y/A})\otimes_Ak \xrightarrow{\sim}H^{3-i}(Z,\Omega^i_{Z/k}), \quad 0\leq i\leq3.\]

Let $H_A=H^3_{\cris}(Z/A)$, and let $\Phi:H_A\to H_A$ denote crystalline Frobenius. The map $\Phi$ is $A$-linear because Witt-vector Frobenius on $A$ is the identity. For each marked lift $Y/A$, let $\Fil_Y^\bullet H_A$ denote the filtration corresponding to the Hodge filtration on $H^3_{\dR}(Y/A)$ under crystalline--de Rham comparison \cite{BerthelotOgus1978}.

\begin{lemma}
\label{lem:AZ-degree-three}
A marked smooth proper lift $Y/A$ of $Z$ is isomorphic to the Achinger--Zdanowicz canonical lift as a marked lift if and only if $\Phi(\Fil_Y^1H_A)\subseteq\Fil_Y^1H_A$.
\end{lemma}

\begin{proof}
Since $k=\F_p$, the Frobenius twist of $Z$ is $Z$. Let $H_k=H_A/pH_A$. Crystalline base change and the torsion-freeness assumptions make $H_A$ free. The assumptions on the Hodge groups in total degree three imply that each $\Fil^i_YH_A$ is a direct summand and has reduction $\Fil^iH_k$.

Let $Y_{\mathrm{can}}$ be the Achinger--Zdanowicz canonical lift and write $F_{\mathrm{can}}^i=\Fil^i_{Y_{\mathrm{can}}}H_A$. By \cite[Theorem~5.0.1]{AchingerZdanowicz2021}, one has $\Phi(F_{\mathrm{can}}^1)\subset F_{\mathrm{can}}^1$. In the degree-three form used in \cite[\S5.8]{AchingerZdanowicz2021}, Mazur's divisibility estimate \cite{Mazur1973} gives $\Phi(F_{\mathrm{can}}^1)\subset pH_A$ and $\Phi(F_{\mathrm{can}}^2)\subset p^2H_A=0$. Since $F_{\mathrm{can}}^1$ is a direct summand, $F_{\mathrm{can}}^1\cap pH_A=pF_{\mathrm{can}}^1$. Thus $\Phi(F_{\mathrm{can}}^1)\subset pF_{\mathrm{can}}^1$.

Let $Y/A$ be a marked smooth proper lift of $Z$. Choose an $A$-module complement $C_0$ of $F_{\mathrm{can}}^1$ in $H_A$. Projection along $C_0$ identifies $\Fil^1_YH_A$ with $F_{\mathrm{can}}^1$: modulo $p$, this projection is the identity on $\Fil^1H_k$, and hence it is an isomorphism by Nakayama's lemma. Thus $\Fil^1_YH_A$ is the graph of an $A$-linear map $g_Y:F_{\mathrm{can}}^1\to C_0$ whose image lies in $pC_0$. Since $C_0$ is free, write $g_Y=p\widetilde\delta_Y$ for an $A$-linear map $\widetilde\delta_Y:F_{\mathrm{can}}^1\to C_0$, determined modulo $p$. Thus $\Fil^1_YH_A=\{x+p\widetilde\delta_Y(x):x\in F_{\mathrm{can}}^1\}$. Our graph convention uses the plus sign in this formula. Reducing $\widetilde\delta_Y$ modulo $p$ and using $C_0/pC_0\simeq H_k/\Fil^1H_k$ gives a map $\bar\delta_Y:\Fil^1H_k\to H_k/\Fil^1H_k$. The first-order Griffiths-transversality calculation in the proof of \cite[Theorem~5.7.1]{AchingerZdanowicz2021} shows that $\bar\delta_Y$ vanishes on $\Fil^2H_k$. It therefore induces $\eta_Y:H^2(Z,\Omega_Z^1)\to H^3(Z,\mathcal O_Z)$.

For every marked lift $Y/A$, one also has $\Phi(\Fil_Y^2H_A)=0$. If $y\in\Fil^2_YH_A$, write $y=x+p\widetilde\delta_Y(x)$ with $x\in F_{\mathrm{can}}^1$, and choose $x_2\in F_{\mathrm{can}}^2$ with $x_2\equiv x\pmod p$; such a choice exists because the reduction of $y$ belongs to $\Fil^2H_k$. Then $x-x_2\in pF_{\mathrm{can}}^1$. The preceding transversality statement gives $\widetilde\delta_Y(x)\in F_{\mathrm{can}}^1+pH_A$, so the term $p\widetilde\delta_Y(x)$ can also be absorbed into $pF_{\mathrm{can}}^1$. We may therefore write $y=x_2+pz$ with $z\in F_{\mathrm{can}}^1$. Hence $\Phi(y)=\Phi(x_2)+p\Phi(z)=0$, because $\Phi(F_{\mathrm{can}}^2)=0$, $\Phi(F_{\mathrm{can}}^1)\subset pF_{\mathrm{can}}^1$, and $p^2=0$ in $A$. Thus $\Phi(\Fil^2_YH_A)=0$.

Since $H_A$ is free over $A=W_2(k)$, multiplication by $p$ induces an isomorphism $H_k=H_A/pH_A\xrightarrow{\sim}pH_A=H_A[p]$. Let $\operatorname{div}_p:pH_A\to H_k$ denote its inverse. Since $k=\F_p$, we identify $F_k^*\operatorname{gr}^1H_k$ with $\operatorname{gr}^1H_k$. For $\bar x$ in this space, choose $x\in F_{\mathrm{can}}^1$ representing $\bar x$ and use $x+p\widetilde\delta_Y(x)$ as its lift in $\Fil_Y^1H_A$. Define
\[\gamma_Y(\bar x)=\left[ \operatorname{div}_p\bigl(\Phi(x)\bigr) +\overline{\Phi\bigl(\widetilde\delta_Y(x)\bigr)} \right]_{\operatorname{gr}^0H_k}.\]
This is well defined. First, changing $\widetilde\delta_Y$ by a map with values in $pC_0$ does not change the second term. Next, suppose that the representative $x$ is changed by $x_2\in F_{\mathrm{can}}^2$. Then $\Phi(x_2)=0$, while Griffiths transversality gives $\widetilde\delta_Y(x_2)\in F_{\mathrm{can}}^1+pH_A$. Its contribution therefore vanishes in $\operatorname{gr}^0H_k$. Finally, if $x$ is changed by $pz$ with $z\in F_{\mathrm{can}}^1$, then
\[\operatorname{div}_p(\Phi(x+pz))- \operatorname{div}_p(\Phi(x))=\overline{\Phi(z)}=0,\]
because $\Phi(F_{\mathrm{can}}^1)\subset pF_{\mathrm{can}}^1$; the change in the second term also vanishes modulo $p$. These are precisely the choices involved in representing a class in $\operatorname{gr}^1H_k$.

For $y=x+p\widetilde\delta_Y(x)\in\Fil_Y^1H_A$, write $\Phi(x)=ph_x$. Then $\Phi(y)=p(h_x+\Phi(\widetilde\delta_Y(x)))$. Moreover, $\Fil_Y^1H_A\cap pH_A=p\Fil_Y^1H_A=pF_{\mathrm{can}}^1$. It follows that
\begin{equation}\label{eq:gamma-stability}
 \gamma_Y=0
 \quad\text{if and only if}\quad
 \Phi(\Fil^1_YH_A)\subset\Fil^1_YH_A.
\end{equation}

For $Y_{\mathrm{can}}$, we choose $\widetilde\delta_{Y_{\mathrm{can}}}=0$, and one has $\gamma_{Y_{\mathrm{can}}}=0$. Reduction of $\Phi$ on $\operatorname{gr}^0H_k=H^3(Z,\mathcal O_Z)$ is the Hasse--Witt map $HW(0)$. With the graph convention above, a direct calculation gives
\begin{equation}\label{eq:AZ-obstruction-difference}
 \gamma_Y-\gamma_{Y_{\mathrm{can}}}
 =HW(0)\circ F_k^*\eta_Y.
\end{equation}
Indeed, since $\Phi(F_{\mathrm{can}}^1)\subset pF_{\mathrm{can}}^1$, the term $\operatorname{div}_p(\Phi(x))$ lies in $\Fil^1H_k$ and therefore vanishes in $\operatorname{gr}^0H_k$. The remaining term is the reduction of crystalline Frobenius applied to the graph map, and crystalline Frobenius induces $HW(0)$ on $\operatorname{gr}^0H_k$.

The map $HW(0)$ is an isomorphism because $Z$ is $1$-ordinary. If $Y$ satisfies the Frobenius-stability condition, then \eqref{eq:gamma-stability} and \eqref{eq:AZ-obstruction-difference} give $\eta_Y=0$. The isomorphism classes of marked lifts over $A$ form a torsor under $H^1(Z,T_Z)$. Let $v(Y,Y_{\mathrm{can}})\in H^1(Z,T_Z)$ denote the difference class. Ogus identifies the map from $v(Y,Y_{\mathrm{can}})$ to $\eta_Y$ with cup product \cite[Corollary~2.12]{Ogus1978}. Since $\omega_Z\simeq\mathcal O_Z$, Serre duality shows that
\[H^1(Z,T_Z)\longrightarrow \operatorname{Hom}\bigl(H^2(Z,\Omega_Z^1),H^3(Z,\mathcal O_Z)\bigr)\]
is an isomorphism. Thus $v(Y,Y_{\mathrm{can}})=0$, so $Y$ and $Y_{\mathrm{can}}$ are isomorphic as marked lifts of $Z$.
\end{proof}

\subsection{The integral splitting and the canonical lift}

\begin{proof}[Proof of \cref{thm:A4-W2-canonical}]
We first identify the Hasse--Witt scalar of the quotient with the Dwork scalar. We then construct an integral projector and transport it to the crystalline lattice by the integral comparison theorem of Cais--Liu. The resulting filtered splitting gives the Frobenius stability needed to apply \cref{lem:AZ-degree-three}.

The connecting map in the Koszul resolution identifies $H^3(\widetilde{\mathscr Y}_k,\mathcal O_{\widetilde{\mathscr Y}_k})$ with $H^5(V_k,\omega_{V_k})$.  In the toric \v Cech description, Frobenius is computed by multiplication by $(P_1P_{2,a})^{p-1}$ followed by extraction of the constant term.  Hence the Hasse--Witt scalar of the cover is $\CT((P_1P_{2,a})^{p-1})$. Let $c_m=\CT(g^m)$.  Direct expansion of the two factors gives
\begin{align*}
 \CT\bigl((P_1P_{2,t})^{p-1}\bigr)
 &=\sum_{m=0}^{p-1}\binom{p-1}{m}^2t^mc_m\\
 &\equiv\sum_{m=0}^{p-1}\binom{p-1}{m}(-t)^mc_m
 =\CT(f_t^{p-1})\pmod p.
\end{align*}
The congruence uses $\binom{p-1}{m}\equiv(-1)^m\pmod p$.  Pullback along the finite \'etale quotient identifies $H^3(X_0,\mathcal O_{X_0})$ with $H^3(\widetilde{\mathscr Y}_k,\mathcal O_{\widetilde{\mathscr Y}_k})^G$. Thus $HW(0)$ on this space has scalar $h_p$, and $X_0$ is $1$-ordinary if and only if $h_p\neq0$.

Let $\Lambda_{\et}=H^3_{\et}(X_{t_*,\overline\Q},\Z_p)/\mathrm{tors}$, and write $V_f=A_p$ and $V_g=B_p$ for the two factors in the rational decomposition of \cref{prop:actual-split}. Dummigan gives $a_3(f_{14,4})=8$ and $a_3(f_{14,2})=-2$ \cite[the table in the proof of Proposition~4.6]{Dummigan2026}. The Tate twist changes the second trace to $-6$.  Thus their Frobenius polynomials at the auxiliary prime $3$ are $P_f(Z)=Z^2-8Z+27$ and $P_g(Z)=Z^2+6Z+27$. Let $\rho_{\et}$ be the Galois representation on $\Lambda_{\et}$, let $\Frob_3$ be arithmetic Frobenius, and let $F_3=\rho_{\et}(\Frob_3^{-1})$.  Thus $F_3$ is geometric Frobenius, as in Dummigan's convention. Define
\begin{equation}\label{eq:A4-integral-projector}
 e_f=\frac{P_g(F_3)(8-F_3)}{14\cdot27}.
\end{equation}
On $V_f$, the relations $P_g(F_3)=14F_3$ and $F_3(8-F_3)=27$ give $e_f=1$.  On $V_g$, the relation $P_g(F_3)=0$ gives $e_f=0$. Thus $e_f$ is the rational Galois equivariant projector onto $V_f$. The denominator $14\cdot27=2\cdot3^3\cdot7$ is a unit in $W$ for every prime under consideration.  Formula \eqref{eq:A4-integral-projector} therefore preserves $\Lambda_{\et}$.  Let $\Lambda_f=e_f\Lambda_{\et}$ and $\Lambda_g=(1-e_f)\Lambda_{\et}$.  Then $\Lambda_{\et}=\Lambda_f\oplus\Lambda_g$.

Let $\mathfrak S=W[[u]]$, let $\vartheta:\mathfrak S\to W$ send $u$ to the uniformizer $p$, and let $S$ be the $p$-adic completion of the divided-power envelope of $\ker(\vartheta)$.  We use the Frobenius on $S$ induced by $u\mapsto u^p$. The map $\operatorname{sp}_0:S\to W$, defined by $\operatorname{sp}_0(f(u))=f(0)$, is the crystalline specialization map.

We use Theorem~5.4 of \cite{CaisLiu2019} in the corrected form of \cite{CaisLiuCorrigendum2020}. The integral comparison theorem of Cais--Liu applies in degree three: one has $3<p-1$, and \cref{prop:A4-W2-cohomology} gives the required torsion-freeness in degrees three and four.  Let $\mathscr M^3=H^3_{\cris}(X_0/S)$.  By \cite[Theorem~5.4(2)]{CaisLiu2019}, $\mathscr M^3$ is a strongly divisible $S$-lattice. Define $e_f^\vee:\Lambda_{\et}^\vee\to\Lambda_{\et}^\vee$ by $e_f^\vee(\lambda)=\lambda\circ e_f$. Under the natural isomorphism of \cite[Theorem~5.4(3)]{CaisLiu2019}, the endomorphism attached to $e_f^\vee$ defines an idempotent on the Breuil--Kisin module associated with $\mathscr M^3$. The equivalence in \cite[Theorem~3.8]{CaisLiu2019} gives a unique idempotent $\widetilde e_f$ on the corresponding filtered Breuil module. The equivalence transports the idempotent as a morphism of filtered Breuil modules.  By construction it preserves the filtration and commutes with Frobenius and monodromy. Hence it is an endomorphism of the strongly divisible lattice $\mathscr M^3$.

Since $K=\Q_p$ has ramification index $e=1$, the condition $p^n\geq e$ used before equation~(5.3) of \cite{CaisLiu2019} permits $n=0$.  The resulting $\varphi^0$-twists are identities, and that equation gives
\[H^3_{\cris}(X_0/W)\otimes_WS\xrightarrow{\sim}\mathscr M^3.\]
Write $H_W=H^3_{\cris}(X_0/W)$ for the crystalline lattice. Base change by $\operatorname{sp}_0:S\to W$ gives a Frobenius-equivariant isomorphism from $H_W$ to $\mathscr M^3\otimes_{S,\operatorname{sp}_0}W$. Taking the inverse isomorphism gives
\[\mathscr M^3\otimes_{S,\operatorname{sp}_0}W\xrightarrow{\sim} H_W.\]
Let $e_{f,\cris}$ be the specialization of $\widetilde e_f$, and let $H_{W,f}=\operatorname{im}(e_{f,\cris})$ and $H_{W,g}=\operatorname{im}(1-e_{f,\cris})$. Then $H_W=H_{W,f}\oplus H_{W,g}$, and both summands are Frobenius-stable.  Under the smooth proper comparison
\[H_W[1/p]\simeq H_p= D_{\cris}\bigl(\Lambda_{\et}\otimes_{\Z_p}\Q_p\bigr),\]
the naturality of \cite[Theorem~4.3(2)]{CaisLiu2019} and of \cite[equation~(5.5)]{CaisLiu2019} identifies $e_{f,\cris}[1/p]$ with the projector induced by $e_f$. Consequently, in the notation of \eqref{eq:Dcris-split}, $H_{W,f}[1/p]=D_f$ and $H_{W,g}[1/p]=D_g$. The Hodge degrees of $D_f$ and $D_g$ are $\{0,3\}$ and $\{1,2\}$, respectively, by \eqref{eq:HT}.

The specialization $\operatorname{sp}_0:u\mapsto0$ gives the crystalline lattice, whereas the Hodge filtration comes from the de Rham specialization $\vartheta:u\mapsto p$.  Let $D=H_W[1/p]$.  By the preceding identification, $D=D_f\oplus D_g$.  The rational projector preserves $\Fil^\bullet D$.  Define $\Fil^iH_{W,f}=H_{W,f}\cap\Fil^iD_f$ and $\Fil^iH_{W,g}=H_{W,g}\cap\Fil^iD_g$. By \cref{prop:A4-W2-cohomology}, the Hodge filtration on $H_W\simeq H^3_{\dR}(\mathscr X/W)$ is saturated and satisfies $\Fil^iH_W=H_W\cap\Fil^iD$.  Since $e_{f,\cris}$ preserves $H_W$ and its rationalization is the filtered projector onto $D_f$, it preserves $\Fil^iH_W$.  The complementary projector has the same property.  Hence $\Fil^iH_W=\Fil^iH_{W,f}\oplus\Fil^iH_{W,g}$. The Hodge degrees give $\Fil^1D_f=\Fil^3D_f$ and $\Fil^1D_g=D_g$.  Intersecting with the two lattices yields
\begin{equation}\label{eq:A4-Fil-split}
 \Fil^1H_W=\Fil^3H_{W,f}\oplus H_{W,g}.
\end{equation}

Write $\varphi_S$ for Frobenius on $\mathscr M^3$, and let $\Phi$ denote crystalline Frobenius on $H_W$. Under $\mathscr M^3\simeq S\otimes_WH_W$, the rational monodromy operator is $N_S\otimes1$, where $N_S=-u\,d/du$.  If $x\in\Fil^3H_W$, its constant section $1\otimes x$ has zero monodromy and its specialization under $\vartheta$ belongs to $\Fil^3D$.  The recursive definition of the Breuil filtration in \cite[equation~(4.5)]{CaisLiu2019} therefore shows that $1\otimes x$ lies in $\Fil^3(\mathscr M^3[1/p])\cap\mathscr M^3 =\Fil^3\mathscr M^3$. Strong divisibility gives $\varphi_S(1\otimes x)\in p^3\mathscr M^3$. Applying $\operatorname{sp}_0$ and using its Frobenius compatibility yields $\Phi(x)\in p^3H_W$.  Thus Frobenius stability of $H_{W,f}$ gives
\[\Phi(\Fil^3H_{W,f})\subset p^3H_W\cap H_{W,f}=p^3H_{W,f}.\]
The last equality uses the direct sum $H_W=H_{W,f}\oplus H_{W,g}$. Moreover, $\Phi(H_{W,g})\subset H_{W,g}$. Base change $W\to A$ identifies the Hodge filtration of $\mathscr X/W$ with that of $X_*^A$. After this base change, the first summand in \eqref{eq:A4-Fil-split} maps to zero under $\Phi$, because $p^3H_{W,f}\subset p^2H_W$, while the second summand remains in $H_{W,g}\otimes_WA$. Consequently \eqref{eq:A4-Fil-split} gives
\begin{equation}\label{eq:A4-F1-stable}
 \Phi\bigl(\Fil^1H^3_{\dR}(X_*^A/A)\bigr)
 \subset\Fil^1H^3_{\dR}(X_*^A/A).
\end{equation}
If $h_p\neq0$, then $X_0$ is $1$-ordinary.  Since $k=\F_p$, its Frobenius twist is $X_0$ itself. By \cref{prop:A4-W2-cohomology} the torsion-freeness, freeness, degeneration, and base change hypotheses of \cref{lem:AZ-degree-three} hold. Consequently that lemma, applied with \eqref{eq:A4-F1-stable}, identifies $X_*^A$ with the Achinger--Zdanowicz canonical lift.
\end{proof}

\subsection{Comparison of the two obstruction scalars}

We now choose compatible bases and compare the Dwork and crystalline obstruction scalars throughout the residue disk.

\begin{proof}[Proof of \cref{thm:A4-W2-comparison}]
Assume now that both Dwork Hasse--Witt conditions hold at $a$. Let $\omega$ denote the residue section in \eqref{eq:omega-hol}. At the canonical lift $X_*^A$, choose a symplectic basis $(e_3,e_2,e_1,e_0)$ of $H^3_{\dR}(X_*^A/A)$ adapted to the Hodge filtration.  Normalize it by
\[e_3=[\omega],\quad e_2=[\nabla_\theta\omega]\pmod{\Fil^3},\quad \langle e_3,e_0\rangle=\langle e_2,e_1\rangle=1.\]
Let $\overline Q_a$ be the reduction of the level two quotient at $a$, and let $\bar\omega_0,\bar\omega_1\in\overline Q_a$ be the reductions of the Dwork basis classes $1/f$ and $\theta(1/f)$. Let $\overline L_a\subset\overline Q_a$ be the line generated by $\bar\omega_0$. The splitting \eqref{eq:BV-split} gives integral representatives of these classes in $M(2)$. Let $\bar e_i$ denote the reduction of $e_i$.  The residue maps send $\bar\omega_0$ to $\bar e_3$ and the class of $\bar\omega_1$ modulo $\overline L_a$ to $\bar e_2$.  By \cref{prop:A4-W2-cohomology}, these classes generate $\operatorname{gr}^3H^3_{\dR}(X_0/k)$ and $\operatorname{gr}^2H^3_{\dR}(X_0/k)$; for $\bar e_2$, this uses the nonzero Kodaira--Spencer class of the diagonal parameter.  Since the source and target of each residue map are one-dimensional, they give isomorphisms $\overline L_a\simeq\operatorname{gr}^3H^3_{\dR}(X_0/k)$ and $\overline Q_a/\overline L_a\simeq \operatorname{gr}^2H^3_{\dR}(X_0/k)$.  Thus $I_a=k\bar e_2$ and $K_a=k\bar e_1$ in the notation of \cref{sec:canonical-lifts}. With these normalizations, $\kappa_a$ sends the map $\bar\omega_0\mapsto\bar\omega_1\bmod\overline L_a$ to the map $\bar e_1\mapsto\bar e_0$ on the corresponding graded pieces. These identifications are compatible with a change of holomorphic generator and parameter: for units $c,v$, replacing $\omega$ by $c\omega$ and $\theta$ by $v\theta$ multiplies both obstruction scalars by $v^{-1}$.

The reductions of $e_1$ and $e_0$ give bases of $\operatorname{gr}^1$ and $\operatorname{gr}^0$.  We use these bases to identify $\gamma_Y$ with its scalar in $k$.

We first compute the Achinger--Zdanowicz obstruction. Crystalline--de Rham comparison identifies the de Rham cohomology of the lifts at $t_\delta$ and $t_*$ with two evaluations of the same crystal. The Gauss--Manin connection gives the Taylor isomorphism between these evaluations. With our convention, transport of the Hodge line from $t_\delta$ to $t_*$ has first-order term $+(t_\delta-t_*)\nabla_{\partial_t}$ \cite{BerthelotOgus1978}. Since $t_\delta-t_*=p\delta$ and $\nabla_{\partial_t}=t^{-1}\nabla_\theta$, the transported Hodge line is $\Fil_\delta^3=\langle e_3+p(\delta/a)e_2\rangle$. The class $p(\delta/a)\in pA$ is well defined. Since $\Fil_\delta^1$ is the annihilator of $\Fil_\delta^3$, the class of $y_\delta=e_1-p(\delta/a)e_0$ generates $\Fil_\delta^1/\Fil_\delta^2$. With the graph convention of \cref{lem:AZ-degree-three}, $\eta_{X_\delta^A}(\bar e_1)=-(\delta/a)\bar e_0$. The obstruction at $X_*^A$ is zero, and $HW(0)$ has scalar $h_p$ on $\operatorname{gr}^0$.  Equation \eqref{eq:AZ-obstruction-difference} gives
\begin{equation}\label{eq:A4-AZ-obstruction}
 \gamma_{X_\delta^A}=-\frac{h_p}{a}\delta\quad\text{in }k.
\end{equation}

For the root $\mathbf v$ fixed in \cref{sec:A4-Dwork}, let $A_m(t)=\CT(f_t^m)$ and $B_{m,Q}(t)=[\mathbf x^{Q\mathbf v}]f_t^m$, as in \cref{app:level two}. For the Beukers--Vlasenko obstruction, let $R_p(t)=-A_{2p-1}(t)-(1-5t)B_{2p-1,p}(t)/t$. The finite formula \eqref{eq:lambda-finite-target} gives $\lambda_1(t,t)\equiv R_p(t)\pmod{p^2}$ for every lift $t$ of $a$. Reduction modulo $p$ gives $f_t^p=f_{t^p}(\mathbf x^p)$. The support of $f_t^{p-1}$ is contained in $(p-1)\Delta$. Since $0$ is the only interior lattice point of $\Delta$, no nonzero lattice vector $z$ satisfies $pz\in(p-1)\Delta$. For a root $\mathbf w$, the exponent $p\mathbf w$ is outside this polytope; for distinct roots $\mathbf v,\mathbf w$, so is $p(\mathbf v-\mathbf w)$. Consequently only the constant term of $f_t^{p-1}$ contributes to the constant and $\mathbf x^{p\mathbf v}$ coefficients of $f_t^{p-1}f_{t^p}(\mathbf x^p)$, and
\[A_{2p-1}(t)\equiv(1-5t^p)A_{p-1}(t),\quad B_{2p-1,p}(t)\equiv-t^pA_{p-1}(t)\pmod p.\]
It follows that $R_p(t)\equiv(t^{p-1}-1)A_{p-1}(t)\pmod p$. Since $a^{p-1}=1$, its derivative satisfies $R_p'(a)=-h_p/a$ in $k$. By \cref{thm:main}, $\lambda_1(t_*,t_*)=0$, so $R_p(t_*)\equiv0\pmod{p^2}$. The function $R_p(t)$ is a finite Laurent polynomial with coefficients in $W$.  Its first-order expansion at $t_*$ therefore gives
\begin{equation}\label{eq:A4-BV-obstruction}
 \frac{\lambda_1(t_\delta,t_\delta)}p
 =-\frac{h_p}{a}\delta\quad\text{in }k.
\end{equation}
Under the comparison $\kappa_a$, equations \eqref{eq:A4-AZ-obstruction} and \eqref{eq:A4-BV-obstruction} prove \cref{thm:A4-W2-comparison}. Since $h_p/a\neq0$, their common zero is $\delta=0$.
\end{proof}

\section{The mirror quartic K3 pencil}
\label{sec:K3}

We apply \cref{thm:rank two-criterion} in weight two. At the CM fiber considered below, the variable cohomology of the mirror quartic splits into a rank two transcendental factor and an algebraic line.

\subsection{The simplicial Dwork model and its CM fiber}

Let
\begin{equation*}
 h(x,y,z)=x+y+z+(xyz)^{-1},
 \quad f_t=1-th.
\end{equation*}
This is the completely symmetric simplicial family in three variables studied in \cite[\S7]{BeukersVlasenkoIII}. The equation $f_t=0$ gives a torus model for the quotient by the diagonal automorphisms preserving the holomorphic two form of the Fermat--Dwork quartic 
\[Y_0^4+Y_1^4+Y_2^4+Y_3^4 =4\psi Y_0Y_1Y_2Y_3, \quad t=(4\psi)^{-1}.\]
For the hypergeometric parameter, write
\begin{equation*}
 z_{\mathrm{hyp}}=\psi^{-4}=(4t)^4=256t^4.
\end{equation*}
Let $\tau$ be a root of
\begin{equation}\label{eq:K3-tau}
 \tau^4=-\frac1{12288};
\end{equation}
then $z_{\mathrm{hyp}}=-1/48$.  At this value the associated Fermat--Dwork quartic has Picard number $20$ and is a Kummer surface associated with the CM order of discriminant $-36$; see \cite[the introduction and Example~3.12]{DembelePanchishkinVoightZudilin2022}. This value also has a modular description. If $\omega_{\mathrm{CM}}=(1+3i)/2$, then the modular function $\rho_{\mathrm{mod}}$ of \cite[(2.4)--(2.6)]{DembelePanchishkinVoightZudilin2022} satisfies $\rho_{\mathrm{mod}}(\omega_{\mathrm{CM}})=-1/48=z_{\mathrm{hyp}}$.

Let $\lambda=t^4$, so that $\rho_{\mathrm{mod}}=256\lambda$. For the associated third order modular differential equation, the excellent Frobenius in the modular coordinate is induced by $\lambda(\omega)\mapsto\lambda(p\omega)$ \cite[\S1 and Appendix~C.2]{BeukersSupercongruencesModularForms}. Thus the K3 fixed point considered below lies in the modular CM setting.

We consider
\[V_\lambda:\quad \lambda(Y_1+Y_2+Y_3+Y_4)^4=Y_1Y_2Y_3Y_4 \quad\text{in }\mathbb P^3.\]
On the dense torus, the map $(x,y,z)\mapsto(x:y:z:(xyz)^{-1})$ identifies $f_t=0$ over $\Q(t)$ with the open subset of $V_{t^4}$ where all four coordinates are nonzero. Indeed, the displayed representative has product one and sum $1/t$; conversely a point of $V_{t^4}$ with nonzero coordinates and sum $S$ can be scaled by $(tS)^{-1}$ to have these two properties. The surface $V_{-1/12288}$ is defined over $\Q$. After base change to $\Q(\tau)$, its open subset where all four coordinates are nonzero is isomorphic to $\{f_\tau=0\}$.  This surface is birational to the diagonal quotient of the Fermat--Dwork quartic.  The relation with this quotient and the third order Picard--Fuchs equation are described in \cite[Example~3.12]{DembelePanchishkinVoightZudilin2022} and \cite[Example~4.1.9 and Theorem~5.1.3] {DoranKellySalernoSperberVoightWhitcher2018}.

Choose an embedding $\Q(\tau)\hookrightarrow\C$. Let $Y_\tau^\circ=\{f_\tau=0\}\subset\Gm^3$. Choose a toric crepant resolution $\widetilde Y_\tau$ of its compactification, and let $D_\tau=\widetilde Y_\tau\setminus Y_\tau^\circ$ be its reduced boundary. Pullback under the quotient map embeds the transcendental cohomology of $\widetilde Y_\tau$ into that of the Fermat--Dwork quartic. The latter has rank two, so $\widetilde Y_\tau$ also has Picard number $20$. In rational Betti cohomology, let $L_{\mathrm{tor},B}$ be the span of the classes of all irreducible components of $D_\tau$, and let $H_{\mathrm{var},B}=L_{\mathrm{tor},B}^{\perp}$. The restriction calculation in \cref{prop:K3-BR} below proves that $L_{\mathrm{tor},B}$ has rank nineteen. The Hodge index theorem makes the intersection form on $L_{\mathrm{tor},B}$ nondegenerate, so $H_{\mathrm{var},B}$ has rank three.  At $z_{\mathrm{hyp}}=-1/48$, the intersection $L_{\mathrm{alg},B}=H_{\mathrm{var},B}\cap \operatorname{NS}(\widetilde Y_{\tau,\overline{\mathbb Q}})_{\mathbb Q}$ is a line because the Picard number is $20$.  It lies in the primitive $(1,1)$-part, and the polarization is nonzero on it.  Its orthogonal complement therefore gives a rational Hodge decomposition
\begin{equation}\label{eq:K3-split}
 H_{\mathrm{var},B}=H_{\mathrm{tr},B}\oplus L_{\mathrm{alg},B}.
\end{equation}
The two summands have Hodge types $\{(2,0),(0,2)\}$ and $\{(1,1)\}$, respectively.  The spaces $L_{\mathrm{tor},B}$ and $L_{\mathrm{alg},B}$ are generated by algebraic cycle classes.  Their cycle-class realizations are Galois-stable: in $\ell$-adic cohomology, $L_{\mathrm{alg}}$ is the intersection of the Galois-stable variable summand with the Galois-stable geometric N\'eron--Severi subspace.  Their orthogonal projectors, defined using the intersection form, are therefore Galois equivariant.  At every good place at which these divisor classes extend, the projectors are compatible with the Betti, $\ell$-adic, de Rham, and crystalline realizations \cite{Grothendieck1966,Tsuji1999}.  We denote the corresponding summands by $H_{\mathrm{var}}$, $H_{\mathrm{tr}}$, and $L_{\mathrm{alg}}$ after specifying a realization. The de Rham realizations of $H_{\mathrm{tr}}$ and $L_{\mathrm{alg}}$ have Hodge degrees $\{0,2\}$ and $\{1\}$, respectively.

We now construct the bounded residue comparison.

\begin{proposition}
\label{prop:K3-BR}
Let $v$ be the place of $\Q(\tau)$ above $p=37$ selected by $\tau\equiv15\pmod {37}$, and let $K=\Q_{37}$. Write $f=f_\tau$, and use the Frobenius lift on the coefficient ring $\sigma_\tau(t)=\tau^{1-p}t^p$, which fixes $\tau$. Derivatives are taken before setting $t=\tau$. Let $H_{\mathrm{var}}$ denotes the crystalline realization over $K$ of the rank three summand $H_{\mathrm{var},B}$ defined above.  Crystalline--de Rham comparison equips it with a Hodge filtration, and it is a weakly admissible filtered $\varphi$-module \cite{Tsuji1999,ColmezFontaine2000}.  There is a bounded surjection
\[\rho_{K3}:CY(h)_\tau[1/p]\twoheadrightarrow H_{\mathrm{var}}\]
which satisfies $\rho_{K3}\cC_p=p^2\varphi^{-1}\rho_{K3}$. The images of $1/f$ and $\theta(1/f)$ generate respectively $\Fil^2H_{\mathrm{var}}$ and $\Fil^1H_{\mathrm{var}}/\Fil^2H_{\mathrm{var}}$.
\end{proposition}

\begin{proof}
First we compute the variable part by the toric Jacobian ring, then we identify it in Betti and de Rham cohomology through the restriction map. Finally we construct a model over $\Z_{37}$ and pass to rigid cohomology.

Let $\Delta$ be the Newton polytope of $h$.  It is the reflexive simplex with vertices $(1,0,0)$, $(0,1,0)$, $(0,0,1)$, and $(-1,-1,-1)$. Every proper face consists of affinely independent vertices, so its logarithmic Jacobian has full rank on the torus.  In the dense torus the critical equations force $x=y=z=(xyz)^{-1}$; hence the only discriminant is $(4t)^4=1$. At \eqref{eq:K3-tau}, $(4\tau)^4-1=-49/48$, a $37$-adic unit.  

Over $\Q(\tau)$, the homogeneous semigroup ring of the simplex is
\[S_\Delta=\Q(\tau)[z_0,z_1,z_2,z_3,z_4]/ (z_1z_2z_3z_4-z_0^4),\]
where $\deg z_i=1$, and the homogenization of $f_t$ is $F_t=z_0-t(z_1+z_2+z_3+z_4)$. The logarithmic Jacobian relations give $z_1=z_2=z_3=z_4=:z$ and $z_0=4tz$. The semigroup relation then gives $\bigl(1-(4t)^4\bigr)z^4=0$. Since $1-(4\tau)^4=49/48$ is a unit, the positive-degree Jacobian quotient at $t=\tau$ is $(z)\Q(\tau)[z]/(z^4)=\Q(\tau)z\oplus\Q(\tau)z^2\oplus\Q(\tau)z^3$. Let $I_\Delta^{(1)}\subset S_\Delta$ be the ideal generated by monomials of degree $m\geq1$ whose exponents lie in $m\Delta^\circ$. The monomial $z_0$ is the degree-one interior monomial, and $z=(4\tau)^{-1}z_0$ in the quotient. Thus the image of the interior ideal $I_\Delta^{(1)}$ is the whole three-dimensional positive-degree quotient.

Let $\overline Y_\tau\subset P_\Delta$ be the closure of $Y_\tau^\circ$ in the projective simplicial toric threefold associated with $\Delta$.  It is an ample nondegenerate anticanonical hypersurface. Write $PH^2(\overline Y_\tau)$ for its primitive cohomology, namely the quotient by the image of $H^2(P_\Delta)$.  Proposition~11.6 of \cite{BatyrevCox1994} gives a natural isomorphism $PH^2(\overline Y_\tau)\xrightarrow{\sim}W_2PH^2(Y_\tau^\circ)$. Here $W_\bullet$ denotes the weight filtration. The group $PH^2(Y_\tau^\circ)$ is the cokernel of $H^2(\Gm^3)\to H^2(Y_\tau^\circ)$.  Since $H^2(\Gm^3)$ is pure of weight four, strictness of the weight filtration gives $W_2PH^2(Y_\tau^\circ)=W_2H^2(Y_\tau^\circ)$.  Theorem~11.8 of \cite{BatyrevCox1994} identifies this isomorphism with the residue map from the image of $I_\Delta^{(1)}$ in the Jacobian ring.  The calculation above therefore shows that both sides have dimension three and have one-dimensional Hodge pieces of types $(2,0)$, $(1,1)$, and $(0,2)$.

Apply the localization sequence to the smooth pair $(\widetilde Y_\tau,D_\tau)$.  Its degree two part gives
\[\bigoplus_{D_i\subset D_\tau}\Q(-1)[D_i] \longrightarrow H^2(\widetilde Y_\tau,\Q) \xrightarrow{j^*}H^2(Y_\tau^\circ,\Q).\]
The image of the first map is $L_{\mathrm{tor},B}$ by definition. Deligne's weight spectral sequence for a smooth proper variety with a simple normal crossings boundary \cite[\S3.2]{Deligne1971} gives $\operatorname{im}(j^*)=W_2H^2(Y_\tau^\circ,\Q)$. The Batyrev--Cox calculation above shows that this image has dimension three.  Since a K3 surface has second Betti number $22$, exactness gives
\[\dim L_{\mathrm{tor},B}=19, \quad H^2(\widetilde Y_\tau,\Q)/L_{\mathrm{tor},B} \xrightarrow{\sim}W_2H^2(Y_\tau^\circ,\Q).\]
A toric ample divisor is supported on the toric boundary, so its restriction to $\widetilde Y_\tau$ belongs to $L_{\mathrm{tor},B}$.  The Hodge index theorem therefore makes the intersection form on $L_{\mathrm{tor},B}$ nondegenerate.  Orthogonal projection consequently identifies the quotient above with $H_{\mathrm{var},B}=L_{\mathrm{tor},B}^\perp$.

Algebraic de Rham comparison \cite{Grothendieck1966} gives the same dimension and restriction statements in de Rham cohomology. The residue theorem identifies the classes of $z_0,z_0^2,z_0^3$, up to nonzero scalar factors, with the images of the residues represented by $1/f,1/f^2,1/f^3$ in the corresponding Hodge graded pieces. Since $\tau h=1-f$, one has
\[\theta\!\left(\frac1f\right)=\frac1{f^2}-\frac1f, \quad \theta^2\!\left(\frac1f\right) =\frac2{f^3}-\frac3{f^2}+\frac1f.\]
The transition matrix is triangular with diagonal entries $1,1,2$. Thus the residues represented by $1/f$, $\theta(1/f)$, and $\theta^2(1/f)$ form a basis of the image of the de Rham cohomology of the proper compactification in $H^2_{\dR}(Y_\tau^\circ/\Q(\tau))$. The restriction map gives unique lifts in the orthogonal complement of the boundary divisor classes in $H^2_{\dR}(\widetilde Y_\tau/\Q(\tau))$, and these lifts form a basis. We extend scalars to $K$ through the embedding $\Q(\tau)\hookrightarrow K$ determined by $v$.

A projective unimodular subdivision of the fan of the dual simplex \cite[Theorem~2.2.24 and Proposition~3.2.1]{Batyrev1994Dual} gives a smooth proper toric threefold $\mathscr P$ over $\Z_{37}$.  The closure $\widetilde{\mathscr Y}_\tau\subset\mathscr P$ is smooth and proper, and it meets the toric boundary transversely.  Thus the boundary together with $\widetilde{\mathscr Y}_\tau$ is a relative simple normal crossings divisor.  Write $Y_0^\circ$ for the affine special fiber in the torus and $\widetilde Y_0$ for the proper special fiber.

For the smooth relative simple normal crossings model over $\Z_{37}$ constructed above, the comparison theorems \cite[Theorem~2.3, Proposition~3.6, Corollary~3.8(c), and Corollary~3.9]{GrosseKlonne2004} give comparison isomorphisms for the proper model and the complement of its boundary, compatible with restriction. Under these isomorphisms, the boundary divisor classes and the algebraic de Rham residue classes represented by $1/f$, $1/f^2$, and $1/f^3$ specialize to the corresponding rigid classes.  Let $D=\widetilde Y_0\setminus Y_0^\circ$.  The degree two part of the rigid localization sequence is
\[\bigoplus_{D_i\subset D}K(-1)[D_i] \longrightarrow H^2_{\rig}(\widetilde Y_0/K) \xrightarrow{j^*}H^2_{\rig}(Y_0^\circ/K).\]
Let $L_{\mathrm{tor},K}$ denote the image of the first map, equivalently the span of the crystalline classes of all irreducible components of $D$. These comparison isomorphisms identify it with the realization of $L_{\mathrm{tor},B}$.  The weight bounds for the localization sequence give $\operatorname{im}(j^*)=W_2H^2_{\rig}(Y_0^\circ/K)$ \cite[Theorems~2.2--2.3]{Chiarellotto1998Weights}. The intersection matrix and the dimensions above are unchanged after extension to $K$.  Restricting $j^*$ to the orthogonal complement of $L_{\mathrm{tor},K}$ therefore gives the Frobenius-equivariant isomorphism
\begin{equation}\label{eq:K3-weight-two-restriction}
 j_{37}:H_{\mathrm{var}}\xrightarrow{\sim}
 W_2H^2_{\rig}(Y_0^\circ/K).
\end{equation}
The three cyclic classes above map to a basis under $j_{37}$.

Apply \cref{thm:general-bounded-residue} with the trivial group, $n=3$, $d=2$, and $H=H_{\mathrm{var}}$.  The smooth projective relative simple normal crossings model was constructed above.  Equation \eqref{eq:K3-weight-two-restriction} is condition~\textup{(iii)} of that theorem, and the three cyclic residue classes give condition~\textup{(iv)}. The theorem therefore gives the bounded surjection $\rho_{K3}$ and the identity $\rho_{K3}\cC_p=p^2\varphi^{-1}\rho_{K3}$. Finally, the Batyrev--Cox residue description above shows that $1/f$ spans $\Fil^2$ and that $\theta(1/f)$ has nonzero image in $\Fil^1/\Fil^2$.  This proves the assertion about the top two Hodge pieces.
\end{proof}

\subsection{The excellent Frobenius at \texorpdfstring{$p=37$}{p=37}}

We now verify the Hasse--Witt conditions at $p=37$ and apply the general criterion. We choose this prime because the CM parameter has a lift in $\Z_{37}$ and the first two Hasse--Witt invariants are units there.

Hensel's lemma gives $\tau\equiv1088\pmod{37^2}$ for the root selected in \cref{prop:K3-BR}. By symmetry, choose the vertex $e_1$, and let $c_k=\CT(h^k)$ and $d_{k,Q}=[\mathbf x^{Qe_1}]h^k$. The two coefficient sequences used in the finite level two calculation are
\begin{equation*}
 c_k=
 \begin{cases}
 \dfrac{k!}{(k/4)!^4},&4\mid k,\\[4pt]
 0,&4\nmid k,
 \end{cases}
 \quad
 d_{k,Q}=
 \begin{cases}
 \dfrac{k!}{(r+Q)!r!^3},&k-Q=4r\geq0,\\[4pt]
 0,&\text{otherwise}.
 \end{cases}
\end{equation*}
The proof of \cref{prop:finite-level two} applies to the simplex polynomial $h$.  For the standard Frobenius lift $\sigma_0(t)=t^{37}$, the target parameter at $t=\tau$ is $u=\tau^{37}$.  The constant coefficient used in that proof is $s_h=\CT(1-uh)=1$, because $\CT(h)=0$.  Substitution gives the first Hasse--Witt scalar $\CT(f_\tau^{36})\equiv4\pmod{37}$ and $37^{-1}\det HW^{(2)}_{\sigma_0}(\tau)\equiv4\pmod{37}$. Thus both Hasse--Witt invariants are units.

By \cite[Lemma~7.2]{BeukersVlasenkoIII}, these Hasse--Witt conditions also hold for the coefficient Frobenius $\sigma_\tau$ of \cref{prop:K3-BR}.

By \cref{prop:K3-BR,lem:low-weight-ordinary}, $H_{\mathrm{var}}$ has slopes $\{0,1,2\}$. At the selected place, the algebraic-cycle projectors described above give the Frobenius-stable decomposition $H_{\mathrm{var}}=H_{\mathrm{tr}}\oplus L_{\mathrm{alg}}$. The algebraic line in \eqref{eq:K3-split} is generated by a divisor class.  Frobenius acts on the finite-dimensional span of divisor classes as $37$ times a finite-order operator.  Hence the algebraic line has slope $1$.  The transcendental factor consequently has slopes $\{0,2\}$.

The vertex coefficients are finite Laurent polynomials. Along the inward ray $-e_1$, let $a_Q^{K3}(t)$ be the coefficient at $-Qe_1$ in the vertex expansion. Direct multinomial extraction gives
\begin{equation}\label{eq:K3-vertex}
 a_Q^{K3}(t)=
 -\sum_{r=0}^{\lfloor(Q-1)/4\rfloor}(-1)^r
 \frac{(Q-1-r)!}{(Q-1-4r)!(r!)^3}\,t^{-(Q-4r)}.
\end{equation}
In particular, \(a_1^{K3}=-t^{-1}\) and \(a_2^{K3}=-t^{-2}\).

\begin{theorem}
For the root $\tau\in\Z_{37}$ specified by \eqref{eq:K3-tau} and $\tau\equiv15\pmod {37}$, $\sigma_{\mathrm{exc},37}(\tau)=\tau$. For every $s\geq1$, the coefficients $a_Q^{K3}(t)$ defined in \eqref{eq:K3-vertex} satisfy
\[a_{37^s}^{K3}(\tau)a_{2\cdot37^{s-1}}^{K3}(\tau) -a_{2\cdot37^s}^{K3}(\tau)a_{37^{s-1}}^{K3}(\tau) \equiv0\pmod{37^{2s}}.\]
\end{theorem}

\begin{proof}
The support of $h$ generates $\mathbb Z^3$, every vertex coefficient is $1$, and the full monomial symmetry group of its simplex is $S_4$, of order $24$.  Thus the relevant lattice index is $1$ and $37\nmid24$, so the additional hypotheses of \cite[Theorem~7.3]{BeukersVlasenkoIII} hold.  The two Hasse--Witt conditions were verified above.  \Cref{prop:K3-BR} gives the comparison map and its Frobenius normalization.  \Cref{lem:low-weight-ordinary} gives the dimension and slopes of the high slope quotient. Equation \eqref{eq:K3-split} and the slope calculation show that $H_{\mathrm{tr}}$ is a rank two subisocrystal with slopes $0$ and $2$. The image of $1/f$ spans $\Fil^2H_{\mathrm{var}}\subset H_{\mathrm{tr}}$. Thus \cref{thm:rank two-criterion} applies with $d=2$ and $E=H_{\mathrm{tr}}$.  The Frobenius lift on the coefficient ring $\sigma_\tau$ fixes $\tau$, so that theorem gives the asserted fixed point.  The congruence follows for every $s\geq1$ by applying the vertex-coefficient conclusion of \cref{thm:rank two-criterion} to the inward ray in \eqref{eq:K3-vertex}.
\end{proof}

\section{A cycle criterion for a motivic projector}
\label{sec:motivic-projector}

We give a cycle criterion for realizing $h^1(E)(-1)$ as a direct summand of the Chow motive of $X_{t_*}$. Such a decomposition would give the rational Hodge decomposition predicted at $t_*$. Let $X=X_{t_*}$, write $f=f_{14,4}$ and $g=f_{14,2}$, and let $E=X_0(14)$ be the elliptic curve associated with $g$ \cite[\S6]{CandelasDeLaOssaElmiVanStraten2020}.

We use Chow motives with rational coefficients in the contravariant convention of \cite[\S\S1.3--1.8]{Scholl1994}, and write $h(X)$ for the Chow motive of $X$. Let $O\in E(\Q)$ be the origin and let $\pi_E^1=\Delta_E-[O\times E]-[E\times O]$. This projector defines $h^1(E)$ and satisfies ${}^t\pi_E^1=\pi_E^1$ \cite[\S3.2]{Scholl1994}.

The following proposition gives construction of  the projector from a correspondence whose action on $H^1(E)(-1)$ is nonzero.

\begin{proposition}
\label{prop:cycle-projector}
Suppose that a cycle $Z\in\mathrm{CH}^2(E\times X)_{\Q}$ defined over $\Q$ induces a nonzero map $Z_*:H^1(E(\C),\Q)(-1)\to H^3(X(\C),\Q)$. Let
\[\Gamma=Z\circ\pi_E^1:h^1(E)(-1)\longrightarrow h(X), \quad \Gamma^\dagger=\pi_E^1\circ{}^tZ: h(X)\longrightarrow h^1(E)(-1).\]
Then $u=\Gamma^\dagger\Gamma$ is invertible, and
\begin{equation}\label{eq:motivic-projector}
 \Pi_g=\Gamma u^{-1}\Gamma^\dagger
 \in\mathrm{CH}^3(X\times X)_{\Q}
\end{equation}
is an idempotent with image isomorphic to $h^1(E)(-1)$. Its image has complexification $H^{2,1}(X)\oplus H^{1,2}(X)$ on $H^3(X(\C),\Q)$ and its kernel has complexification $H^{3,0}(X)\oplus H^{0,3}(X)$. Assuming Dummigan's Conjecture~1.1, the corresponding $\ell$-adic summands are those attached to $g(-1)$ and $f$, respectively.
\end{proposition}

\begin{proof}
By \cref{prop:permutohedral-deformations}, the four Hodge numbers of $H^3(X)$ are all one. The Betti realization of $\Gamma$ is a nonzero morphism of rational Hodge structures. Its source has one-dimensional components of types $(2,1)$ and $(1,2)$. Complex conjugation shows that the maps on these two components are both nonzero. Thus $\Gamma_*$ is injective and its image has complexification $H^{2,1}(X)\oplus H^{1,2}(X)$. The restriction of cup product to this image is nondegenerate, and its orthogonal complement has complexification $H^{3,0}(X)\oplus H^{0,3}(X)$.

Since ${}^t\pi_E^1=\pi_E^1$, the realization of $\Gamma^\dagger$ is the adjoint of $\Gamma_*$ for the cup-product pairings. Hence $u$ has invertible Betti realization and is nonzero. By \cite[Proposition~3.3]{Scholl1994},
\[\operatorname{End}(h^1(E)(-1)) \simeq\operatorname{End}_{\Q}(E)\otimes_{\Z}\Q.\]
Every nonzero endomorphism of $E$ is an isogeny, so $u$ is invertible. Composition of correspondences gives
\[\Pi_g^2 =\Gamma u^{-1}(\Gamma^\dagger\Gamma)u^{-1}\Gamma^\dagger =\Pi_g.\]
The maps $\Gamma$ and $u^{-1}\Gamma^\dagger$ identify its image with $h^1(E)(-1)$. On Betti cohomology, $\Pi_g$ is the orthogonal projector onto the image of $\Gamma_*$, which proves the asserted Hodge decomposition.

Under Dummigan's Conjecture~1.1, \cref{prop:actual-split} gives the direct sum of the representations attached to $g(-1)$ and $f$. Betti--$\ell$-adic comparison shows that the $\ell$-adic realization of $\Gamma$ is nonzero. Its source is the absolutely irreducible representation attached to $g(-1)$, so its image is the unique copy of that representation in $H^3_{\et}(X_{\overline\Q},\Q_\ell)$. The complementary summand is the representation attached to $f$.
\end{proof}

The projector gives compatible splittings in rational Betti, de Rham, and $\ell$-adic cohomology. After spreading out $X$, the projector, and its idempotent relation, it also gives an integral crystalline splitting at every prime in the resulting smooth proper spread where the projector's denominators are invertible.

\subsection{Constraints on a correspondence}

We first consider the correspondences given by the known elliptic ruled surfaces on the Hulek--Verrill small-resolution cover of \cref{prop:birational-H3}. Their underlying elliptic curves are isogenous to $E$. Let $W$ be the subspace of the cover's rational Betti cohomology $H^3$ generated by their cylinder maps, and let $W^\perp$ denote its orthogonal complement for cup product. The quotient $W/(W\cap W^\perp)$ is the direct sum of four copies of $H^1(E)(-1)$ \cite[Propositions~4.4--4.5]{Dummigan2026}. The calculation in \cite[proof of Proposition~4.6]{Dummigan2026} gives $W^G\subset W\cap W^\perp$ for the order-ten group $G$. Thus cup product vanishes on $W^G$.

After composition with isogeny correspondences, their cylinder maps have source $H^1(E)(-1)$. Averaging over $G$ gives maps to the cohomology of the quotient $X_{t_*}^{\mathrm{sm}}$. Under Dummigan's Conjecture~1.1, the argument of \cref{prop:actual-split} shows that any such nonzero map has image equal to the nondegenerate summand attached to $g(-1)$. Pullback to the cover multiplies cup product by $10$, but this image would lie in $W^G$, where cup product vanishes. Therefore these averaged maps are zero. They remain zero after composition with a correspondence from $X_{t_*}^{\mathrm{sm}}$ to $X$.

The numerical period relations in \cite[\S4.2 and Table~7]{CandelasDeLaOssaElmiVanStraten2020} suggest a further constraint on an integral cylinder map. In the coordinates of that table with $\kappa=1$, let \(L_{\mathrm{out}}=\langle v_1,v_2\rangle_{\Z}\) and \(L_g=\langle w_1,w_2\rangle_{\Z}\) be lattices where
\[v_1=(4,-15,-5,0)^t,\quad v_2=(0,0,2,1)^t, \quad w_1=(3,-6,0,1)^t,\quad w_2=(1,-2,-5,-1)^t.\]
For the lattice argument below, assume the predicted integral symplectic identification of $H^3(X(\C),\Z)/\mathrm{tors}$ with $\Z^4$, under which $L_g\otimes_{\Z}\C$ is $H^{2,1}(X)\oplus H^{1,2}(X)$.

The two lattices are orthogonal, and each has alternating pairing of absolute value $7$ in its displayed basis. The matrix with columns $v_1,v_2,w_1,w_2$ has determinant $-49$, so $L_{\mathrm{out}}\oplus L_g$ has index $49$ in $\Z^4$. Moreover, $L_g$ is saturated: the maximal minors of $(w_1,w_2)$ include $-4$ and $5$, and hence have greatest common divisor $1$.

Let $B$ be an elliptic curve and let
\[c:H^1(B(\C),\Z)(-1)\to H^3(X(\C),\Z)/\mathrm{tors}\]
be an integral cylinder map induced by an algebraic correspondence. Define its pairing multiplier $m(c)$ by
\[\langle c(\alpha),c(\beta)\rangle_X =m(c)\langle\alpha,\beta\rangle_B.\]
If $c$ is nonzero, its image lies in $L_g$ and has finite index $n=[L_g:\operatorname{im}c]$. In integral bases, the matrix of $c$ has determinant $\pm n$, so $m(c)=\pm7n$. In particular, a cylinder map with nonzero saturated image has multiplier $\pm7$. Replacing $c$ by a rational multiple $ac$ multiplies its pairing multiplier by $a^2$.

Suppose that $i:S\hookrightarrow X$ is a smooth divisor and $\pi:S\to B$ is a $\PP^1$-bundle. Its cylinder map is $c=i_*\pi^*$. The self-intersection formula and adjunction give
\[\langle i_*\pi^*\alpha,i_*\pi^*\beta\rangle_X =\int_S\pi^*(\alpha\wedge\beta)c_1(N_{S/X}) =-2\int_B\alpha\wedge\beta,\]
because $N_{S/X}\simeq K_S$ and $\deg(K_S|_{\PP^1})=-2$; see \cite[\S6.3]{Fulton1998} for the self-intersection formula. Thus $c$ is nonzero and has multiplier $-2$, contrary to the divisibility by $7$. Under the stated period prediction, $X$ therefore contains no smooth divisor that is a $\PP^1$-bundle over an elliptic curve.

\subsection{A possible construction}

We can look for such a correspondence using a family of curves on $X$ parameterized by an elliptic curve. Suppose that the Hilbert scheme of curves on $X$ has a smooth geometrically integral genus one component $B$ defined over $\Q$. Choose a point $O_B\in B(\Q)$ as origin and an isogeny $\varphi:B\to E$ defined over $\Q$. Let $\mathcal U\subset B\times X$ be the universal curve, with projections $q:\mathcal U\to B$ and $e:\mathcal U\to X$. Let $\pi_B^1$ be the Chow--K\"unneth projector defined by $O_B$, and let
\[\Gamma_B=[\mathcal U]\circ\pi_B^1: h^1(B)(-1)\longrightarrow h(X), \quad \Gamma=\Gamma_B\circ\varphi^*: h^1(E)(-1)\longrightarrow h(X).\]
Since $\varphi^*$ is an isomorphism on rational $H^1$, a nonzero Betti realization of $\Gamma_B$ gives a nonzero realization of $\Gamma$. Then \cref{prop:cycle-projector} gives the motivic projector. Computing $\Gamma^\dagger\Gamma$ makes the formula for this projector explicit.

Let $p_{ij}$ denote the projections from $B\times X\times B$ onto the indicated factors. The composition of correspondences is
\[u_B:=\Gamma_B^\dagger\Gamma_B =\pi_B^1\circ(p_{13})_* \bigl(p_{12}^*[\mathcal U]\cdot p_{23}^*[{}^t\mathcal U]\bigr)\circ\pi_B^1 \in\operatorname{End}(h^1(B)(-1)),\]
where $\Gamma_B^\dagger=\pi_B^1\circ{}^t[\mathcal U]$ \cite[\S1.3]{Scholl1994}. The corresponding endomorphism of $h^1(E)(-1)$ is $u=\varphi_*u_B\varphi^*$.

Let $c_B=(\Gamma_B)_*$ and $c_E=c_B\circ\varphi^*$ be the integral cylinder maps on $H^1$. The projection formula gives
\[m(c_E)=(\deg\varphi)\,m(c_B).\]
We note thath if $q$ is a $\PP^1$-bundle and $e$ is a closed immersion, then $m(c_B)=-2$ and $m(c_E)=-2\deg\varphi$. The lattice constraint above excludes this case under the period prediction.

\section{Further questions}
\label{sec:future}

\subsection{Bounded residue comparisons in other families}

We would like to apply \cref{thm:rank two-criterion} to other Calabi--Yau Laurent families. Under the geometric hypotheses of \cref{thm:general-bounded-residue}, we need to verify two statements: restriction identifies the selected crystalline summand with the lowest weight part of the open cohomology, and the cyclic residues span this summand. What geometric or combinatorial conditions on the Laurent polynomial ensure these properties?

We can also ask for comparisons beyond the level two quotient. For every $r\geq1$, \cref{lem:slope} gives $\rho(\cF_rM)\subset H_{\varphi,\leq d-r}$. The induced map on the quotients is surjective. If these quotients have the same dimension, it is an isomorphism $(M/\cF_rM)[1/p]\simeq H/H_{\varphi,\leq d-r}$. We ask when higher Hasse--Witt conditions give this equality of dimensions, and which cohomological splittings force congruences among the Laurent coefficients.

\subsection{A modular correspondence}
Let $E=X_0(14)$ be the elliptic curve attached to $f_{14,2}$, as in \cref{sec:motivic-projector}. Dummigan's Conjecture~1.1 predicts that $H^3_{\et}(X_{t_*,\overline{\Q}},\Q_\ell)$ contains a summand isomorphic to $H^1_{\et}(E_{\overline{\Q}},\Q_\ell)(-1)$ \cite[Conjecture~1.1]{Dummigan2026}. We ask whether an algebraic correspondence realizes this factor.

\begin{conjecture}
\label{conj:A4-modular-cycle}
There is a cycle $Z\in\mathrm{CH}^2(E\times X_{t_*})_{\Q}$ whose Betti realization induces a nonzero map $H^1(E)(-1)\to H^3(X_{t_*})$.
\end{conjecture}

By \cref{prop:cycle-projector}, such a cycle gives an idempotent whose image on $H^3$ is the factor attached to $f_{14,2}(-1)$. Its complement is the factor attached to $f_{14,4}$. Thus \cref{conj:A4-modular-cycle} would prove the predicted rational Hodge splitting.

\subsection{Other predicted rank two parameters}

The conjugate parameters $t_\pm=33\pm8\sqrt{17}$ in the Hulek--Verrill family are further predicted rank two attractors. Dummigan's Conjecture~1.2 gives their expected modular decompositions over $\Q(\sqrt{17})$ \cite[Conjecture~1.2]{Dummigan2026}.

We would like to apply \cref{thm:rank two-criterion} at these parameters. Assuming the predicted decompositions, one need to construct the bounded comparison in \cref{thm:general-bounded-residue} and verify the Dwork Hasse--Witt conditions and the slope conditions on the factor containing the holomorphic line. At places of residue degree one, the criterion would then give excellent fixedness and the corresponding supercongruences.

At a place $v$ of residue degree $f_v>1$, the excellent Frobenius carries the parameter through its Frobenius orbit. The corresponding question is whether $\sigma_{\mathrm{exc},v}^{f_v}$ fixes the parameter. This requires the comparisons and Cartier relations along the orbit; see \cref{prop:finite-frobenius-orbit}.

\subsection{Algebraicity of rank two attractor parameters}
Our arithmetic questions concern parameters defined over number fields. The two parameters considered in this paper are algebraic: $t_*=-1/7$ is rational, and for the mirror quartic pencil, the modular description identifies the relevant parameters with CM values. This raises a question about the original complex geometry, which Frits Beukers also asked more generally: must an isolated attractor be algebraic?

Moore’s attractor conjecture predicts algebraicity for isolated attractor points  \cite[Conjecture~8.2.2]{Moore1998}.  For the Hermitian symmetric Calabi--Yau variations studied by Lam, the attractor points are rank two and are CM points \cite[Theorem~3.6.3 and \S3.9]{LamAttractorConjecture}.  Lam and Tripathy consider the attractor condition that there exists a nonzero class $\gamma\in H^d(X,\Z)$ orthogonal to $H^{d-1,1}(X)$. By contrast, they prove that assuming a case of the Zilber--Pink conjecture, the algebraicity fails in every odd dimension other than $1,3,5,$ and $9$ \cite[Theorem~1.1.3]{LamTripathyAttractors}.  Their result does not settle the algebraicity question for rank two attractors.  The rank two condition makes $H^{d,0}\oplus H^{0,d}$ the complexification of a rational Hodge summand and is therefore expected to have a motivic origin.  Baldi, Klingler, and Ullmo interpret such points in one-parameter Calabi--Yau variations as atypical special points and conjecture finiteness in the setting they consider \cite[\S5.5]{BaldiKlinglerUllmo}.

 We formulate the rank two case of Beukers's question as follows.

\vspace{0.5em}
\noindent\textbf{Question.}
Let $\mathscr X\to S$ be a non-isotrivial Calabi--Yau family defined over a number field.  If $t\in S(\C)$ is an isolated rank two attractor point, must $t$ belong to $S(\overline{\Q})$? 
 
 A positive answer to the question would allow us to formulate the finite-place comparison questions for every isolated rank two attractor.

\subsection{Distribution of ordinary and canonical primes}
For the \(A_4\) fiber, our results apply at two-Dwork-ordinary primes. We ask how often this condition holds and whether the two modular forms predicted by Dummigan determine it. For a prime $p>5$, $p\neq7$, let $a_p^{(4)}=a_p(f_{14,4})$ and $a_p^{(2)}=a_p(f_{14,2})$. Dummigan's Conjecture~1.1 implies that
\begin{equation}\label{eq:A4-local-modular-polynomial}
\det\!\left(1-T\operatorname{Frob}_p
\mid H^3_{\et}(X_{t_*,\overline{\Q}},\Q_\ell)\right)
=
(1-a_p^{(4)}T+p^3T^2)(1-pa_p^{(2)}T+p^3T^2),
\end{equation}
where $\ell\neq p$ and $\operatorname{Frob}_p$ denotes geometric Frobenius. The first factor has Newton slopes $\{0,3\}$ precisely when $a_p^{(4)}\not\equiv0\pmod p$, and the second has slopes $\{1,2\}$ precisely when $a_p^{(2)}\not\equiv0\pmod p$.  It follows from \cref{thm:main} that every two-Dwork-ordinary prime is ordinary for both modular forms.

We first compare the first Hasse–Witt scalar with the weight four Fourier coefficient.

\begin{proposition}
\label{prop:first-HW-modular}
Assume Dummigan's Conjecture~1.1. For every prime $p>5$, $p\neq7$, one has
\[h_p\equiv a_p^{(4)}\pmod p.\]
\end{proposition}

\begin{proof}
By \eqref{eq:A4-Fil-split}, projection induces an isomorphism $H_W/\Fil^1H_W\simeq H_{W,f}/\Fil^3H_{W,f}$. After reduction modulo $p$, the left side is $H^3(X_0,\mathcal O_{X_0})$.  Mazur's divisibility \cite{Mazur1973} gives $\Phi(\Fil^1H_W)\subset pH_W$, so the reduction of Frobenius induces the Hasse–Witt map whose scalar is $h_p$.  On the weight four summand, $\Phi(\Fil^3H_{W,f})\subset p^3H_{W,f}$, so the trace of Frobenius on \(H_{W,f}\) reduces to the scalar on \(H_{W,f}/\operatorname{Fil}^3H_{W,f}\). The first factor of \eqref{eq:A4-local-modular-polynomial} gives $\operatorname{tr}(\Phi\mid H_{W,f})=a_p^{(4)}$. This proves the congruence.
\end{proof}

At a two-Dwork-ordinary prime, let $\alpha_p$ be the unique $p$-adic unit root of $X^2-a_p^{(4)}X+p^3$. Under \eqref{eq:Q2-high}, the Cartier eigenline generated by the class of $1/f$ corresponds to the slope three line of $D_f$. Since Cartier corresponds to $p^3\varphi^{-1}$, one has $\lambda_0(t_*,t_*)=\alpha_p$. Consequently, for every $\mathbf m\in C_{\mathbf b}\cap\Z^n$ and every $s\geq1$, Proposition~5.11 of \cite{BeukersVlasenkoIII} gives
\begin{equation*}
 a_{p^s\mathbf m}(t_*)\equiv
 \alpha_p a_{p^{s-1}\mathbf m}(t_*)\pmod{p^{2s}}.
\end{equation*}

We next compare the second Hasse–Witt invariant with the two modular forms. Fix the basis $(1/f,\theta(1/f))$ and the coefficient Frobenius $\sigma_0(t)=t^p$, and let
\[h_p^{(2)}= \left(p^{-1}\det HW^{(2)}_{\sigma_0}(t_*)\right)\bmod p\in\F_p.\]
The reduction of this normalized determinant is independent of the choice of coefficient Frobenius \cite[Lemma~7.2]{BeukersVlasenkoIII}.  The modular decomposition suggests the following integral comparison.

\begin{conjecture}
\label{conj:A4-modular-HW}
For every prime $p>5$, $p\neq7$, one has \(h_p^{(2)}=a_p^{(4)}a_p^{(2)} \in \F_p.\)
\end{conjecture}

Together with \cref{prop:first-HW-modular}, this conjecture says that $p$ is two-Dwork-ordinary at $t_*$ if and only if both $f_{14,4}$ and $f_{14,2}$ are ordinary at $p$.  On the two-Dwork-ordinary locus, \eqref{eq:Q2-high} and the identity $C=p^3\varphi^{-1}$ already give the displayed equality: the two eigenvalues of $C$ are the unit root of the weight four factor and $p$ times the unit root of the weight two factor. The conjecture extends this equality to the remaining good primes and therefore predicts the converse implication between the two ordinary conditions.  Direct evaluation of \cref{prop:finite-level two} verifies the conjecture for $p=11,13,17,19,29,31,113$, the good primes listed in \cite[the table in the proof of Proposition~4.6]{Dummigan2026}. Extending the verification to every good prime requires control of the integral level two lattice at primes where the second Hasse--Witt condition fails.

\begin{conjecture}
\label{conj:A4-density}
The set
\[\{p>5:p\neq7\text{ and }a_p^{(4)}a_p^{(2)}\not\equiv0\pmod p\}\]
has natural density one.
\end{conjecture}

\Cref{conj:A4-density} includes the ordinary-prime problem for the weight four form $f_{14,4}$; even the infinitude of ordinary primes for a fixed newform of weight greater than two is open \cite[Remark~2.3]{LongoVigni2021}. Assuming Dummigan's Conjecture~1.1 and the two conjectures above, the two-Dwork-ordinary primes would have density one. The main theorem and \cref{thm:A4-W2-comparison} would then show that the excellent Frobenius fixes the parameter; they would also give the supercongruences and the canonical lift modulo $p^2$ at this set of primes.  If \cref{conj:full-canonical-comparison} also holds under its stated hypotheses, then $t_*$ parametrizes the Brantner--Taelman canonical lift at all Witt orders at this set of primes.

We can ask the fixed-point question for any algebraic rank two attractor parameter. Let $t$ be a parameter defined over a number field. At every finite place $v$ where the chosen integral model, the parameter, its residue disk, and the excellent Frobenius are defined over $W(k(v))$, write $k(v)=\F_{p^{f_v}}$ and let $t_v$ be the resulting $W(k(v))$-point. The excellent Frobenius lifts the $p$-power map, so it carries the residue disk of $\bar t_v$ to that of $\bar t_v^p$. Since the $f_v$th power of Witt Frobenius is the identity on $W(k(v))$, the return map on the original disk is $\sigma_{\mathrm{exc},v}^{f_v}$. Define
\[\mathcal P_{\mathrm{exc}}(t)=
 \left\{v:\begin{array}{l}
 \text{the excellent Frobenius is defined at $v$, and}\\
 \sigma_{\mathrm{exc},v}^{f_v}(t_v)=t_v
 \end{array}\right\}.\]
When the residue of $t_v$ is fixed by the $p$-power map, the corresponding one-step condition is $\sigma_{\mathrm{exc},v}(t_v)=t_v$.

\vspace{0.5em}
\noindent\textbf{Question.}
If $t$ is a rank two attractor parameter defined over a number field, does $\mathcal P_{\mathrm{exc}}(t)$ have positive density?

In the opposite direction, we ask whether canonical lifting at many finite places forces a rational Hodge splitting in characteristic zero.

\vspace{0.5em}
\noindent\textbf{Question.}
Let $t$ be a parameter of a Calabi--Yau family over a number field. Suppose that at a set of good places $v$ of positive density, the lift defined by $t_v$ is the Brantner--Taelman canonical formal lift. Must the characteristic zero fiber admit a nontrivial rational Hodge decomposition, or a motivic projector realizing such a decomposition?

This asks whether a condition on lifts at finite places can detect the Hodge specialness studied through Hodge loci and atypical intersections \cite[\S5.5]{BaldiKlinglerUllmo}.

\appendix

\section{Permutohedral geometry of the good pair}
\label{app:good-pair}

Let $I=\{0,\ldots,5\}$ and $N_I=\Z^I/\Z(1,\ldots,1)$. For every nonempty proper subset $S\subset I$, write $u_S=\sum_{i\in S}\bar e_i\in N_I$.  The $A_5$ braid fan has rays $\mathbb R_{\geq0}u_S$ and cones $\langle u_{S_1},\ldots,u_{S_r}\rangle$, where $S_1\subsetneq\cdots\subsetneq S_r$. It is the normal fan of the permutohedron, hence is complete and projective. A maximal chain is determined by an ordering $(i_0,\ldots,i_5)$ of $I$.  Taking $i_5$ as affine origin, its five ray generators are the successive partial sums of the basis $\bar e_{i_0},\ldots,\bar e_{i_4}$; the corresponding matrix is triangular with diagonal entries $1$.  Thus every cone is unimodular, over any base ring.  We denote the resulting smooth projective toric fivefold by $V$ and the divisor of the ray $u_S$ by $E_S$.

For $i\in I$, define the two opposite simplex line bundles
\begin{equation}\label{eq:perm-simplex-bundles}
 \mathcal A_i=\mathcal O_V\!\left(\sum_{i\in S}E_S\right),
 \quad
 \mathcal B_i=\mathcal O_V\!\left(\sum_{i\notin S}E_S\right).
\end{equation}
The two opposite maps from the braid fan to the simplex fan give toric blow-downs $\pi_{\mathcal A},\pi_{\mathcal B}:V\to\PP^5$ satisfying
\begin{equation}\label{eq:perm-canonical}
 \mathcal A_i=\pi_{\mathcal A}^*\mathcal O_{\PP^5}(1),
 \quad
 \mathcal B_i=\pi_{\mathcal B}^*\mathcal O_{\PP^5}(1),
 \quad
 \mathcal A_i\otimes\mathcal B_i\simeq\omega_V^{-1}.
\end{equation}
Each blow-down is a composition of blow-ups of coordinate linear spaces, so $R\pi_*\mathcal O_V=\mathcal O_{\PP^5}$.  The projection formula therefore gives
\begin{equation}\label{eq:inverse-simplex-acyclic}
 R\Gamma(V,\mathcal A_i^{-1})=
 R\Gamma(V,\mathcal B_i^{-1})=0,
 \quad
 H^{>0}(V,\mathcal A_i)=H^{>0}(V,\mathcal B_i)=0.
\end{equation}

We record the cohomological consequences needed in \cref{prop:permutohedral-deformations,prop:A4-W2-cohomology}.

\begin{proposition}
\label{prop:permutohedral-vanishing}
Let $\mathcal L_1=\mathcal A_0$ and $\mathcal L_2=\mathcal B_0$.  For every boundary divisor $E_S$ one has
\[\begin{aligned}
 &h^0(V,\mathcal O_V(E_S))=1,
 &&H^{q>0}(V,\mathcal O_V(E_S))=0,\\
 &R\Gamma(V,\mathcal O_V(-E_S))=0,
 &&R\Gamma(V,\mathcal O_V(K_V+E_S))=0,\\
 &H^q(V,\mathcal O_V(E_S)\otimes\mathcal L_i^{-1})=0
 &&(0\leq q\leq3,\ i=1,2),\\
 &R\Gamma(V,\mathcal L_1\otimes\mathcal L_2^{-1})=0,
 &&R\Gamma(V,\mathcal L_2\otimes\mathcal L_1^{-1})=0.
\end{aligned}\]
All these statements hold over an arbitrary field.
\end{proposition}

\begin{proof}
For every nonempty $J\subset I$, let $N_J=\Z^J/\Z(1,\ldots,1)$.  Let $V_J$ denote the permutohedral toric variety with cocharacter lattice $N_J$.  We use $\mathcal A_i,\mathcal B_i$ for the analogous simplex line bundles on these varieties. The star of $u_S$ is the product of the two smaller braid fans, hence
\begin{equation}\label{eq:perm-boundary-product}
 E_S\simeq V_S\times V_{S^c}.
\end{equation}
Suppose first that $0\notin S$ and choose $j\in S$.  Restricting the Cartier data in \eqref{eq:perm-simplex-bundles} gives
\begin{equation}\label{eq:perm-boundary-restrictions}
\begin{aligned}
 \mathcal A_0|_{E_S}&=\mathcal O\boxtimes\mathcal A_0,
 &\mathcal B_0|_{E_S}&=\mathcal B_j\boxtimes\mathcal O,\\
 \mathcal O_V(E_S)|_{E_S}
   &=\mathcal A_j^{-1}\boxtimes\mathcal B_0^{-1}.
\end{aligned}
\end{equation}
For example, the last identity follows by subtracting the local Cartier character with coordinates $1$ at $0$, $-1$ at $j$, and zero elsewhere. The case $0\in S$ follows by complementing $S$ and interchanging $\mathcal A$ and $\mathcal B$.

At least one factor in \eqref{eq:perm-boundary-product} has positive dimension, so \eqref{eq:inverse-simplex-acyclic} and K\"unneth show that the normal bundle $\mathcal O_{E_S}(E_S)$ is acyclic.  The two divisor sequences for $E_S$ then give the first two displayed assertions; the assertion involving $K_V+E_S$ follows by Serre duality.

Equations \eqref{eq:perm-canonical} and \eqref{eq:perm-boundary-restrictions} also give
\[\begin{aligned}
 (E_S-\mathcal L_1)|_{E_S}
   &=\mathcal A_j^{-1}\boxtimes\omega_{V_{S^c}},\\
 (E_S-\mathcal L_2)|_{E_S}
   &=\omega_{V_S}\boxtimes\mathcal B_0^{-1}.
\end{aligned}\]
These bundles are acyclic unless the factor carrying the negative simplex bundle is a point.  In that extremal case the only cohomology is $H^4\simeq k$.  Thus their cohomology vanishes in degrees $0$ through $3$; the divisor sequence and \eqref{eq:inverse-simplex-acyclic} give the third displayed assertion of the proposition.

Finally, $\mathcal L_1\otimes\mathcal L_2^{-1}=K_V+2\mathcal L_1$ and $\mathcal L_2\otimes\mathcal L_1^{-1}=K_V+2\mathcal L_2$. By Serre duality, their cohomology is dual to that of $-2\mathcal L_1$ and $-2\mathcal L_2$.  Projection formula identifies the latter with the cohomology of $\mathcal O_{\PP^5}(-2)$.  This cohomology vanishes in every degree.  This proves the last line.
\end{proof}

\section{A finite level two Hasse--Witt formula}
\label{app:level two}

For the $A_4$ family $f_t=1-tg$ of \eqref{eq:A4-g}, we give the finite Hasse--Witt calculation used in \cref{sec:A4-W2}. Let $\mathbf v=e_1-e_0$ be a root of $A_4$. Statements below involving $\lambda_1$ or $\sigma_{\mathrm{exc}}$ are made on the locus where the first two Hasse--Witt conditions hold. For $k,Q\geq0$ let
\[c_k=\CT(g^k)= \sum_{r_0+\cdots+r_4=k} \left(\frac{k!}{r_0!r_1!r_2!r_3!r_4!}\right)^2\]
and
\[d_{k,Q}=[\mathbf x^{Q\mathbf v}]g^k= \sum_{r_0+\cdots+r_4=k-Q} \frac{k!^2}{r_0!(r_0+Q)!r_1!(r_1+Q)! (r_2!r_3!r_4!)^2},\]
with $d_{k,Q}=0$ for $k<Q$.  Define
\[A_m(t)=\sum_{k=0}^m\binom{m}{k}(-t)^kc_k=\CT(f_t^m), \quad B_{m,Q}(t)=\sum_{k=Q}^m\binom{m}{k}(-t)^kd_{k,Q} =[\mathbf x^{Q\mathbf v}]f_t^m.\]

The following formula is used in the $W_2$ comparison. The same argument gives the calculation at $p=37$.

\begin{proposition}
\label{prop:finite-level two}
For the standard lift $\sigma_0(t)=t^p$, let $u=t^p$ and $s_u=1-5u$. The level two Hasse--Witt matrix in the basis $(1/f,1/f^2)$ is
\[HW^{(2)}_{\sigma_0}(t)=
 \begin{pmatrix}\alpha&\beta\\ \alpha_2&\beta_2\end{pmatrix}\]
modulo $p^2$. The entries are defined by
\begin{align*}
 C_0&=2s_uA_{p-1}-A_{2p-1},&
 C_{\mathbf v}&=-2u A_{p-1}-B_{2p-1,p},\\
 D_0&=2s_uA_{p-2}-A_{2p-2},&
 D_{\mathbf v}&=-2u A_{p-2}-B_{2p-2,p},
\end{align*}
and
\[\alpha=-C_{\mathbf v}/u,\quad \beta=C_0-s_u\alpha, \quad \alpha_2=-D_{\mathbf v}/u,\quad \beta_2=D_0-s_u\alpha_2.\]
In the basis $(1/f,\theta(1/f))$, the upper-right coefficient is
\begin{equation}\label{eq:lambda-finite}
 \lambda_1^{(0)}(t)\equiv\beta
 =-A_{2p-1}(t)-\frac{1-5t^p}{t^p}B_{2p-1,p}(t)
 \pmod{p^2}.
\end{equation}
\end{proposition}

\begin{proof}
By \cite[Definition~5.1 and equation~(10)]{BeukersVlasenkoIII}, the numerator used for the class $1/f=f/f^2$ is $2f^{p-1}f^{\sigma_0}(\mathbf x^p)-f^{2p-1}$. Its constant and $\mathbf x^{p\mathbf v}$ coefficients are $C_0$ and $C_{\mathbf v}$. Indeed, let $\Delta_{A_4}$ denote the $A_4$ root polytope.  It has $0$ as its unique interior lattice point; if $\mathbf w\neq\mathbf v$ are support exponents, then $p(\mathbf v-\mathbf w)\notin(p-1)\Delta_{A_4}$. This excludes every other contribution. Complete symmetry forces the numerator to have the form $\alpha f^{\sigma_0}+\beta$ after applying Cartier.  Comparing its vertex and constant coefficients gives the displayed formulas.  Repeating the argument for the class represented by $1/f^2$ gives the second row.  Finally $1/f^2=1/f+\theta(1/f)$, and substitution gives \eqref{eq:lambda-finite}.
\end{proof}

Allowing an arbitrary Frobenius target gives the form used in the obstruction comparison.

\begin{corollary}
Let $u$ be a unit and assume $u\equiv t^p\pmod p$.  Let the Frobenius lift on the coefficient ring send the source parameter $t$ to the target $u$. Then
\begin{equation}\label{eq:lambda-finite-target}
 \lambda_1(t,u)\equiv
 -A_{2p-1}(t)-\frac{1-5u}{u}B_{2p-1,p}(t)
 \pmod{p^2}.
\end{equation}
\end{corollary}

\begin{proof}
Replace $f^{\sigma_0}(x^p)$ in the proof of \cref{prop:finite-level two} by $f_u(x^p)$.  The numerator becomes $2f_t^{p-1}f_u(x^p)-f_t^{2p-1}$. Its constant and vertex coefficients give \eqref{eq:lambda-finite-target} by the same comparison.
\end{proof}

Let $F_p(t)=\sum_{k=0}^{p-1}c_kt^k$.  On a unit residue disk on which the first two Hasse--Witt conditions hold, $F_p(t)$ is invertible.  If $\sigma_{\mathrm{exc}}(t)=t^p(1+h(t))$, the reverse-series construction of \cite[proof of Theorem~7.15]{BeukersVlasenkoIII} gives the first correction
\begin{equation}\label{eq:second-order-correction}
 \frac{h(t)}p\equiv
 \frac{\lambda_1^{(0)}(t)}pF_p(t)^{-1}\pmod p.
\end{equation}
Equivalently, if $\sigma_{\mathrm{exc}}(t)=t^p+pg_p(t)\pmod{p^2}$, then $g_p(t)\equiv t^p\lambda_1^{(0)}(t)F_p(t)^{-1}/p\pmod p$. Equations \eqref{eq:lambda-finite} and \eqref{eq:second-order-correction} express the first $p$-adic correction using only finite sums.
\section{A finite formula for the vertex coefficients}
\label{app:vertex}

The coefficients in \eqref{eq:aQ-def-intro} are given by the following finite formula. Let
\begin{equation*}
 \delta_n=
 \sum_{a+b+c=n}\left(\frac{n!}{a!b!c!}\right)^2.
\end{equation*}
The first several values are $1,3,15,93,639,\ldots$. For $0\leq k<Q$, define
\begin{equation*}
 c_{Q,k}=
 \sum_{r=0}^{Q-k-1}
 (-1)^{Q-k-1-r}\delta_r
 \binom{k+r}{k}^2
 \binom{Q+k+r}{Q-k-1-r}.
\end{equation*}
Then the $A_4$ vertex coefficient is
\begin{equation}\label{eq:aQ-finite}
 a_Q(t)=-\sum_{k=0}^{Q-1}c_{Q,k}t^{-k-1}.
\end{equation}
Let $z=y_0/y_1$, and for $i=2,3,4$ let $u_i=y_i/y_1$, $v_i=y_0/y_i$, so $u_iv_i=z$.  With $U=z+u_2+u_3+u_4$ and $V=z+v_2+v_3+v_4$, one has $zg=(1+U)(1+V)$ and the vertex expansion
\[\frac1{f_t}=-\sum_{k\geq0} z^{k+1}t^{-k-1}(1+U)^{-k-1}(1+V)^{-k-1}.\]
To obtain $z^Q$, choose paired copies of $u_i,v_i$ with total multiplicity $r$, and unpaired copies of $z$ with total multiplicity $Q-k-1-r$.  The paired choices contribute $\delta_r$. The binomial expansion contributes the sign $(-1)^{Q-k-1-r}$. Vandermonde's identity gives the remaining factor $\binom{k+r}{k}^2\binom{Q+k+r}{Q-k-1-r}$.  This proves \eqref{eq:aQ-finite}; applying $\theta=t\,d/dt$ also gives $\theta a_Q(t)=\sum_{k=0}^{Q-1}(k+1)c_{Q,k}t^{-k-1}$.

For $n\geq2$, the integers $\delta_n$ satisfy
\begin{equation}\label{eq:delta-rec}
 n^2\delta_n-\bigl(10n^2-10n+3\bigr)\delta_{n-1}
 +9(n-1)^2\delta_{n-2}=0.
\end{equation}
To prove this recurrence, let $F_{m,j}=\binom mj^2\binom{2j}{j}$, with $F_{m,j}=0$ for $j>m$.  Then $\delta_m=\sum_{j=0}^mF_{m,j}$. For $n\geq2$ and $0\leq j\leq n$,
\[\begin{aligned}
 \sum_{r=0}^j\!\Bigl[n^2F_{n,r}
  &-(10n^2-10n+3)F_{n-1,r}
  +9(n-1)^2F_{n-2,r}\Bigr]\\
  &=\frac{2(j-n)^2(2j+1)(3j-4n+3)}{n^2}F_{n,j}.
 \end{aligned}\]
Taking $j=n$ proves the recurrence. In particular, $a_1(t)=-t^{-1}$ and $a_2(t)=-t^{-1}-t^{-2}$. These are the values used in \eqref{eq:M0}.  Formula \eqref{eq:aQ-finite} makes every congruence in \cref{prop:off-diagonal-congruence} a finite integer calculation.

\bibliographystyle{alpha}
\bibliography{references.bib}

@article{AchingerZdanowicz2021,
  author  = {Achinger, Piotr and Zdanowicz, Maciej},
  title   = {Serre--Tate theory for {Calabi--Yau} varieties},
  journal = {Journal f\"ur die reine und angewandte Mathematik},
  volume  = {780},
  year    = {2021},
  pages   = {139--196},
  eprint  = {1807.11295},
  archivePrefix = {arXiv}
}

@book{BerthelotOgus1978,
  author    = {Berthelot, Pierre and Ogus, Arthur},
  title     = {Notes on Crystalline Cohomology},
  series    = {Mathematical Notes},
  volume    = {21},
  publisher = {Princeton University Press and University of Tokyo Press},
  year      = {1978}
}

@article{BeukersVlasenkoI,
  author  = {Beukers, Frits and Vlasenko, Masha},
  title   = {Dwork crystals {I}},
  journal = {International Mathematics Research Notices},
  year    = {2021},
  volume  = {2021},
  number  = {12},
  pages   = {8807--8844},
  doi     = {10.1093/imrn/rnaa119},
  eprint  = {1903.11155},
  archivePrefix = {arXiv}
}

@article{BeukersVlasenkoIII,
  author  = {Beukers, Frits and Vlasenko, Masha},
  title   = {Dwork crystals {III}: from excellent {F}robenius lifts towards supercongruences},
  journal = {International Mathematics Research Notices},
  year    = {2023},
  volume  = {2023},
  number  = {23},
  pages   = {20433--20483},
  doi     = {10.1093/imrn/rnad101},
  eprint  = {2105.14841},
  archivePrefix = {arXiv}
}

@misc{BrantnerTaelman2024,
  author  = {Brantner, Lukas and Taelman, Lenny},
  title   = {Deformations and lifts of {Calabi--Yau} varieties in characteristic $p$},
  year    = {2024},
  eprint  = {2407.09256},
  archivePrefix = {arXiv},
  primaryClass = {math.AG}
}

@article{CaisLiu2019,
  author  = {Cais, Bryden and Liu, Tong},
  title   = {Breuil--Kisin modules via crystalline cohomology},
  journal = {Transactions of the American Mathematical Society},
  volume  = {371},
  number  = {2},
  year    = {2019},
  pages   = {1199--1230},
  doi     = {10.1090/tran/7280},
  eprint  = {1610.09706},
  archivePrefix = {arXiv}
}

@article{CaisLiuCorrigendum2020,
  author  = {Cais, Bryden and Liu, Tong},
  title   = {Corrigendum to ``Breuil--Kisin modules via crystalline cohomology''},
  journal = {Transactions of the American Mathematical Society},
  volume  = {373},
  number  = {3},
  year    = {2020},
  pages   = {2251--2252},
  doi     = {10.1090/tran/7894}
}

@article{CandelasDeLaOssaElmiVanStraten2020,
  author  = {Candelas, Philip and de la Ossa, Xenia and Elmi, Mohamed and van Straten, Duco},
  title   = {A one parameter family of {Calabi--Yau} manifolds with attractor points of rank two},
  journal = {Journal of High Energy Physics},
  year    = {2020},
  volume  = {2020},
  number  = {10},
  pages   = {202},
  doi     = {10.1007/JHEP10(2020)202},
  eprint  = {1912.06146},
  archivePrefix = {arXiv}
}

@article{CandelasDeLaOssaKuuselaMcGovern2023,
  author  = {Candelas, Philip and de la Ossa, Xenia and Kuusela, Pyry and McGovern, Joseph},
  title   = {Mirror symmetry for five-parameter {Hulek--Verrill} manifolds},
  journal = {SciPost Physics},
  volume  = {15},
  year    = {2023},
  pages   = {144},
  doi     = {10.21468/SciPostPhys.15.4.144},
  eprint  = {2111.02440},
  archivePrefix = {arXiv}
}

@article{Deligne1971,
  author  = {Deligne, Pierre},
  title   = {Th\'eorie de {H}odge. {II}},
  journal = {Publications Math\'ematiques de l'IH\'ES},
  volume  = {40},
  year    = {1971},
  pages   = {5--57},
  doi     = {10.1007/BF02684692}
}

@article{DeligneIllusie1987,
  author  = {Deligne, Pierre and Illusie, Luc},
  title   = {Rel\`evements modulo $p^2$ et d\'ecomposition du complexe de de {R}ham},
  journal = {Inventiones Mathematicae},
  volume  = {89},
  number  = {2},
  year    = {1987},
  pages   = {247--270},
  doi     = {10.1007/BF01389078}
}

@misc{Dummigan2026,
  author  = {Dummigan, Neil},
  title   = {Modularity of a certain ``rank-2 attractor'' {Calabi--Yau} threefold},
  year    = {2026},
  eprint  = {2602.20188},
  archivePrefix = {arXiv},
  primaryClass = {math.NT}
}

@article{FerraraKallosh1996,
  author  = {Ferrara, Sergio and Kallosh, Renata},
  title   = {Supersymmetry and attractors},
  journal = {Physical Review D},
  volume  = {54},
  number  = {2},
  year    = {1996},
  pages   = {1514--1524},
  doi     = {10.1103/PhysRevD.54.1514},
  eprint  = {hep-th/9602136},
  archivePrefix = {arXiv}
}

@article{Grabowski2025,
  author  = {Grabowski, Przemys{\l}aw},
  title   = {Canonical liftings of {Calabi--Yau} hypersurfaces: {Dwork} hypersurfaces},
  journal = {Manuscripta Mathematica},
  volume  = {176},
  year    = {2025},
  eid     = {27},
  doi     = {10.1007/s00229-025-01625-y},
  eprint  = {2401.07945},
  archivePrefix = {arXiv}
}

@article{DembelePanchishkinVoightZudilin2022,
  author  = {Demb{\'e}l{\'e}, Lassina and Panchishkin, Alexei and Voight, John and Zudilin, Wadim},
  title   = {Special hypergeometric motives and their {$L$}-functions: {Asai} recognition},
  journal = {Experimental Mathematics},
  volume  = {31},
  number  = {4},
  year    = {2022},
  pages   = {1278--1290},
  doi     = {10.1080/10586458.2020.1737990},
  eprint  = {1906.07384},
  archivePrefix = {arXiv}
}

@article{DoranKellySalernoSperberVoightWhitcher2018,
  author  = {Doran, Charles F. and Kelly, Tyler L. and Salerno, Adriana and Sperber, Steven and Voight, John and Whitcher, Ursula},
  title   = {Zeta functions of alternate mirror {Calabi--Yau} families},
  journal = {Israel Journal of Mathematics},
  volume  = {228},
  number  = {2},
  year    = {2018},
  pages   = {665--705},
  doi     = {10.1007/s11856-018-1783-0},
  eprint  = {1612.09249},
  archivePrefix = {arXiv}
}

@article{Katz1979,
  author  = {Katz, Nicholas M.},
  title   = {Slope filtration of {$F$}-crystals},
  journal = {Ast\'erisque},
  volume  = {63},
  year    = {1979},
  pages   = {113--163}
}

@article{Mazur1973,
  author  = {Mazur, Barry},
  title   = {Frobenius and the {H}odge filtration: estimates},
  journal = {Annals of Mathematics},
  volume  = {98},
  number  = {1},
  year    = {1973},
  pages   = {58--95},
  doi     = {10.2307/1970906}
}

@article{Ogus1978,
  author  = {Ogus, Arthur},
  title   = {Griffiths transversality in crystalline cohomology},
  journal = {Annals of Mathematics},
  volume  = {108},
  number  = {3},
  year    = {1978},
  pages   = {395--419},
  doi     = {10.2307/1971182}
}

@article{BaldiKlinglerUllmo,
  author  = {Baldi, Gregorio and Klingler, Bruno and Ullmo, Emmanuel},
  title   = {On the Distribution of the {Hodge} Locus},
  journal = {Inventiones Mathematicae},
  volume  = {235},
  number  = {2},
  pages   = {441--487},
  year    = {2024},
  doi     = {10.1007/s00222-023-01226-0},
  eprint  = {2107.08838},
  archivePrefix = {arXiv},
  primaryClass  = {math.AG}
}

@article{BatyrevCox1994,
  author  = {Batyrev, Victor V. and Cox, David A.},
  title   = {On the {H}odge structure of projective hypersurfaces in toric varieties},
  journal = {Duke Mathematical Journal},
  volume  = {75},
  number  = {2},
  year    = {1994},
  pages   = {293--338},
  doi     = {10.1215/S0012-7094-94-07509-1},
  eprint  = {alg-geom/9306011},
  archivePrefix = {arXiv}
}

@incollection{Batyrev1999,
  author    = {Batyrev, Victor V.},
  title     = {Birational {Calabi--Yau} $n$-folds have equal {B}etti numbers},
  booktitle = {New Trends in Algebraic Geometry},
  series    = {London Mathematical Society Lecture Note Series},
  volume    = {264},
  publisher = {Cambridge University Press},
  year      = {1999},
  pages     = {1--12},
  doi       = {10.1017/CBO9780511721540.002},
  eprint    = {alg-geom/9710020},
  archivePrefix = {arXiv}
}

@article{Chiarellotto1998Weights,
  author  = {Chiarellotto, Bruno},
  title   = {Weights in rigid cohomology applications to unipotent {$F$}-isocrystals},
  journal = {Annales scientifiques de l'\'Ecole Normale Sup\'erieure},
  series  = {4},
  volume  = {31},
  number  = {5},
  year    = {1998},
  pages   = {683--715},
  doi     = {10.1016/S0012-9593(98)80004-9}
}

@article{GrosseKlonne2004,
  author  = {Gro{\ss}e-Kl\"onne, Elmar},
  title   = {De {R}ham cohomology of rigid spaces},
  journal = {Mathematische Zeitschrift},
  volume  = {247},
  number  = {2},
  year    = {2004},
  pages   = {223--240},
  doi     = {10.1007/s00209-003-0544-9},
  eprint  = {1408.3330},
  archivePrefix = {arXiv}
}

@article{KatzMessing1974,
  author  = {Katz, Nicholas M. and Messing, William},
  title   = {Some consequences of the {R}iemann hypothesis for varieties over finite fields},
  journal = {Inventiones Mathematicae},
  volume  = {23},
  year    = {1974},
  pages   = {73--77},
  doi     = {10.1007/BF01405203}
}

@article{Yang2021,
  author  = {Yang, Wenzhe},
  title   = {Rank-2 attractors and {Deligne}'s conjecture},
  journal = {Journal of High Energy Physics},
  volume  = {2021},
  number  = {3},
  year    = {2021},
  pages   = {150},
  doi     = {10.1007/JHEP03(2021)150},
  eprint  = {2001.07211},
  archivePrefix = {arXiv},
  primaryClass = {math.AG}
}

@incollection{Moore2007,
  author       = {Moore, Gregory W.},
  title        = {Les Houches Lectures on Strings and Arithmetic},
  booktitle    = {Frontiers in Number Theory, Physics, and Geometry II},
  editor       = {Cartier, Pierre and Julia, Bernard and Moussa, Pierre and Vanhove, Pierre},
  pages        = {265--347},
  publisher    = {Springer},
  address      = {Berlin},
  year         = {2007},
  eprint       = {hep-th/0401049},
  archivePrefix= {arXiv}
}

@article{DworkOgus1986,
  author  = {Dwork, Bernard and Ogus, Arthur},
  title   = {Canonical liftings of {J}acobians},
  journal = {Compositio Mathematica},
  volume  = {58},
  number  = {1},
  pages   = {111--131},
  year    = {1986}
}

@article{Tsuji1999,
  author  = {Tsuji, Takeshi},
  title   = {$p$-adic {\'e}tale cohomology and crystalline cohomology in the semi-stable reduction case},
  journal = {Inventiones Mathematicae},
  volume  = {137},
  number  = {2},
  pages   = {233--411},
  year    = {1999}
}

@article{ColmezFontaine2000,
  author  = {Colmez, Pierre and Fontaine, Jean-Marc},
  title   = {Construction des repr\'esentations $p$-adiques semi-stables},
  journal = {Inventiones Mathematicae},
  volume  = {140},
  number  = {1},
  pages   = {1--43},
  year    = {2000}
}

@article{Grothendieck1966,
  author  = {Grothendieck, Alexander},
  title   = {On the de {R}ham cohomology of algebraic varieties},
  journal = {Publications Math\'ematiques de l'IH\'ES},
  volume  = {29},
  pages   = {95--103},
  year    = {1966}
}

@article{Batyrev1994Dual,
  author  = {Batyrev, Victor V.},
  title   = {Dual polyhedra and mirror symmetry for {C}alabi--{Y}au hypersurfaces in toric varieties},
  journal = {Journal of Algebraic Geometry},
  volume  = {3},
  number  = {3},
  pages   = {493--535},
  year    = {1994}
}

@incollection{Scholl1994,
  author    = {Scholl, Anthony J.},
  title     = {Classical motives},
  booktitle = {Motives (Seattle, WA, 1991)},
  series    = {Proceedings of Symposia in Pure Mathematics},
  volume    = {55, Part 1},
  pages     = {163--187},
  publisher = {American Mathematical Society},
  address   = {Providence, RI},
  year      = {1994}
}

@book{Fulton1998,
  author    = {Fulton, William},
  title     = {Intersection Theory},
  edition   = {Second},
  series    = {Ergebnisse der Mathematik und ihrer Grenzgebiete. 3. Folge},
  volume    = {2},
  publisher = {Springer-Verlag},
  address   = {Berlin},
  year      = {1998}
}

@article{LongoVigni2021,
  author  = {Longo, Matteo and Vigni, Stefano},
  title   = {On {B}loch--{K}ato Selmer groups and {I}wasawa theory of $p$-adic {G}alois representations},
  journal = {New York Journal of Mathematics},
  volume  = {27},
  pages   = {437--467},
  year    = {2021},
  eprint  = {2010.10251},
  archivePrefix = {arXiv},
  primaryClass  = {math.NT}
}

@unpublished{BeukersSupercongruencesModularForms,
  author        = {Beukers, Frits},
  title         = {Supercongruences using modular forms},
  year          = {2024},
  note          = {To appear in Algebra \& Number Theory},
  eprint        = {2403.03301},
  archivePrefix = {arXiv},
  primaryClass  = {math.NT}
}

@unpublished{LamAttractorConjecture,
  author        = {Lam, Yeuk Hay Joshua},
  title         = {The Attractor Conjecture for {Calabi--Yau} Variations
                   of {Hodge} Structures},
  year          = {2020},
  eprint        = {2010.02063},
  archivePrefix = {arXiv},
  primaryClass  = {math.NT}
}

@article{LamTripathyAttractors,
  author  = {Lam, Yeuk Hay Joshua and Tripathy, Arnav},
  title   = {Attractors Are Not Algebraic},
  journal = {Compositio Mathematica},
  volume  = {160},
  number  = {5},
  year    = {2024},
  pages   = {1073--1100},
  doi     = {10.1112/S0010437X24007036}
}

@unpublished{Moore1998,
  author        = {Moore, Gregory W.},
  title         = {Arithmetic and Attractors},
  year          = {1998},
  eprint        = {hep-th/9807087},
  archivePrefix = {arXiv},
  primaryClass  = {hep-th}
}

\end{document}